\documentclass[12pt]{amsart}
\usepackage{hyperref}
\usepackage{amssymb}
\usepackage{epsfig}

\usepackage[margin=1.1in]{geometry}

\newtheorem*{statement}{Theorem}
\newtheorem{theorem}{Theorem}[section]
\newtheorem{thm}[theorem]{Theorem}
\newtheorem{lemma}[theorem]{Lemma}
\newtheorem{lem}[theorem]{Lemma}
\newtheorem{proposition}[theorem]{Proposition}
\newtheorem{prop}[theorem]{Proposition}
\newtheorem{corollary}[theorem]{Corollary}
\newtheorem{cor}[theorem]{Corollary}
\newtheorem{conjecture}[theorem]{Conjecture}

\theoremstyle{definition}

\newtheorem{problem}[theorem]{Problem}

\newtheorem{question}[theorem]{Question}
\theoremstyle{remark}
\newtheorem{example}{Example}[section]
\newtheorem{remark}[example]{Remark}
\newcommand\remind[1]{{\tt**** #1 }}

\newcommand\ip[1]{\langle #1 \rangle}
\def\L{{\mathbb L}}
\def\a{{\mathbf a}}
\def\b{{\mathbf b}}
\def\c{{\mathbf c}}
\def\j{{\mathbf j}}
\def\i{{\mathbf i}}
\def\k{{\mathbf k}}
\def\tS{\tilde{S}}

\def\R{\mathbb R}
\def\C{\mathfrak C}

\def\Z{{\mathbb Z}}

\def\G{GL_n(\R((t)))}
\def\g{GL_n(\R[t,t^{-1}])}

\def\id{{\rm id}}
\def\wt{{\rm wt}}

\def\Imm{\mathrm{Imm}}

\def\fin{{\rm fin}}
\def\re{{\rm re}}
\def\Inv{{\rm Inv}}
\def\im{{\rm im}}
\def\pol{{\rm pol}}
\def\nil{{\rm nil}}

\def\bW{\tilde{\mathcal W}}
\def\aW{\tilde{W}}

\begin{document}
\author{Thomas Lam and Pavlo Pylyavskyy}

 \thanks{T.L. was partially supported by NSF grants DMS-0600677, DMS-0652641, and DMS-0968696, and a Sloan Fellowship. P.P. was partially supported by NSF grant DMS-0757165.}
 \email{tfylam@umich.edu}
 \email{pavlo@umich.edu}

\title{Total positivity for loop groups II: \\ Chevalley generators}

\begin{abstract}
This is the second in a series of papers developing a theory of
total positivity for loop groups. In this paper, we study infinite
products of Chevalley generators.  We show that
the combinatorics of infinite reduced words underlies the theory, and develop
the formalism of infinite sequences of braid moves, called a braid limit.  We relate this
to a partial order, called the limit weak order, on infinite reduced words.

The limit semigroup  generated by Chevalley generators has a transfinite structure.  We prove a form of unique factorization for its elements, in effect reducing their study to infinite products which have the order structure of ${\mathbb N}$.  For the latter infinite products, we show that one always has
a factorization which matches an infinite Coxeter element.  

One of the technical tools we employ is a totally positive exchange
lemma which appears to be of independent interest.  This result states
that the exchange lemma (in the context of Coxeter groups) is
compatible with total positivity in the form of certain
inequalities.
\end{abstract}
\maketitle

\tableofcontents

\section{Introduction}
This is the second in a series of papers where we develop a theory
of total positivity for the formal loop group $\G$ and polynomial
loop group $\g$.  We assume the reader has some familiarity with the
first paper \cite{LP}, and refer the reader to the Introduction there for the original motivation.

Let us briefly recall the main definitions from \cite{LP}. Suppose $A(t)$ is a matrix with entries which are real polynomials,
or real power series.  We associate to $A(t)$ a real infinite periodic matrix
$X = (x_{i,j})_{i,j = -
\infty}^{\infty}$ satisfying $x_{i+n,j+n}=x_{i,j}$  and $$
a_{ij}(t) = \sum_{k=-\infty}^{\infty} x_{i,j+kn} t^{k}.
$$ We call $X$ the {\it {unfolding}} of $A(t)$. We declare that $A(t)$ is {\it {totally nonnegative}} if and only if $X$ is
totally nonnegative, that is, all minors of $X$ are nonnegative.

Let $\G$ denote the {\it formal loop group}, consisting of $n \times
n$ matrices $A(t) = (a_{ij}(t))_{i,j = 1}^n$ whose entries are
formal Laurent series such that $\det(A(t)) \in \R((t))$ is a
non-zero formal Laurent series. We let $\g \subset \G$ denote the
polynomial loop group, consisting of $n \times n$ matrices with
Laurent polynomial coefficients, such that the determinant is a
non-zero real number. We write $\G_{\geq 0}$ for the set of totally
nonnegative elements of $\G$. Similarly, we define $\g_{\geq 0}$.
Let $U \subset \G$ denote the subgroup of upper-triangular unipotent
matrices, and let $U_{\geq 0}$ (resp.~$U_{\geq 0}^\pol$) denote the
semigroup of upper-triangular unipotent totally nonnegative matrices
in $\G$ (resp.~$\g$).  In \cite[Theorem 4.2]{LP} we explained how
the analysis of $\G_{\geq 0}$ and $\g_{\geq 0}$ can to a large
extent be reduced to analysis of $U_{\geq 0}$ and $U_{\geq 0}^\pol$.
Proceeding with the latter, in \cite{LP} we showed

\begin{statement}
Let $X \in U_{\geq 0}$.  Then $X$ has a factorization of the form
$$
X = ZAVBW
$$
where $Z$ (resp.~$W$) is a (possibly infinite) product of
non-degenerate \emph{curls} (resp.~\emph{whirls}), $A$ and $B$ are
(possibly infinite) products of Chevalley generators, and $V$ is a
\emph{regular} matrix.
\end{statement}

The regular factor $V$ will be studied in \cite{LPIII}.

In \cite{LP}, we studied the factors $Z$ and $W$ in detail.  In
particular, they are uniquely determined by $X$, and furthermore
they are infinite products of the forms $\prod_{i=1}^\infty N^{(i)}$
and $\prod_{i=-\infty}^{-1} M^{(i)}$, where the $N^{(i)}$ and the
$M^{(i)}$ are distinguished matrices called {\it {curls}} and {\it
{whirls}}. A { {whirl}} is a matrix given in infinite periodic form
by $M = (m_{i,j})_{i,j=-\infty}^{\infty} = M(a_1, \ldots, a_n)$ with
$m_{i,i}=1$, $m_{i,i+1}=a_{i}$ and the rest of the entries equal to
zero, where the indexing of the parameters is taken modulo $n$.
Define $X^c \in U$ to be the matrix obtained by applying to $X \in
U$ the transformation $x_{i,j} \mapsto (-1)^{|i-j|}x_{i,j}$, and
define $X^{-c} := (X^c)^{-1}$. It was shown in \cite[Lemma 4.5]{LP}
that if $X \in U_{\geq 0}$ then $X^{-c} \in U_{\geq 0}$. A {{curl}}
is a matrix $N$ of the form $N(a_1, \ldots, a_n) := M(a_1, \ldots,
a_n)^{-c}$.

In this paper we shall show (Theorem \ref{thm:uniqueness}) that the
factors $A$ and $B$ are unique, though in general they are not infinite products with the order structure of ${\mathbb N}$, but instead have a transfinite structure.

Let $e_{i}(a)$ denote {\it {the affine Chevalley generators}}, which differ from the
identity matrix by the entry $a$ in the $i$-th row and
$(i+1)$-st column (in infinite periodic matrix notation). The two kinds of Chevalley generators for $n=2$ are shown below.
$$
e_1(a) = \left(
\begin{array}{c|cc|cc|c}
\ddots & \vdots & \vdots   & \vdots&\vdots \\
\hline
\dots& 1 & a&0& 0&\dots \\
\dots&0&1&0&0&\dots \\
\hline
\dots&0&0& 1 & a & \dots\\
\dots&0&0&0 & 1 & \dots \\
\hline & \vdots &  \vdots  & \vdots & \vdots &
\ddots\end{array} \right)
, \;
e_2(b) = \left(
\begin{array}{c|cc|cc|c}
\ddots & \vdots & \vdots   & \vdots&\vdots \\
\hline
\dots& 1 & 0&0& 0&\dots \\
\dots&0&1&b&0&\dots \\
\hline
\dots&0&0& 1 & 0 & \dots\\
\dots&0&0&0 & 1 & \dots \\
\hline & \vdots &  \vdots  & \vdots & \vdots &
\ddots\end{array} \right)
$$
Our
approach is based upon the study of the map
\begin{equation}\label{E:ei}
e_\i: (a_1,a_2,\cdots) \longmapsto e_{i_1}(a_1) e_{i_2}(a_2) \cdots
\end{equation}
which converges for $a_k > 0$ satisfying $\sum_k a_k < \infty$.

\subsection{Cell decomposition in the finite case}

Let $U_{\geq 0}^\fin \subset GL_n(\R)$ denote the semigroup of
totally nonnegative upper-triangular unipotent $n \times n$
matrices.  The following result of Lusztig \cite{Lus} gives a cell
decomposition of $U_{\geq 0}^\fin$.

\begin{theorem} \label{thm:Lus}\
\begin{enumerate} \item We have
$$U_{\geq 0}^\fin = \bigsqcup_{w \in S_n} U_{\geq 0}^w$$
where
$$
U_{\geq 0}^w = \{e_{i_1}(a_1) e_{i_2}(a_2) \cdots e_{i_\ell}(a_\ell)
\mid a_k > 0\}
$$
and $w = s_{i_1} s_{i_2} \cdots s_{i_\ell}$ is a reduced expression.
\item
The set $U_{\geq 0}^w$ does not depend on the choice of reduced
expression, and the map $e_\i: \R^{\ell}_{>0} \mapsto U_{\geq 0}^w$
is a bijection.
\end{enumerate}
\end{theorem}

The cells $U_{\geq 0}^w$ can also be obtained by intersecting
$U_{\geq 0}^\fin$ with the Bruhat decomposition $GL_n(\R)= \sqcup_{w
\in S_n} B_- w B_-$.  In Theorem \ref{thm:Ew} we establish the
analogue of Theorem \ref{thm:Lus} for the totally nonnegative part
$U_{\geq 0}^{\pol}$ of the polynomial loop group, with the affine
symmetric group replacing the symmetric group.

\subsection{Infinite products of Chevalley generators}
Let $\aW$ denote the affine symmetric group $\tS_n$, with simple
generators $\{s_i \mid i \in \Z/n\Z\}$.  An infinite word $\i =i_1
i_2 \cdots$ in the alphabet $\Z/n\Z$ is {\it reduced} if each
initial subword is a reduced word for some element of $\aW$.

For infinite reduced words the map $e_\i$ of \eqref{E:ei}
does not satisfy many of the good properties which exist for finite
reduced words. Let $E_\i$ denote the image of $e_\i$ and let $\Omega = \cup_\i E_\i$. Then in contrast to the finite case,
\begin{enumerate}
\item
The map $e_\i$ is not injective in general.
\item
We can find infinite reduced words $\i, \j$ such that $E_\i
\subsetneq E_\j$.
\item
We can find infinite reduced words $\i, \j$ such that $E_\i \cap
E_\j \neq \emptyset$ but neither $E_\i$ nor $E_\j$ is contained in
the other.
\item
$\Omega$ is not a semigroup.
\end{enumerate}

We shall tackle these difficulties by:
\begin{enumerate}
\item
Giving a conjectural classification (Conjecture
\ref{conj:injective}) of infinite reduced words $\i$ such that
$e_\i$ is injective, and proving this in an important case when the
domain is restricted (Proposition \ref{prop:ainj}).
\item
Giving a criterion (Theorem \ref{thm:TPlimit}) for $E_\i \subset
E_\j$, using the notion of {\it braid limits} and the {\it limit
weak order}.
\item
Showing that the union of $E_\i$ over the finite set of infinite
Coxeter elements covers $\Omega$ (Corollary \ref{cor:Coxeter}).
\item
For each $X \in \Omega$, constructing a distinguished factorization
$X = e_\i(\a)$ (Theorem \ref{thm:aswomega}).
\item
Showing that the limit semigroup generated by $\Omega$ satisfies
some form of unique factorization (Theorem \ref{thm:uniqueness}).
\end{enumerate}

\subsection{Limit weak order}
The inversion set $\Inv(\i)$ of an infinite reduced word is an
infinite set of positive real roots of $\aW$.  These inversion sets
were classified by Cellini and Papi \cite{CP} (who called them
compatible sets) and by Ito \cite{Ito} (who called them biconvex
sets).

Inclusion of inversion sets gives rise to a partial order on
(equivalence classes of) infinite reduced words, which we call the
{\it limit weak order}. We show that the limit weak order
$(\bW,\leq)$ can also be obtained by performing (possibly infinite)
sequences of braid moves on infinite reduced words.  We encourage
the reader to look at Example \ref{ex:012} for an example of such a
sequence, which we call a {\it braid limit}, denoted $\i \to \j$.

The limit weak order $(\bW,\leq)$ is an infinite poset which (unlike
usual weak order) contains intervals which themselves are infinite.
To analyze it, we divide $\bW$ into {\it blocks}.  Each block is
isomorphic to a product of usual weak orders of (smaller) affine
symmetric groups (Theorem \ref{thm:blockorder}).  The partial order
between the blocks themselves is isomorphic to the face poset of the
braid arrangement (Theorem \ref{thm:blocks}).  We explicitly express
(Proposition \ref{prop:explicit}) the (unique) minimal element of
each block as an infinite reduced word. In particular, the minimal
elements of $(\bW,\leq)$ are exactly the infinite products
$c^\infty$ of Coxeter elements $c$ of $\aW$ (Theorem \ref{thm:CoxTFAE}), which
are in bijection with the edges of the braid arrangement.

Many of the results concerning limit weak order generalize to other
infinite Coxeter groups, but some (for example, Theorem
\ref{thm:CoxTFAE}) do not.

\subsection{Braid limits and total nonnegativity}

When we perform infinitely many braid transformations to a product $e_{i_1}(a_1) e_{i_2}(a_2) \cdots$, and take a limit, a priori it is not clear that the resulting product is equal to the original one. In fact, this is false in Kac-Moody generality.  The following central result (Theorem \ref{thm:TPlimit}) shows that this is true in affine type $A$, thus laying a foundation for our investigations of $\Omega$.
\begin{statement}[TNN braid limit theorem]
If $\i \to \j$ is a braid limit between two infinite reduced words,
then $E_\i \subset E_\j$.  In other words, every totally nonnegative
matrix $X$ which can be expressed as $X = e_\i(\a)$ can be expressed
as $X = e_\j(\a')$.
\end{statement}

As stated above, the infinite Coxeter elements $c^\infty$ are the
minimal elements of limit weak order.  It follows that we have the
finite (but not disjoint) union $\Omega = \cup_c E_{c^\infty}$.  We
use the TNN braid limit theorem to show that $e_\i$ can only be
injective when $\i$ is minimal in its block of limit weak order
(Proposition \ref{prop:notinjective}). We conjecture that the
converse also holds (Conjecture \ref{conj:injective}).  Finally we
use the $\epsilon$-sequence of \cite{LP} to establish injectivity in
some cases (Proposition \ref{prop:ainj}).

\subsection{ASW factorizations}
To tackle the lack of injectivity of $e_\i$, we give two different
approaches.

In the first approach, we study the {\it {ASW factorization}} of
\cite{LP}, applied to matrices $X \in \Omega$. Let us recall the
definition here. For $X \in U_{\geq 0}$ let $\epsilon_i(X) = \lim_{j
\to \infty} \frac{x_{i,j}}{x_{i+1,j}}$. It was shown in \cite[Lemma
5.3]{LP} that there exists a factorization $X = N(\epsilon_1,
\ldots, \epsilon_n) Y$, where $N(\epsilon_1, \ldots, \epsilon_n)$ is
a curl with parameters $\epsilon_i = \epsilon_i(X)$ and $Y \in
U_{\geq 0}$ is some totally nonnegative matrix. We refer to the
extraction of the curl factor $N(\epsilon_1, \ldots, \epsilon_n)$
from $X$ as the ASW factorization of $X$. We also use the same
terminology for the {\it {repeated}} extraction of such a factor.

The main difficulty here can be stated rather simply: suppose $X \in
\Omega$ and $X = e_i(a) X'$, where $a > 0$ and $X' \in U_{\geq 0}$,
then is $X'$ necessarily in $\Omega$?  We answer this affirmatively.
As a consequence, we obtain a distinguished factorization of each $X
\in \Omega$, decomposing $\Omega$ into a {\it disjoint} union of
pieces which we call {\it ASW-cells} (Theorem \ref{thm:aswomega} and
Proposition \ref{prop:aswcell}).  The ASW-cells are labeled by
certain pairs $(w,v) \in \aW \times \aW$ of affine permutations,
which we call compatible.  Our study of ASW factorization also leads
to our first proof of the TNN braid limit theorem, and in addition
proves the following theorem (Theorem \ref{thmcor:uniqueness}).
\begin{statement}[Unique factorization in $\L_r$]
Denote by $\L_r$ the right limit-semigroup generated by Chevalley
generators (see Section \ref{sec:Omegaunique} for precise
definitions).  Each element of $\L_r$ has a unique factorization
into factors which lie in $\Omega$, with possibly one factor which
is a finite product of Chevalley generators.
\end{statement}

\subsection{Greedy factorizations}

In a second approach to the lack of injectivity of $e_\i$, we study
{\it greedy} factorizations.  These are factorizations $X =
e_\i(\a)$ where for a fixed $\i$, $a_1$ is maximal and having
factored out $e_{i_1}(a_1)$, the second parameter $a_2$ is also
maximal, and so on.  Clearly, if $\i$ is fixed, there is at most one
greedy factorization of $X$, so ``injectivity'' is automatic.  Our
main result (Theorem \ref{thm:grp}) concerning greedy factorizations
is that they are preserved under braid moves (or even braid limits).
We also give formulae in some special cases for the parameters
$a_1,a_2,\ldots$ in a greedy factorization in terms of limits of
ratios of minors of $X$.  We have already studied minor limit ratios
in \cite{LP}.  The minor limit ratios used for greedy factorizations
are distinguished by the fact that a single limit involves ratios of
minors of different sizes.

\subsection{Totally positive exchange lemma}
One of our proofs of the TNN braid limit theorem is based upon the
Totally Positive Exchange Lemma (Theorem \ref{thm:TPex}).  This is a
result about {\it finite} products $e_{i_1}(a_1) \cdots
e_{i_k}(a_k)$ of Chevalley generators, which seems to be of
independent interest.  Recall the usual exchange condition for Coxeter groups.
\begin{statement}[Exchange Lemma]
If $\bar w = s_{i_1} s_{i_2} \dotsc s_{i_k}$ is a reduced expression
for an element $w$ of a Coxeter group, and $s_r \bar w$ is not
reduced, then $s_r w = s_{i_1} \dotsc \hat s_{i_l} \dotsc s_{i_k}$
for a unique index $l$, where $\hat s_{i_l}$ denotes omission of a
generator.
\end{statement}
The Totally Positive Exchange Lemma states that in (affine) type $A$
when an exchange is performed on the level of Chevalley generators,
certain inequalities between the parameters $a_i$ before and after
the exchange hold.
\begin{statement}[Totally Positive Exchange Lemma]
Suppose
$$
X = e_r(a) e_{i_1}(a_1) \cdots e_{i_\ell}(a_\ell) = e_{i_1}(a'_1)
\cdots e_{i_\ell}(a'_\ell) e_j(a')
$$
are reduced products of Chevalley generators such that all
parameters are positive (so that $j$ has been exchanged for $r$).
For each $m \leq \ell$ and each $x \in \Z/n\Z$ define $S = \{s \leq
m \mid i_s = x\}$. Then
$$
\sum_{s \in S} a'_{i_s} \leq \begin{cases} \sum_{s \in S} a_{i_s} & \mbox{if $x \neq r$,} \\
a + \sum_{s \in S} a_{i_s} & \mbox{if $x= r$.}
\end{cases}
$$
\end{statement}
We give two proofs of the TP Exchange Lemma.  The first proof relies
on explicit formulae for the parameters $a_i$, given by the
Berenstein-Zelevinsky Chamber Ansatz \cite{BZ}.  The second proof is
less direct, and relies on reducing the result to a statement about
calculating certain joins in weak order.  We shall return to the TP
Exchange Lemma in a more general setting in future work \cite{LPKM}.

\subsection{Open problems and conjectures}
Section \ref{sec:problems} contains a list of questions and conjectures together with some partial results, most of which are concerned with the subsets $E_\i \subset \Omega$ and the maps $e_\i$.  A particularly powerful conjecture is the Principal ideal conjecture (Conjecture \ref{con:pic}) which states that the set $\{\i \mid X \in E_\i\}$ is a principal ideal in limit weak order.


\section{Notations and definitions}
\subsection{Total nonnegativity}

To every $\bar X(t) \in \G$ we associate an infinite periodic matrix
$X$, which are related via
$$
\bar x_{ij}(t) = \sum_{k=-\infty}^{\infty} x_{i,j+kn} t^{k}.
$$
As in \cite{LP}, we will abuse notation by rarely distinguishing
between a matrix $\bar X(t) \in \G$, and its unfolding $X$ which is
an infinite periodic matrix.  The matrix $\bar X(t)$ is called the
folded matrix.  For finite sets $I, J \subset \Z$ of the same
cardinality we let $\Delta_{I,J}(X)$ denote the corresponding minor,
always in the unfolded matrix.

Following \cite{LP}, we let $U \subset \G$ denote the group of
unipotent upper-triangular matrices, $U_{\geq 0}$ (resp.~$U_{>0}$)
denote the totally nonnegative (resp. totally positive) part of $U$.
Both $U_{\geq 0}$ and $U_{>0}$ are semigroups.  We let $U^\pol
\subset U$ and $U^\pol_{\geq 0} \subset U_{\geq 0}$ denote the
corresponding matrices which belong to the polynomial loop group
(both $X$ and $X^{-1}$ are required to have polynomial entries).

Recall that in \cite{LP} we have defined an (anti-)involution $X
\mapsto X^{-c}$ of $U_{\geq 0}$.  We say that $X \in U$ is {\it
entire} if all the (folded) entries are entire functions.  We say
that $X \in U$ is {\it doubly-entire} if both $X$ and $X^{-c}$ is
entire. We say that $X$ is {\it finitely-supported} if all the
(folded) entries are polynomials.

Let $I =\{i_1 < i_2 < \cdots < i_k\}$ and $J = \{j_1 < j_2 < \cdots
<j_k\}$ be subsets of $\Z$.  Write $I \leq J$ if $i_r \leq j_r$ for
$r \in [1,k]$.  The minors $\Delta_{I,J}(X)$ are the
upper-triangular minors: all other minors vanish on $U_{\geq 0}$. We
say that $X \in U_{\geq 0}$ is {\it totally positive} (see
\cite[Corollary 5.9]{LP}) if for all $I \leq J$, the minor $\Delta_{I,J}(X)$
is strictly positive.

\subsection{Affine symmetric group}
Let $\aW$ denote the affine symmetric group, with simple generators
$\{s_i \mid i \in \Z/n\Z\}$, satisfying the relations $s_i^2 = 1$,
$s_i s_{i+1} s_i = s_{i+1} s_i s_{i+1}$ and $s_i s_j = s_j s_i$ for
$|i - j| > 1$.  The indices are always to be taken modulo $n$.  The
length $\ell(w)$ for $w \in \aW$ is the length of the shortest
expression $w = s_{i_1} s_{i_2} \cdots s_{i_\ell}$ of $w$ in terms
of simple generators.  We call such an expression a reduced
expression, and $i_1 i_2 \cdots i_\ell$ a reduced word for $w$.  The
(right) weak order on $\aW$ is defined by $v < w$ if $w = vu$ for
$u$ satisfying $\ell(w) = \ell(v) + \ell(u)$.  Right weak order is
graded by the length function $\ell(w)$, and the covering relations
are of the form $w < ws_i$.  A left descent of $w \in \aW$ is an
index (or simple root, or simple generator) $i$ (or $\alpha_i$ or
$s_i$) such that $s_i w < w$.  Similarly one defines ascents. Note
that $i$ is a left descent of $w$ if and only if there is a reduced
word for $w$ beginning with $i$.

The affine symmetric group $\aW$ can be identified with the group of
bijections $w: \Z \to \Z$ satisfying $w(i + n) = w(i) + n$ and
$\sum_{i=1}^n w(i) = n(n+1)/2$.  Group multiplication is given by
function composition.  Left multiplication by $s_i$ swaps the values
$i$ and $i+1$, while right multiplication swaps positions.  The
window notation for $w \in \aW$ is the sequence $[w(1)w(2)\cdots
w(n)]$, which completely determines $w$.  The symmetric group $W =
S_n$ embeds in $\aW$ in the obvious manner.

Let $Q^\vee = \oplus_{i=1}^{n-1} \Z \cdot \alpha_i^\vee$ denote the
(finite) coroot lattice, which we identify with
$$\{(\lambda_1,\lambda_2,\ldots,\lambda_n) \in \Z^n \mid \sum_{i=1}^n
\lambda_i = 0\}.$$  Given $\lambda \in Q^\vee$, one has a
translation element $t_\lambda \in \aW$, given by
$$
t_\lambda(i) = i+n\,\lambda_i,
$$
so that $t_{(0,0,\ldots,0)}$ is the identity affine permutation. The
affine symmetric group can be presented as a semidirect product
$W\ltimes Q^\vee$, where $v t_\lambda v^{-1} = t_{v \cdot \lambda}$
for $v \in W$.  We say that $\lambda \in Q^\vee$ is {\it dominant}
if $\lambda_1 \geq \lambda_2 \geq \cdots \geq \lambda_n$.

\begin{lem}\label{lem:tlength}
Let $\lambda \in Q^\vee$ be dominant and $w \in W$ be arbitrary.
Then $$\ell(t_{w \cdot \lambda}) = 2(n\lambda_1 + (n-1)\lambda_2 +
\cdots + \lambda_{n-1}).$$ In particular, $\ell(t_{w \cdot
\lambda})$ does not depend on $w$.
\end{lem}

Let $\Delta_0 = \{\alpha_{i,j} \mid 1 \leq i \neq j \leq n\}$ denote
the root system of $W$, and write $\alpha_i = \alpha_{i,i+1}$ for
the simple roots.  We let $\Delta$ denote the root system of $\aW$,
with simple roots $\{\alpha_i \mid i \in \Z/n\Z\}$ and null root
$\delta = \alpha_0 + \alpha_1 + \cdots + \alpha_{n-1}$. We have
$\Delta = \{n \delta \mid n \in \Z - \{0\}\} \cup \{n \delta +
\alpha \mid n \in \Z \;\text{and}\; \alpha \in \Delta_0\}$.  The
real roots $\{n \delta + \alpha \mid n \in \Z \;\text{and}\; \alpha
\in \Delta\}$ are denoted $\Delta_\re$.

Recall that we have
$$
\Delta_0 = \Delta_0^+ \cup \Delta_0^- =  \{ \alpha_{i,j} = \alpha_i
+ \cdots + \alpha_{j-1} \mid i < j\} \cup \{ \alpha_{i,j} =
-\alpha_{j,i} \mid i > j\}.$$ Thus the root $\alpha_{i,j}$ is
positive if $i<j$, negative if $i>j$, and positive simple if $j =
i+1$. The roots in $\Delta_{\re}$ are $\alpha_{i,j} + k \delta$, $k
\in \mathbb Z$. A real affine root is positive if $i<j$ and $k \geq 0$, or if $i
>j$ and $k \geq 1$.

\section{Polynomial loop group} \label{sec:polyloop}

\subsection{Relations for Chevalley generators}

By \cite[Theorem 2.6]{LP} the semigroup $U^\pol_{\geq 0}$ is
generated by the Chevalley generators $e_i(a)$ with nonnegative
parameters $a \geq 0$. We recall the standard relations for
Chevalley generators \cite{Lus}:
\begin{align}
\label{E:chevrel1} e_i(a)\,e_j(b) &= e_j(b) e_i(a) &\mbox{if $|i-j| \geq 2$} \\
\label{E:chevrel2} e_i(a)\,e_{i+1}(b)\,e_i(c) &=
e_{i+1}(bc/(a+c))\,e_i(a+c) \, e_{i+1}(ab/(a+c)) & \mbox{for each $i
\in \Z/n\Z$}
\end{align}
for nonnegative parameters $a,b,c$. For a reduced word $\i = i_1 i_2
\cdots i_\ell$ of $w \in \aW$, and a collection of parameters $a_k
\in \R$, write $e_\i(\a)$ for $e_{i_1}(a_1) \cdots
e_{i_\ell}(a_\ell)$.  Denote $E_{\i}$ the image of the map $\a
\mapsto e_{\i}(\a)$, as $\a$ ranges over $\R^\ell_{>0}$.  The
following result follows from relations \eqref{E:chevrel1} and
\eqref{E:chevrel2}.

\begin{lemma} \label{lem:Einv}
If $\i$ and $\j$ are two reduced words of $w \in \aW$ then $E_{\i} =
E_{\j}$.
\end{lemma}


Therefore we can introduce the notation $E_w = E_{\i}$ which is
independent of the reduced word $\i$ of $w$.

\subsection{Decomposition of $U_{\geq 0}^\pol$}

\begin{theorem} \label{thm:Ew}
We have a disjoint union
$$U_{\geq 0}^{\pol} = \bigsqcup_{w \in \aW} E_w.$$
\end{theorem}

The fact that the subsets $\{E_w\}_{w \in \aW}$ cover $U_{\geq
0}^{\pol}$ follows from \cite[Theorem 2.6]{LP}.  To prove that the
$E_w$-s are disjoint we will describe a necessary condition for $X
\in U_{\geq 0}^{\pol}$ to belong to $E_w$.


We refer to the matrix entry positions of an infinite ($\Z \times \Z$)
matrix $X$ as {\it {cells}}.  Define a partial order on cells:
$(i,j) \leq (i',j')$ if $i \geq i'$ and $j \leq j'$.  In other words
$c \leq c'$ if $c'$ is to the northeast of $c$.  A cell with
coordinates $(i,w^{-1}(i))$ for some $i$ is called a {\it
$w$-{dot}}.  The collection of $w$-dots is denoted $C_w$.  We say
that a finite set of cells $C$ is {\it $w$-dominated} if for every
cell $(i,j)$ we have
$$
\#\{c \in C \mid c \geq (i,j)\} \leq \#\{c \in C_w \mid c \geq
(i,j)\}.
$$

Let $I = \{i_1 < i_2 < \cdots < i_k\}, J = \{j_1 < j_2 < \cdots
j_k\} \subset \Z$.  Define $C = C(I,J)$ to be the set of cells
$\{(i_1,j_1),(i_2,j_2),\ldots,(i_k,j_k)\}$.

\begin{prop}\label{prop:mv}
Let $w \in \aW$ and suppose $X \in E_w$.  Let $I \leq J$.  Then
$\Delta_{I,J}(X)
> 0$ if and only if $C(I,J)$ is $w$-dominated.
\end{prop}
\begin{proof}
We shall write $\Delta_{C}(X)$ for $\Delta_{I,J}(X)$ when $C =
C(I,J)$.

The proof proceeds by induction on the length $\ell(w)$ of $w$.  The
base case of the identity permutation $w = \id$ is trivial: an
upper-triangular minor is non-zero if and only if all cells in $C(I,J)$ are on the diagonal, if
and only if $C(I,J)$ is $\id$-dominated.

Assume now that $w = s_iv$ where $\ell(w)=\ell(v)+1$, and that we
already know the validity of the statement for elements of $E_v$.
The set $C_w$ differs from $C_v$ by swapping the cells in the
$kn+i$-th and $kn+i+1$-th rows for each $k$.  It follows from
$\ell(w)=\ell(v)+1$ that any $v$-dominated set is also
$w$-dominated.  Multiplication of $Y \in E_v$ by $e_i(a)$ on the
left adds $a$ times row $kn+i+1$ to the row $kn+i$ for each $k \in
\mathbb Z$.  In particular, any positive minor of $Y$ is also a
positive minor of $X$.

Now let $C = C(I,J)$ be $w$-dominated for some $I \leq J$.  We may
assume that $C(I,J)$ does not contain any cells on the diagonal,
since the value of $\Delta_{I,J}(X)$ does not change if the diagonal
cells are removed; in addition $w$-dominance is preserved under
removal of cells.  If $C$ is also $v$-dominated then $\Delta_{C}(X)
\geq \Delta_{C}(Y)
> 0$ by the inductive assumption. Otherwise $C$ is not
$v$-dominated. Since $C_w$ and $C_v$ differ only in the $(kn+i)$-th
and $(kn+i+1)$-th rows, $C$ must contain a cell in one of those
rows. Checking a number of cases, one deduces that for some $k$, $C$
contains a cell in the $(kn+i)$-th row but not in the $(kn+i+1)$-th
row. Let $C'$ be obtained from $C$ by moving all cells in the
$kn+i$-th rows down one row, whenever the row below is not occupied.
It is easy to check that $C'$ is $v$-dominated. Furthermore,
$\Delta_C(X)$ is a positive linear combination of minors of $Y$, one
of which is $\Delta_{C'}(Y)$. Thus $\Delta_C(X)
> 0$ if $C$ is $w$-dominated.

Now suppose $C = C(I,J)$ is not $w$-dominated for some $I \leq J$.
The minor $\Delta_C(X)$ is a linear combination of minors of the
form $\Delta_{C'}(Y)$, where $C'$ is obtained from $C$ by moving
cells in the $kn+i$-th rows for some values of $k$ down one row
(assuming the row below is not occupied).  We claim that all the
minors $\Delta_{C'}(Y)$ vanish.  It is enough to show that $C'$ is
never $v$-dominated, assuming that $C'$ consists only of cells above
the diagonal.  For each $(a,b)$, let $A(a,b) = \#\{c \in C \mid c
\geq (a,b)\}$, $A'(a,b) = \#\{c \in C' \mid c \geq (a,b)\}$,
$A_w(a,b) = \#\{c \in C_w \mid c \geq (a,b)\}$, and $A_v(a,b) =
\#\{c \in C_v \mid c \geq (a,b)\}$.  Suppose $(a,b)$ satisfies
$A(a,b) > A_w(a,b)$.  If $a$ is not of the form $kn+i$ then we have
$A(a,b) = A'(a,b)$ and $A_w(a,b) = A_v(a,b)$ so that $C'$ is not
$v$-dominated.  So assume $a = kn+i$.  We may assume that $C$
contains a cell $c = (a,r)$ in row $a$, and that $b \leq r$.  If
$b \leq w^{-1}(a)$ then $A(a-1,b) > A_w(a-1,b)$,
reducing to the previous case.  If $b > w^{-1}(a)$ then $A_v(a+1,b) = A_w(a,b)$ and $A'(a+1,b) \geq
A(a,b)$, which implies $C'$ is not $v$-dominated.
\end{proof}

\begin{example}
Let $n=3$ and let $w = s_0 s_1$. The window notation for $w$ is $[2,0,4]$ and the $w$-dots are the cells with coordinates $(3k+1, 3k)$, $(3k, 3k+2)$ and $(3k+2, 3k+1)$ for $k \in \Z$. Take $I = (0,1)$ and $J = (1,2)$. Then $I \leq J$, but $C(I,J)$ is not $w$-dominated. Indeed, $C(I,J) = \{(0,1), (1,2)\}$ and for $(i,j)=(1,1)$ we have $\#\{c \in C(I,J) \mid c \geq (1,1)\} = 2$, while $\#\{c \in C_w \mid c \geq
(1,1)\} = 1$. Therefore for $X \in E_w$ we have $\Delta_{I,J}(X) = 0$. On the other hand, if we pick $I = (-2,0,1)$ and $J = (-1,0,2)$ then it is not hard to check that $C(I,J)$ is $w$-dominated and therefore $\Delta_{I,J}(X) > 0$ for $X \in E_w$.
\end{example}

\begin{remark}
Proposition \ref{prop:mv} can be applied in the special case $w \in W$, naturally generalizing the conditions appearing in \cite[Proposition 4.1]{FZ}. Note however that unlike \cite{FZ} we deal only with totally nonnegative matrices and do not aim to provide a minimal set of sufficient conditions.
\end{remark}

\begin{proof}[Proof of Theorem \ref{thm:Ew}]
We claim that the minor vanishing/non-vanishing conditions of
Proposition \ref{prop:mv} are incompatible for two distinct elements
$w, v \in \aW$.  Indeed, assume there exists $X \in E_w \cap E_v$.
The set of numbers $A_w(i,j) = \#\{c \in C_w \mid c \geq (i,j)\}$
for all $i,j$ completely determine $w$.  If $w \neq v$ there is
$(i,j)$ such that $A_w(i,j) \neq A_v(i,j)$ and $A_w(i',j') =
A_v(i',j')$ for all $(i',j') > (i,j)$.  We may assume that $A_w(i,j)
> A_v(i,j)$.  The $w$-dots and $v$-dots strictly to the north or east or northeast of
$(i,j)$ coincide.  Let $I$ (resp.~$J$) be the rows (resp.~columns)
containing $w$-dots to northeast of $(i,j)$, including $(i,j)$
itself.  The fact that $I \leq J$ is easy to see by induction.  It
is clear that $C=C(I,J)$ is not $v$-dominated, and one checks a
number of cases to see that $C$ is $w$-dominated.  We obtain a
contradiction from Proposition \ref{prop:mv} by looking at
$\Delta_C(X)$.
\end{proof}

The following result is crucial for later parts of the paper.

\begin{theorem} \label{thm:fin} \
\begin{enumerate}
\item
For $\i$ a reduced decomposition of $w$, the map $e_\i: \R^\ell_{>0}
\to E_w$ is injective.
\item
If $X \in E_w$ and $X = YZ$ where $Y$ and $Z$ are totally
nonnegative, then $Y \in E_v$ and $Z \in E_u$ for some $v, u \in
\aW$ and $v \leq w$ in weak order.
\end{enumerate}
\end{theorem}

\begin{proof}
We prove (1).  Assume that $\a, \a' \in \R^\ell_{>0}$ are two
distinct sets of parameters such that $e_\i(\a)= X = e_\i(\a')$.
Without loss of generality we assume $a_1 \neq a'_1$, for otherwise
we may remove $e_{i_1}(a_1)$ from $X$ and remove $s_{i_1}$ from $w$,
and apply the argument to the resulting affine permutation.  We also
assume without loss of generality that $a_1
> a'_1$.  Then $e_{i_1}(-a'_1) X$ lies both in $E_w$ and $E_{s_i w}$,
which is a contradiction to Theorem \ref{thm:Ew}.

We prove (2).  Since $X$ is finitely supported, so are $Y$ and $Z$.
By \cite[Lemma 5.1 and Theorem 5.5]{LP} and the calculation $1 =
\det(\bar X(t)) = \det(\bar Y(t)) \det(\bar Z(t))$, it follows that
$Y$ and $Z$ factor into a product of $e_i(a)$-s, and thus for some
$v, u \in \aW$ we have $Y \in E_v$, $Z \in E_u$. Finally, to see why
$v \leq w$ one can think of multiplying $Y$ by a sequence of
Chevalley generators from $Z$ to obtain matrices in $E_v,
E_{v^{(1)}}, E_{v^{(2)}}, \ldots, E_{w}$. It is easy to see that $v
\leq v^{(1)} \leq v^{(2)} \leq \cdots \leq w$.
\end{proof}

If $\i = i_1 i_2 \cdots i_\ell$ and $\j = j_1 j_2 \cdots j_\ell$ are
two reduced words for $w \in \aW$, then applying the relations
\eqref{E:chevrel1} and \eqref{E:chevrel2}, we obtain a map
$$
R_\i^\j: \R_{>0}^\ell \to \R_{>0}^\ell
$$
such that $e_\i(\a) = e_\j(R_\i^\j(\a))$. The following is an
immediate corollary of Theorem \ref{thm:fin}.

\begin{corollary}\label{lem:Rwelldefined}
The map $R_\i^\j$ is well-defined, and does not depend on the order
in which we apply \eqref{E:chevrel1} and \eqref{E:chevrel2}.
\end{corollary}

\section{Infinite reduced words, and braid limits} \label{sec:braidlim}

\subsection{Biconvex sets}\label{sec:biconvex}
Let $I \subset \Delta^+_\re$ be a (possibly infinite) set of positive real affine
roots.  We call $I$ {\it biconvex} if for any $\alpha, \beta \in
\Delta^+$ one has
\begin{enumerate}
\item
if $\alpha, \beta \in I$ and $\alpha + \beta \in \Delta$ then
$\alpha+\beta \in I$,
\item
if $\alpha + \beta \in I$ then either $\alpha \in I$ or $\beta \in
I$.
\end{enumerate}
Note that in (1) one must include the case that $\alpha + \beta$ is
not a real root.  Biconvex sets were studied in \cite{CP, Ito} for
an arbitrary affine Weyl group.  (Cellini and Papi \cite{CP} use the
word ``compatible'' instead.)

Let $I \subset \Delta^+_\re$ be a biconvex set.  It is easy to see
that for each $\alpha \in \Delta_0^+$, the intersection
$$
I \cap\{\ldots, 3\delta-\alpha,
2\delta-\alpha,\delta-\alpha,\alpha,\alpha+\delta,\alpha+2\delta,\ldots\}
$$
is one of the following: (a) empty, (b)
$\{\alpha,\alpha+\delta,\ldots,\alpha+m_\alpha \delta\}$, (c)
$\{\alpha,\alpha+\delta,\ldots\}$, (d)
$\{\delta-\alpha,2\delta-\alpha,\ldots,-m_\alpha\delta-\alpha\}$, or
(e) $\{\delta-\alpha,2\delta-\alpha,\ldots\}$.  In (b), $m_\alpha >
0$ but in (d), $m_\alpha < 0$.  In cases (a), (c), (e), we set
$m_\alpha$ to be $0$, $\infty$, $-\infty$ respectively.  The proof
of the following result is straightforward.

\begin{proposition}\label{prop:table}
A set of positive real roots is biconvex if and only if for any
$\alpha,\beta,\gamma \in \Delta^+_0$ such that $\alpha + \beta =
\gamma$ we have one of the following possibilities for $m_{\alpha}$,
$m_{\beta}$ and $m_{\gamma}$:
\begin{center}
\begin{tabular}{|c|c|c|}
\hline $m_{\alpha}$ & $m_{\beta}$ & $m_{\gamma}$
\\
\hline finite & finite & $m_{\alpha} + m_{\beta}$\\
\hline finite & finite & $m_{\alpha} + m_{\beta}-1$\\
\hline $\pm \infty$ & finite & $\pm \infty$\\
\hline finite & $\pm \infty$ & $\pm \infty$\\
\hline $\pm \infty$ & $\pm \infty$ & $\pm \infty$\\
\hline $\pm \infty$ & $\mp \infty$ & anything\\
\hline
\end{tabular}
\end{center}
\end{proposition}

\subsection{Infinite reduced words, inversion sets}
If $w \in \aW$ has reduced expression $w = s_{i_1} s_{i_2} \cdots
s_{i_k}$, then the {\it inversion set} of $w$ is the set of real
roots given by
$$
\Inv(w) = \{ \alpha_{i_1}, s_{i_1} \alpha_{i_2}, s_{i_1}s_{i_2}
\alpha_{i_3}, \ldots, s_{i_1} s_{i_2} \cdots s_{i_{k-1}}
\alpha_{i_k}\} \subset \Delta_\re.
$$
It is well known that $|\Inv(w)| = \ell(w)$.

The inversions can be read directly from (the window notation of) an
affine permutation $w \in \aW$ as follows. For a finite positive
root $\alpha = \alpha_{i,j} \in \Delta_0^+$ let
$$m_{\alpha} = \min\{k \mid nk > w^{-1}(i) - w^{-1}(j)\}.$$ Then if
$m_{\alpha} > 0$ the affine roots $\alpha, \ldots,
\alpha+(m_{\alpha}-1) \delta$ are inversions of $w$, while if
$m_{\alpha} < 0$ the affine roots $\delta - \alpha, \ldots, (-
m_{\alpha}) \delta - \alpha$ are inversions of $w$. If $m_{\alpha} =
0$ neither $\alpha$ nor $-\alpha$ are inversions of $w$.  In
particular, if $\alpha + m \delta$ (resp.~$m\delta-\alpha$) is an
inversion for $m
> 0$, then so is $\alpha + m' \delta$ (resp.~$m'\delta-\alpha$) for $0 \leq m' \leq m$.

Let $\i = i_1 i_2 i_3 \cdots$ be either a finite, or (countably)
infinite word with letters from $\Z/n\Z$.  We call $\i$ reduced if
$w_\i^{(k)} = s_{i_1}s_{i_2} \cdots s_{i_k} \in \aW$ has length $k$
for every $k$.  We define the inversion set of $\i$ to be $\Inv(\i)
= \cup_k \Inv(w_\i^{(k)}) \subset \Delta^+_\re$. We call a subset $I
\subset \Delta^+_\re$ an inversion set if $I = \Inv(\i)$ for some
finite or infinite reduced word $\i$.  If $w \in \aW$ then by
$w^\infty$ we mean the infinite word obtained by repeating a reduced
word for $w$.  By Lemma \ref{lem:tlength}, $t^\infty$ is reduced for
any translation.  If $\i$ is an infinite word, and $w \in \aW$ we
may write $w \i$ for the infinite word obtained by prepending $\i$
with a reduced word of $w$.  Note that if $w \i$ is reduced then
\begin{equation}\label{eq:inv}
\Inv(w \i) = \Inv(w) \sqcup w \cdot \Inv(\i)
\end{equation}

  Biconvex sets were studied and classified in the
case of an arbitrary affine Weyl group by Ito \cite{Ito}, and
Cellini and Papi \cite{CP} (under the name of compatible sets).

\begin{thm}[{\cite{Ito, CP}}]\label{thm:ICP}
Suppose $I \subset \Delta^+_\re$ is infinite.  Then the following
are equivalent:
\begin{enumerate}
\item $I$ is an inversion set;
\item $I$ is biconvex;
\item $I = \Inv(v t^\infty) = \Inv(v) \sqcup v \cdot
\Inv(t^\infty)$ for some $v \in \aW$ and translation element $t$
such that $vt^\infty$ is reduced.
\end{enumerate}
\end{thm}
For completeness, we provide a proof of Theorem \ref{thm:ICP}.
\begin{proof}
It is well known, and easy to prove by induction, that $\Inv(w)$ is
biconvex for $w \in \aW$.  Since increasing unions of biconvex sets
are biconvex, we have (1) implies (2).  Since (3) implies (1) is
obvious, it suffices to show that every infinite biconvex set $I$ is
of the form $\Inv(v t^\infty)$. Let $\{m_\alpha \mid \alpha \in
\Delta_0^+\}$ be as in Proposition \ref{prop:table}.  We claim that
the $m_\alpha$-s can be all made into $\{0,+\infty,-\infty\}$, in
finitely many steps, by a sequence of the following operations on
$I$: take some $\alpha_i \in I$ for $i \in \Z/n\Z$ (the root
$\alpha_0 = \delta - \alpha_{1,n}$ is allowed), then change $I$ to
$s_i \cdot (I - \{\alpha_i\})$.  From the definitions, one sees that
$s_i \cdot (I - \{\alpha_i\})$ is still biconvex, and that in this
way the ``finite'' $m_\alpha$-s can be made closer to $0$.  This
sequence of operations corresponds to the $v \in \aW$ of (3).  To
complete the proof, Proposition \ref{prop:bijbraid} below shows that
if every $m_\alpha \in \{0,+\infty,-\infty\}$ then $I$ is of the
form $\Inv(t_\lambda^\infty)$ for some $\lambda \in Q^\vee$.
\end{proof}

\subsection{Blocks and the braid arrangement}
The braid arrangement is the finite, central hyperplane arrangement
in $\R^n = \{(x_1,x_2,\ldots,x_n) \mid x _i\in \R\}$ given by the hyperplanes $x_i -
x_j = 0$, for $i < j$. Equivalently, the hyperplanes may be written
as $\ip{\alpha_{i,j},x} = 0$.

A {\it pre-order} $\preceq$ on a set $S$ is a reflexive, transitive
relation. Any pre-order determines an equivalence relation: $s \sim
s'$ if $s \preceq s'$ and $s' \preceq s$.  A pre-order is called
{\it total} if the induced partial order on equivalence classes is a
total order, or equivalently, if every pair of elements can be
compared. The faces of the braid arrangement are in bijection with
total pre-orders on $[n]$. Total pre-orders are essentially the same
as set compositions.  For example, the pre-order
$\{2,4\}\prec\{1,5\}\prec \{3\}$ corresponds to the set composition
$\Gamma = (\{2,4\},\{1,5\},\{3\})$, which corresponds to the open
face $F = \{(x_1,x_2,x_3,x_4,x_5) \mid x_2 = x_4 < x_1 =x_5 < x_3\}$
of the braid arrangement.

We say that two infinite biconvex sets $I$ and $J$ are {\it in the
same block} if $|I - J| + |J - I|$ is finite.

\begin{prop}\label{prop:bijbraid}
The map $F \mapsto \Inv(t_\lambda^\infty)$ establishes a bijection
between the (open) faces of the braid arrangement (excluding the
lowest dimensional face) and blocks of infinite biconvex sets, where
$\lambda$ is any element of $Q^\vee \cap F$.
\end{prop}

It is clear that a block $B$ is determined uniquely by knowing which
$m_{\alpha}$-s are infinite, and among those which are $+ \infty$
and which are $- \infty$.  Using the $m_\alpha$-s, we define a
relation $\preceq_B$ on $[n]$ as follows: if $i < j$ in $[n]$ then
$i \preceq j$ if $m_{\alpha_{i,j}}$ is finite or $+\infty$, and $j
\preceq i$ if $m_{\alpha_{i,j}}$ is finite or $-\infty$.

\begin{lem}
The relation $\preceq_B$ defined above is a total pre-order.
\end{lem}
\begin{proof}
It suffices to show that $\preceq$ is transitive.  Suppose $i
\preceq j$ and $j \preceq k$.  There are several cases to consider.
We take for example the case $i < k < j$.  In that case we have
$m_{\alpha_{i,j}}$ is finite or $+ \infty$ and $m_{\alpha_{k,j}}$ is finite or $-\infty$. 
Then looking through the table in Proposition \ref{prop:table}, we deduce that
$m_{\alpha_{i,k}}$ is either finite or $+ \infty$, which implies $i \preceq k$.  The
other cases are similar.
\end{proof}

\begin{proof}[Proof of Proposition \ref{prop:bijbraid}]
Any $\lambda \in Q^\vee$ gives rise to an infinite biconvex set: we
have, for each $\alpha \in \Delta_0^+$,
$$
m_\alpha = \begin{cases}
\infty & \mbox{if $\ip{\alpha,\lambda} < 0$,}\\
-\infty & \mbox{if $\ip{\alpha,\lambda} > 0$,}\\
0 & \mbox{otherwise.}
\end{cases}
$$
It is clear from this that $\Inv(t_\lambda^\infty)$ depends exactly
on the face of the braid arrangement that $\lambda$ lies in.  Now
let $I$ be any infinite biconvex set, and let $\preceq$ be the
pre-order constructed above.  Let $F$ be the face of the braid
arrangement corresponding to $\preceq$.  One checks from the
definitions that $\Inv(t_\lambda^\infty)$ and $I$ are in the same
block, where $\lambda \in Q^\vee \cap F$.
\end{proof}

\begin{example}
Let us take the face $F = \{(x_1,x_2,x_3,x_4,x_5) \mid x_2 = x_4 < x_1 =x_5 < x_3\}$ as in the example above. The corresponding block is determined by the conditions $$m_{\alpha_{1,2}} = m_{\alpha_{1,4}} = m_{\alpha_{3,4}} = m_{\alpha_{3,5}} = - \infty,$$ $$m_{\alpha_{1,3}} = m_{\alpha_{2,3}} = m_{\alpha_{2,5}} = m_{\alpha_{4,5}} = \infty,$$ $m_{\alpha_{2,4}}$ and $m_{\alpha_{1,5}}$ are finite. One choice of $\lambda \in Q^\vee \cap F$ is $\lambda = (0,-1,2,-1,0)$, the window notation for the corresponding translation is $t_{\lambda} = [1,-3,13,-1,5]$ and one possible reduced factorization is $$t_{\lambda} = s_2 s_4 s_3 s_1 s_0 s_4 s_3 s_2 s_1 s_0 s_2 s_1 s_4 s_3.$$
\end{example}

\subsection{Braid limits}
Let $\j$ and $\i$ be infinite words in $\Z/n\Z$.  We shall say that
$\j$ is a {\it braid limit} of $\i$ if it can be obtained from $\i$
by a possibly infinite sequence of braid moves.  More precisely, we
require that one has $\i = \j_0, \j_1, \j_2, \ldots$ such that
$\lim_{k \to \infty} \j_k = \j$ and each $\j_k$ differs from
$\j_{k+1}$ by finitely many braid-moves.  Here, the limit $\lim_{k
\to \infty} \j_k = \j$ of words is taken coordinate-wise: $j_r =
\lim_{k \to \infty} (\j_k)_r$.  We write $\i \to \j$ to mean there
is a braid limit from $\i$ to $\j$.

\begin{lem}\label{lem:redtored}
Suppose $\i$ is an infinite reduced word, and $\i \to \j$.  Then
$\j$ is also an infinite reduced word.
\end{lem}

The converse of Lemma \ref{lem:redtored} is false.  The following
example illustrates this, and also the fact that $\i \to \j$ does
not imply $\j \to \i$.

\begin{example}\label{ex:012}
Let $n = 3$, and $\i = 1(012)^\infty = 1012012012 \cdots$, which one
can check is reduced.  Then $\i \to \j = (012)^\infty = 012012012
\cdots$:
\begin{align*}
& \;\;\;\;\,\underline{101}2012012 \cdots \\
& \sim 01\underline{020}12012 \cdots \\
& \sim 0120\underline{212}012 \cdots \\
& \sim 012012\underline{101}2 \cdots \\
& \sim \cdots
\end{align*}
However, there is no braid limit from $\j$ to $\i$ since no braid
moves can be performed on $\j$ at all.  The same calculation also
shows that $11012012012 \cdots \to \i$.
\end{example}

Two infinite reduced words $\i$ and $\j$ are called {\it
braid-equivalent} if there are braid limits $\i \to \j$ and $\j \to
\i$.  Indeed, braid limits define a preorder on the set of all
infinite reduced words, and the equivalence classes of this preorder
are exactly braid-equivalent infinite reduced words.

\begin{lem}\label{lem:braidlimit}
Let $\i$ and $\j$ be infinite reduced words.  We have $\Inv(\j)
\subset \Inv(\i)$ if and only if there is a braid limit from $\i$ to
$\j$.
\end{lem}
\begin{proof}
Assume there is a braid limit $\i = \j_0, \j_1, \j_2, \ldots$ from
$\i$ to $\j$. For every initial part $w_\j^{(m)}$ there is a large
enough $k$ such that $(\j_k)_r = \j_r$ for $r = 1, \ldots, m$. Since
in passing from $\i=\j_0$ to $\j_k$ only finitely many braid moves
happened,  $w_\j^{(m)}$ is initial for some $w_\i^{(M)}$, and thus
$\Inv(w_\j^{(m)}) \subset \Inv(w_\i^{(M)})$. Since such $M$ can be
found for any $m$, we conclude that $\Inv(\j) \subset \Inv(\i)$.

Assume now $\Inv(\j) \subset \Inv(\i)$. Since $s_{j_1}$ is initial
in $\j$, the corresponding simple root $\alpha_{j_1} \in \Inv(\j)$
and thus $\alpha_{j_1} \in \Inv(\i)$. Thus for sufficiently large
$m$, the simple root $\alpha_{j_1} \in \Inv(w_\i^{(m)})$. Let $\j_1$
be obtained from $\i$ by applying braid moves to the first $m$
factors to place $s_{j_1}$ in front.  Now apply \eqref{eq:inv} to
$\j = (j_1)(j_2 j_3 \cdots)$ and $\j_1 = (j_1)(j'_1 j'_2 \cdots)$ to
see that $\Inv(j_2 j_3 \cdots) \subset \Inv(j'_1 j'_2 \cdots)$.
Repeating the argument, we construct a braid limit from $\i$ to
$\j$.
\end{proof}

\begin{cor}\label{cor:braidequiv}
Suppose $\i$ and $\j$ are two infinite reduced words. Then $\Inv(\i)
=\Inv(\j)$ if and only if they are braid-equivalent.
\end{cor}

\subsection{Exchange lemma for infinite reduced words}
The following result follows immediately from the usual exchange
lemma \cite{Hum}.
\begin{lemma}\label{lem:exchange}
Let $\i$ be an infinite reduced word and $j \in \Z/n\Z$.  Then
either
\begin{enumerate}
\item
$j \i$ is an infinite reduced word, or
\item
there is a unique index $k$ such that $\i' = i_1 i_2 \cdots i_{k-1}
i_{k+1} \cdots$ is reduced and such that $s_j w_\i^{(k)} =
w_\i^{(k-1)}$.
\end{enumerate}
\end{lemma}

For example, let $n=3$ and let $\i = (01 2)^2 1 (0 1 2)^{\infty}$. Then $s_1 w_{\i} = (s_0 s_1 s_2)^{\infty}$, so that $k=7$, while $s_2 w_{\i}$ is reduced.

Let $\i$ and $\j$ be two infinite reduced words.  We say that $\j$
is obtained from $\i$ by {\it infinite exchange} if Case (2) of
Lemma \ref{lem:exchange} always occurs when we place $j_1$ in front
of $\i$, then place $j_2$ in front of the resulting $\i'$, and so
on.  For example, with $\i$ and $\j$ as in Example \ref{ex:012}, $\j$ is obtained from $\i$ by infinite exchange: $j_1 = 0$ is exchanged for the second $1$ in $\i$, then $j_2 = 1$ is exchanged for the first $1$ in $\i$, then $j_3 = 2$ is exchanged for the second $0$ in $\i$, and so on.  It is straightforward to see that

\begin{prop}\label{prop:limitexchange}
There is a braid limit $\i \to \j$ if and only if $\j$ can be
obtained from $\i$ by infinite exchange.
\end{prop}

\begin{remark}
If $\j$ is obtained from $\i$ by infinite exchange then every letter
of $\i$ is eventually ``exchanged''.  (However, the analogous
statement fails for arbitrary Coxeter groups.)
\end{remark}

\subsection{Limit weak order}
We call a braid-equivalence class $[\i]$ of infinite reduced words a
{\it limit element} of $\aW$.  We let $\bW$ denote the set of limit
elements of $\aW$.  We define a partial order, called the {\it limit
weak order}, on $\bW$ by
$$
[\i] \leq [\j] \Leftrightarrow \Inv(\i) \subset \Inv(\j).
$$
Equivalently, by Lemma \ref{lem:braidlimit}, $[\i] \leq [\j]$ if
and only if there is a braid limit $\j \to \i$.  This partial order
does not appear to have been studied before. It is clear that one
also obtains a partial order on $\aW \cup \bW$.

\begin{thm}
The partial order $(\aW \cup \bW,\leq)$ is a meet semi-lattice.
\end{thm}

\begin{proof}
It is known (\cite[Theorem 3.2.1]{BB}) that $\aW$ is a meet
semi-lattice.  Let us take infinite reduced words $\i$ and $\j$ (the
other case of $w \in \aW$ and $\i \in \bW$ is similar).  Let $w_k$
be the initial part of length $k$ of $\i$, and let $u_k$ be the
initial part of length $k$ of $\j$. Then since weak order on the
affine symmetric group is a meet-semilattice, we can define $v_k =
w_k \wedge u_k$, where $\wedge$ denotes the meet operation. Then
$v_k \leq v_{k+1}$, since $v_k \leq w_k \leq w_{k+1}$ and
$v_k \leq u_k \leq u_{k+1}$. If the sequence $v_1 \leq
v_{2}\leq \cdots$ stabilizes then we obtain an element $w$ of
$\aW$. Otherwise, we obtain an element $[\k]$ of $\bW$.  It is clear
that $w$ or $[\k]$ is indeed the maximal lower bound of $[\i]$ and
$[\j]$.
\end{proof}

We say that $\i$ and $\j$, or $[\i]$ and $[\j]$, are {\it in the
same block} if $\Inv(\i)$ and $\Inv(\j)$ are.  The partial order
$(\bW,\leq)$ descends to blocks: we have $B \leq B'$ if one
can find $[\i] \in B$ and $[\j] \in B'$ such that $[\i] \leq
[\j]$.  It is convenient to also consider $\aW$ as a block by
itself.  The following strengthening of Proposition
\ref{prop:bijbraid} is immediate.

\begin{theorem}\label{thm:blocks}
The map of Proposition \ref{prop:bijbraid} gives a poset isomorphism
of the partial order of blocks of $\aW \cup \bW$ and the inclusion
order of the faces of the braid arrangement.
\end{theorem}

Note that the maximal blocks of $\bW$ consist of single elements,
which are the maximal elements of $\bW$.

If we label the faces of the braid arrangement by set compositions,
then the inclusion order on faces is the refinement order on set
compositions: $\Gamma = (\gamma_1,\gamma_2,\ldots,\gamma_k) \preceq
\Gamma'= (\gamma'_1,\gamma'_2,\ldots,\gamma'_r)$ if and only if
$\gamma_1 = \cup_{i=1}^{r_1} \gamma'_i$, $\gamma_2
=\cup_{i=r_1+1}^{r_1+r_2}\gamma'_i$, and so on.

The infinite translation elements $t^\infty$ are exactly the minimal
elements in a block.  More generally, if $w \in \aW$ is so that
$w^\infty$ is reduced, then $[w^\infty]$ is minimal in its block. To
see this, write $w = vt$ where $v \in W$.  Then for some $m$ we have
$v^m = 1$, so that $w^m$ will be a translation element.

\begin{thm}[{cf. \cite[Proposition 3.9]{CP}
\cite{Ito}}]\label{thm:blockorder} Suppose a block $B$ of $\bW$
corresponds to a set composition with sets of sizes
$a_1,a_2,\ldots,a_k$, where $\sum_i a_i = n$.  Then the restricted
partial order $(B, \leq)$ is isomorphic to the product
$\tilde{S}_{a_1} \times \tilde{S}_{a_2} \times \cdots \times
\tilde{S}_{a_k}$ of weak orders of (smaller) affine symmetric
groups.  In particular, limit weak order is graded when restricted
to a block.
\end{thm}

\begin{proof}
Fixing the block $B$ fixes some infinite or negative infinite values
of certain $m_\alpha$-s.  To check if an assignment of specific
finite values to the remaining $m_{\alpha}$-s is biconvex, it
suffices to check only the triples $\alpha+\beta=\gamma$ such that
all three values $m_\alpha$, $m_\beta$, and $m_\gamma$ are finite.
Such finite values correspond to the equivalence classes of the
total pre-order $\preceq_B$ on $[n]$ associated to $B$, or
equivalently, to the parts of the set-composition $\Gamma$.  For
each part $\gamma \subset [n]$ of $\Gamma$ one has to choose finite
$m_\alpha$-s corresponding to an element of $\tilde{S}_{|\gamma|}$.
\end{proof}

\begin{example}
Let $\Gamma = (\{2,4,5\},\{1,3\})$ and let $B$ be the corresponding block. Then an element of $B$ is uniquely determined by the (finite) values of $m_{\alpha_{2,4}}$, $m_{\alpha_{2,5}}$, $m_{\alpha_{4,5}}$ and $m_{\alpha_{1,3}}$. The first three values determine an element of $\tilde S_{3}$, while the last one determines an element of $\tilde S_2$. Thus this block is isomorphic to the weak order of $\tilde S_3 \times \tilde S_2$.
\end{example}

\begin{remark} Theorem \ref{thm:blocks} can be interpreted in terms
of the Tits cone in the geometric realization of $\aW$.  An infinite
reduced word can be thought of as an infinite sequence of chambers
in the Tits cone, starting from the fundamental chamber.  Theorem
\ref{thm:ICP}(3) (together with Corollary \ref{cor:braidequiv}) can
be interpreted as saying that every such sequence is
braid-equivalent to a sequence which starts off with a finite
sequence of moves (determined by some initial Weyl group element),
and then heads straight in some direction (determined by the
translation element) for the remaining infinite sequence of moves.

One way to pick such a direction is to pick a point on the boundary
of the Tits cone, which in this case is simply a hyperplane.  The
line joining an interior point of a chamber to a non-zero point in
the boundary passes through infinitely-many chambers, and gives the
trailing infinite sequence of moves.  The intersection of the
hyperplane arrangement with the boundary of the Tits cone is simply
the (finite) braid arrangement (which in some contexts is called the
spherical building at infinity). This gives a geometric
interpretation of the classification of Theorem \ref{thm:blocks}.
\end{remark}
\subsection{Explicit reduced words}\label{subsec:explicit}
Let $B$ be a block corresponding to a set composition $\Gamma =
(\gamma_1,\gamma_2,\ldots,\gamma_k)$.  We now explain how to write
down an explicit infinite reduced word for the minimal element of
$B$.  Let $\lambda$ be a point in the open face $F$ corresponding to
$\Gamma$.  Thus $\lambda_i = \lambda_j$ if $i, j \in \gamma_r$ for
some $r$, and $\lambda_i < \lambda_j$ if $i \in \gamma_r$, $j \in
\gamma_s$ for $r < s$.  For example, if $\Gamma =
(\{2,4\},\{1,5\},\{3\})$, we may pick $\lambda = (2,1,3,1,2)$. Here we drop the convention that $\sum_i \lambda_i = 0$. We may act on $\lambda$ with simple generators $s_i$ (acting on
positions), where $s_0$ acts by swapping $\lambda_1$ and
$\lambda_n$.

For simplicity, let us suppose that we have chosen (the unique)
$\lambda$ such that $$\{\lambda_1,\lambda_2,\ldots,\lambda_n\} =
\{1,2,\ldots,k\},$$ as in the above example. Let us start with
$\lambda$, act with the $s_i$, and suppose after a sequence of $p$
moves $s_{i_1}, s_{i_2}, \ldots, s_{i_p}$ we obtain $\lambda$
again, while adhering to the following conditions:
\begin{enumerate}
\item[(A)]
Acting with $s_i$ creates a descent at each step.  In other words,
we may act with $s_i$ on $\lambda$ if $\lambda_i < \lambda_{i+1}$
(indices taken modulo $n$).
\item[(B)] For each $r \in \{1,2,\ldots,k-1\}$, at some point we swap $r$
with $r+1$.
\end{enumerate}

\begin{prop}\label{prop:explicit}
Let $w = s_{i_1}s_{i_2} \cdots s_{i_p}$.  Then $w^\infty$ is reduced
and $\Inv(w^\infty)$ is the minimal element of the block $B$.
\end{prop}
\begin{proof}
To see that $w$ is reduced, decorate the $1$'s inside $\lambda$ as
$1_1, 1_2, \ldots, 1_r$ from left to right, and similarly for the
$2$'s.  We may thus think of $\lambda$ as an affine permutation
under the ordering $1_1 < 1_2 < \cdots < 1_r < 2_1 < 2_2 < \cdots$,
with $w$ acting on the right.  If a letter moves from $\lambda_1$ to
$\lambda_n$ (resp.~$\lambda_n$ to $\lambda_1$), we increase its
``winding number'' by 1 (resp.~decrease by 1). One can check that
Condition (A) translates to the fact that the length of the affine
permutation is always increasing: whenever we swap a letter $a$ with
$b$ where $a < b$, it is always the case that the letter $b$ has a
greater winding number than the letter $a$.  It follows that $w$,
and similarly also $w^\infty$ is reduced.

By the comment before Theorem \ref{thm:blockorder} it follows that
$\Inv(w^\infty) =\Inv(t^\infty)$ for some translation $t$.  If we
imagine the letters in $\lambda$ repeated indefinitely in both
directions, we obtain a ``sea'' of 1's, 2's, 3's, and so on.
Condition (B) says that as we act with $w$, the 2's travel to the
left with respect to the 1's, and the 3's travel to the left with
respect to the 2's, and so on. This implies that $t$ is in the same
face of the braid arrangement as $\lambda$.
\end{proof}

\begin{example}
For $\lambda = (2,1,3,1,2)$ as above one possible choice of $w$ is $$w = s_2 s_1 s_4 s_0 s_1 s_4 s_3.$$ One can calculate that $w^2 = t_{(0,-1,2,-1,0)}$, noting that $(0,-1,2,-1,0)$ is in the same open face of the braid arrangement as $\lambda$.  The resulting action on $\lambda$ is $$(2,1,3,1,2) \to (2,3,1,1,2) \to (3,2,1,1,2) \to (3,2,1,2,1) \to$$ $$(1,2,1,2,3) \to (2,1,1,2,3) \to (2,1,1,3,2) \to (2,1,3,1,2),$$ and one easily checks that both conditions (A) and (B) are satisfied.
\end{example}

\begin{example} \label{ex:hex}
For $n=3$ the face complex of the braid arrangement is dual to the face complex of a hexagon, its edges and vertices correspond to the $12$ blocks in $\bW$. If we label vertices and edges of the hexagon by the corresponding set compositions, reading them in the circular order would produce the following list: $(\{1\},\{2\},\{3\})$, $(\{1,2\},\{3\})$, $(\{2\},\{1\},\{3\})$, $(\{2\},\{1,3\})$, $(\{2\},\{3\},\{1\})$, $(\{2,3\},\{1\})$, $(\{3\},\{2\},\{1\})$, $(\{3\},\{1,2\})$, $(\{3\},\{1\},\{2\})$, $(\{1,3\},\{2\})$, $(\{1\},\{3\},\{2\})$, $(\{1\},\{2,3\})$. A list of corresponding possible choices of $w$-s for each of the blocks would be $s_1s_2s_1s_0$, $s_2s_1s_0$, $s_2s_1s_0s_1$, $s_2s_0s_1$, $s_2s_0s_2s_1$, $s_0s_2s_1$, $s_0s_1s_2s_1$, $s_0s_1s_2$, $s_1s_0s_1s_2$, $s_1s_0s_2$, $s_1s_0s_2s_0$, $s_1s_2s_0$.
\end{example}

\subsection{Infinite Coxeter elements}
Recall that a Coxeter element $c \in \aW$ is an element with a
reduced word which uses each $i \in \Z/n\Z$ exactly once.  It is a
standard fact that Coxeter elements of $\aW$ are in bijection with
acyclic orientations of the Dynkin diagram of $\aW$, which is a
$n$-cycle labeled by $\Z/n\Z$: the simple generator $s_i$ occurs to
the left of $s_{i+1}$ in $c$ if and only if the edge $(i, i+1)$
points from $i+1$ to $i$.  From any such acyclic orientation $O$, we
obtain a set composition $\Gamma_O$ of $[n]$ with two parts:
$$
\Gamma_O = (\{i \mid i-1 \to i \; \text{in} \; O\},\{i \mid i \to
i-1 \; \text{in} \; O\}).
$$

\begin{prop}
Let $c$, $O$, $\Gamma_O$ correspond under the above bijections. Then
the affine permutation $w \in \aW$ of Proposition
\ref{prop:explicit} can be chosen to be $c$.
\end{prop}
\begin{proof}
In the case of a two-part set composition, the vector $\lambda$
consists of 1's and 2's only.  The action of $c$ moves each $2$ left
to the position of the next $2$.
\end{proof}

As a consequence we obtain the following result.
\begin{cor}\label{cor:infcoxred}
An infinite Coxeter element $c^\infty$ is reduced.
\end{cor}
Corollary \ref{cor:infcoxred} was proved by Kleiner and Pelley \cite{KP} in the much
more general Kac-Moody setting (see also \cite{Spe}).

We have thus given explicit bijections between the following sets:
Coxeter elements of $\aW$, acyclic orientations of a $n$-cycle, set
compositions of $[n]$ with two parts, total pre-orders on $[n]$ with
two equivalence classes, and edges of the braid arrangement.

\begin{example}
The Coxeter element $s_2 s_4 s_0 s_1 s_3 \in \tilde S_5$ corresponds to the pentagon orientation $1 \longrightarrow 2 \longleftarrow 3 \longrightarrow 4 \longleftarrow 5 \longleftarrow 1$, to the set composition $\Gamma = (\{2,4\},\{1,3,5\})$, to the total pre-order $\{2,4\} \prec \{1,3,5\}$, to the edge $F = \{(x_1,x_2,x_3, x_4, x_5) \mid x_2 = x_4 < x_1 =x_3 = x_5\}$.
\end{example}

An infinite reduced word $\i$, or its equivalence class $[\i]$, is
{\it fully commutative} if no (3-term) braid moves $i(i+1)i
\rightarrow (i+1)i(i+1)$ can be applied to any $\j \in [\i]$.

\begin{lem} \label{lem:commutlimit}
Suppose $\i \to \j$ is a braid limit of infinite reduced words where
only commutation moves $i\,j \sim j \,i$ where $|i - j|> 1$ are
used. Then $\i$ and $\j$ are braid equivalent.
\end{lem}
\begin{proof}
Let $i\in \Z/n\Z$.  Between every two occurrences of $i$ in $\i$ one
has the reduced word of a (rotation of a) usual finite permutation.
It follows any two consecutive occurrences of $i$ cannot be too far
apart.  Thus in particular, any particular letter in $\j$ only
traveled a finite distance from its original position in $\i$. Using
this and the definition of braid limit, one can construct a braid
limit $\j \to \i$.
\end{proof}

\begin{thm}\label{thm:CoxTFAE}
Let $\i$ be an infinite reduced word.  The following are equivalent:
\begin{enumerate}
\item $[\i] = [c^\infty]$, where $c$ is a Coxeter element,
\item $[\i]$ is fully commutative,
\item $[\i]$ is a minimal element of $(\bW,\leq)$.
\end{enumerate}
\end{thm}

\begin{proof}
We have already established the equivalence of (1) and (3).  Since
in $c^{\infty}$ there is an occurrence of $s_{i-1}$ and $s_{i+1}$
between any two consecutive occurrences of $s_i$, it is clear that
no braid move can possibly be applied and thus $c^{\infty}$ is fully
commutative, giving (1) implies (2).  On the other hand, suppose
$\i$ is fully commutative, and we have a braid limit $\i \to \j$.
Then only commutation moves occurred in the braid limit $\i \to \j$,
so by Lemma \ref{lem:commutlimit}, we have $[\j] =[\i]$.  This shows
that (2) implies (3).
\end{proof}

\section{$\Omega$}
\label{sec:omega}

\subsection{Infinite products of Chevalley generators}
Let $\i = i_1 i_2 \cdots$ be an infinite reduced word (or simply an
infinite word), and $\a = (a_1,a_2,\ldots) \in \R_{>0}^\infty$.  By
\cite[Lemma 7.1]{LP}, the limit
$$
e_\i(\a) := e_{i_1}(a_1) e_{i_2}(a_2) \cdots = \lim_{k \to \infty}
e_{i_1}(a_1) \cdots e_{i_k}(a_k)
$$
converges if and only if $\sum_i a_i < \infty$.  We let $\ell^1_{>0}
\subset \R_{>0}^\infty$ denote the set of infinite sequences of
positive real numbers with finite sum.  We may consider $e_\i$ as a
map $\ell^1_{>0} \to U_{\geq 0}$.  We let $E_\i := \im(e_\i) \subset
U_{\geq 0}$ denote the image of $e_\i$, as in the finite case. 

\begin{example} \label{ex:X}
Take $n=3$, $0 < a < 1$ and consider the following element of $E_{(0 1 2)^{\infty}}$: $$X = e_0(1) e_1(1) e_2(1) e_0(a) e_1(a) e_2(a) \dotsc e_0(a^k) e_1(a^k) e_2(a^k) \dotsc = \prod_{k \geq 0} (e_0(a^k) e_1(a^k) e_2(a^k)).$$ Denote $\eta(i,j) = \sum_{i < r < j} (j-r)$ as $r$ assumes all values in the range that are divisible by $3$. For example, $\eta(2,6) = 6-3 = 3$. One can compute that $$x_{i,j} = a^{\eta(i,j)} \prod_{t = 1}^{j-i} (1-a^{t})^{-1} .$$ Using this formula, one computes: $X = $
$$ \hspace{-23pt}
\left(
\begin{array}{c|ccc|ccc|c}
\ddots & \vdots &  \vdots &  \vdots &  \vdots & \vdots&\vdots \\
\hline
\dots& 1 & \frac{1}{1-a} & \frac{1}{(1-a) (1-a^2)} & \frac{a}{(1-a) (1-a^2) (1-a^3)} &\frac{a^2}{(1-a) (1-a^2) (1-a^3)(1-a^4)} &  \frac{a^3}{(1-a) (1-a^2) (1-a^3) (1-a^4) (1-a^5)} & \dots \\
\dots&0&1 & \frac{1}{1-a} &\frac{a}{(1-a) (1-a^2)}& \frac{a^2}{(1-a) (1-a^2) (1-a^3)} & \frac{a^3}{(1-a) (1-a^2) (1-a^3)(1-a^4)} & \dots \\
\dots& 0 &0&1 & \frac{1}{1-a}&\frac{1}{(1-a) (1-a^2)} & \frac{1}{(1-a) (1-a^2) (1-a^3)} &\dots \\
\hline
\dots& 0 & 0 & 0 & 1 & \frac{1}{1-a}& \frac{1}{(1-a) (1-a^2)} & \dots \\
\dots& 0 & 0 & 0 &0&1 & \frac{1}{1-a} & \dots \\
\dots& 0 & 0 & 0 &0&0&1 & \dots \\
\hline & \vdots & \vdots &  \vdots &  \vdots  & \vdots & \vdots &
\ddots\end{array} \right)
$$
\end{example}

By analogy with Lemma \ref{lem:Einv} it can be shown that $E_{\i} =
E_{\j}$ if $[\i] = [\j]$, cf. Corollary \ref{cor:Ebraid}. Other
properties that hold in finite case do not extend however. For
example, sets $E_{\i}$ and $E_{\j}$ may have non-empty intersection
even if $[\i] \not = [\j]$ (Corollary \ref{cor:Eprec}), and the map
$\a \mapsto e_{\i}(\a)$ is not always injective (Proposition
\ref{prop:notinjective}).

We define
$$
\Omega := \bigcup_{\i} E_\i
$$
where the union is over all infinite reduced words.

\begin{lem}\label{lem:entire}
Every $X \in \Omega$ is doubly-entire and totally positive.
\end{lem}
\begin{proof}
That $X$ is doubly-entire follows from \cite[Lemma 7.2]{LP} applied
to $X$ and $X^{-c}$.  Now suppose $X \in \Omega$ is not totally
positive.  Then by \cite[Lemma 5.8 and Lemma 5.10]{LP},
$\epsilon_i(X) > 0$ for all $i$, so $X$ cannot be entire.  (See
Section \ref{sec:epsilon} for the definition of $\epsilon_i(X)$.)
\end{proof}

The following result shows that we do not lose anything by only
considering reduced words.  The proof will be delayed until Section
\ref{sec:proofnonreduced}.

\begin{prop}\label{prop:nonreduced}
We have
$$
\Omega \cup U_{\geq 0}^\pol= \bigcup_{\i} E_\i
$$
where the union is taken over {\it all} (not necessarily reduced)
infinite or finite words.
\end{prop}

\subsection{Braid limits in total nonnegativity}
Suppose $\i \to \j$ is a braid limit of infinite reduced words.
Applying \eqref{E:chevrel1} and \eqref{E:chevrel2} possibly an
infinite number of times, we obtain a map $R_\i^\j: \ell^1_{>0} \to
\ell^1_{>0}$.  This map is well-defined because by the definition of
braid limit, any coordinate of $\a' = R_\i^\j(\a)$ will eventually
stabilize; in addition, the moves \eqref{E:chevrel1} and
\eqref{E:chevrel2} preserve the sum of parameters, so the image lies
in $\ell^1_{>0}$.

\begin{prop}\label{prop:TPlimit}
Let $\i$ and $\j$ be infinite reduced words.  The map $R_\i^\j$ does
not depend on the braid limit $\i \to \j$ chosen.
\end{prop}
\begin{proof}
Suppose we are given two braid limits $\i \to_1 \j$ and $\i \to_2
\j$.  Let $\i = \j_0,\j_1,\j_2,\ldots$ be the sequence of infinite
reduced words for $\i \to_1 \j$, and let $\i =
\k_0,\k_1,\k_2,\ldots$ be the sequence for $\i \to_2 \j$.

Let $\c \in \ell^1_{>0}$ and write $\a = (R_\i^\j)_1(\c)$, and $\b =
(R_\i^\j)_2(\c)$.  For each $r > 0$, we shall show that
$(a_1,a_2,\ldots,a_r) = (b_1,b_2,\ldots,b_r)$.  By the definition of
braid limit, we can pick $s$ sufficiently large such that the first
$r$ letters in $\j_s$, and in $\k_s$, are both equal to the first
$r$ letters in $\j$.  Now pick $m$ sufficiently large so that all
the braid moves involved in going from $\i$ to $\j_s$, and from $\i$
to $\k_s$ occurs in the first $m$ letters.  Then $w_{\j_s}^{(m)} =
w_{\k_s}^{(m)}$, so that $(\j_s)_1 (\j_s)_2 \cdots (\j_s)_m$ can be
changed to $(\k_s)_1 (\k_s)_2 \cdots (\k_2)_m$ via finitely many
braid moves, not involving the first $r$ letters.  Using Lemma
\ref{lem:Rwelldefined}, this shows that $(a_1,a_2,\ldots,a_r) =
(b_1,b_2,\ldots,b_r)$.
\end{proof}

\begin{prop}\label{prop:Rtransitive}
For braid limits $\i \to \j \to \k$ we have $R_\i^\k = R_\j^\k \circ
R^\j_\i$.
\end{prop}
\begin{proof}
A pair of braid limits $\i \to \j \to \k$ gives rise to a braid
limit $\i \to \k$, obtained by interspersing the braid moves used in
$\i \to \j$, and those used in $\j \to \k$.
\end{proof}

The following result is one of our main theorems.  We shall give two
proofs of this result, in Sections \ref{sec:proofTPlimit} and
\ref{sec:TPlemma}.

\begin{thm}[TNN braid limit theorem] \label{thm:TPlimit}
Suppose $\i \to \j$ is a braid limit.  Then $e_\i(\a) =
e_\j(R_\i^\j(\a))$.
\end{thm}

\begin{remark}
While Theorem \ref{thm:TPlimit} is an obvious analogue of the Lemma
\ref{lem:Einv} for finite reduced words, it is not true in greater
generality: it fails when considered in arbitrary Kac-Moody groups.
\end{remark}

\begin{cor}\label{cor:Ebraid}
Suppose $\i$ and $\j$ are braid equivalent infinite reduced words.
Then $E_\i = E_\j$.
\end{cor}

\begin{cor}\label{cor:Eprec}
Suppose $[\i] \leq [\j]$ in $(\bW,\leq)$. Then $E_{[\j]} \subset
E_{[\i]}$.
\end{cor}

\begin{cor}\label{cor:Coxeter}
We have $\Omega = \cup_c E_{c^\infty}$, where the union is over all
Coxeter elements $c$.
\end{cor}
\begin{proof}
Follows immediately from Theorem \ref{thm:CoxTFAE} and Corollaries
\ref{cor:Ebraid} and \ref{cor:Eprec}.
\end{proof}

\begin{example}\label{ex:coxeternotdisjoint}
The union in Corollary \ref{cor:Coxeter} is not in general disjoint.
If $c \neq c'$ it is possible to have $E_{c^\infty} \cap
E_{(c')^\infty} \neq \emptyset$.  For example, take $n = 3$.  Then
by Example \ref{ex:hex} one has $(1012)^\infty \to (012)^\infty$
and $(1012)^\infty \to (102)^\infty$.  Thus $E_{(1012)^\infty}
\subset E_{(012)^\infty} \cap E_{(102)^\infty}$.
\end{example}

We will present $\Omega$ as a disjoint union in Section
\ref{sec:ASWcells}.

\section{Injectivity} \label{sec:inj}
By Theorem \ref{thm:fin}, the maps
$$
e_\i: \R_{>0}^\ell \to E_w
$$
are injective for a reduced word $\i$ of $w \in \aW$. The same is
not true for the maps $e_\i: \ell_{>0}^1 \to E_\i$.

\begin{example}\label{ex:1012}
Take $n = 3$ and consider the braid limit $\i = 1(012)^\infty \to
(012)^\infty = \j$ described in Example \ref{ex:012}.  Then using
Theorem \ref{thm:TPlimit} we obtain
$$
e_\i(a_1,a_2,\ldots) = e_1(a) e_\i(a'_1,a_2,\ldots) = e_1(a)
e_\j(\R_\i^\j(a'_1,a_2,\ldots)) = e_\i(a,\R_\i^\j(a'_1,a_2,\ldots))
$$
where $0 < a < a_1$ is arbitrary and $a_1 = a + a'_1$.  We
generalize this in Proposition \ref{prop:notinjective} below.
\end{example}

Similarly, $R_\i^\j$ is a bijection when $\i$ and $\j$ are finite
reduced words, but is neither injective nor surjective for general
infinite reduced words (see Remark \ref{rem:Rnotbij}).

\subsection{Injective reduced words, and injective braid limits}
Let $\i$ be an infinite reduced word.  Then $\i$ is {\it injective}
if the map $e_\i$ is injective.  We shall also say that a braid
limit $\i \to \j$ is injective if $R_\i^\j$ is injective.

\begin{prop}
Injectivity of infinite reduced words depends only on the braid
equivalence class.
\end{prop}
\begin{proof}
Let $\i$ and $\j$ be braid equivalent infinite reduced words.
Suppose $\i$ is not injective, so that $e_\i(\a) = e_\i(\a')$ for
some $\a \neq \a'$. Then by Theorem \ref{thm:TPlimit} we have
$e_\j(R_\i^\j(\a)) = e_\j(R_\i^\j(\a'))$.  By Proposition
\ref{prop:Rtransitive}, we have $R_\j^\i(R_\i^\j(\a)) = \a \neq \a'
= R_\j^\i(R_\i^\j(\a'))$.  Thus $\j$ is not injective either.
\end{proof}

\begin{prop}\label{prop:notinjective}
Let $\i$ be an infinite reduced word which is not minimal in its
block, and let $X \in E_\i$.  Then $e_\i^{-1}(X) \subset
\ell^1_{>0}$ is uncountable.  In particular, $\i$ is not injective.
\end{prop}
\begin{proof}
Let us say that $\i$ has {\it rank} $\rho(\i) = p$ if $|\Inv(\i) -
\Inv(\j)| = p$, where $\j$ is the minimal element in the block of
$\i$.  Using Theorem \ref{thm:ICP}, we may write $\i =
s_{i_1}s_{i_2}\cdots s_{i_r}t^\infty$.  For any reduced expression
$i\, \k$, we have $\rho(i \,\k) - \rho(\k) \in \{0,1\}$.  It follows
that we may write $\i = u\,i\, \j$, where $\j = t_\lambda^\infty$ is
minimal in its own block, and $\rho(i\,\j) = 1$.

It suffices to prove the claim for the case that $u$ is trivial,
since prepending $u$ would not change the non-injectivity.  Now if
$\i = i\, \j$ is reduced and rank 1, then neither $\alpha_i$ not
$\delta-\alpha_i$ lies in $\Inv(\j)$.  It follows that $\ip{\alpha_i,\lambda} =
0$, or equivalently, $s_i \lambda =\lambda$.  (This calculation
holds even if $i = 0$, where for example the inner product is
calculated by setting $\delta = 0$, giving
$\ip{-\alpha_{1,n},\lambda} = 0$.)  But then we have $\Inv(\i) =
\{\alpha_i\} \cup s_i \cdot \Inv(\j) = \{\alpha_i\} \cup \Inv(\j)$ so that $\i \to \j$.  We then have
$$
e_\i(a_1,a_2,\ldots) = e_1(a) e_\i(a'_1,a_2,\ldots) = e_1(a)
e_\j(\R_\i^\j(a'_1,a_2,\ldots)) = e_\i(a,\R_\i^\j(a'_1,a_2,\ldots))
$$
for any $0 < a < a_1 = a + a'_1$.
\end{proof}

\begin{conjecture}\label{conj:injective}
Suppose $\i$ is an infinite reduced word which is minimal in its
block.  Then $e_\i$ is injective.
\end{conjecture}

We shall provide evidence for this conjecture below.  In particular,
for an infinite Coxeter element $\i = c^\infty$, we will find (many)
matrices $X \in E_\i$ such that $|e_\i^{-1}(X)| = 1$. In view of
Corollary \ref{cor:Coxeter}, the case of infinite Coxeter elements
is especially interesting.

\subsection{$\epsilon$-sequences and $\epsilon$-signature}
\label{sec:epsilon} Let $X \in U_{\geq 0}$ be infinitely supported.
Recall from \cite{LP} that for $i \in \Z/n\Z$, we define
$$
\epsilon_i(X) = \lim_{j \to \infty} \frac{x_{i,j}}{x_{i+1,j}}.
$$
Note that this limit is monotonic: $x_{i,j}/x_{i+1,j} \geq
x_{i,j+1}/x_{i+1,j+1} \geq \cdots$.  Call
$(\epsilon_1,\epsilon_2,\ldots,\epsilon_n)$ the $\epsilon$-sequence
of $X$.  

\begin{example}
In Example \ref{ex:X} it is clear that $\eta(0,j) = \eta(1,j) = \eta(2,j)$. We compute $$\epsilon_1 = \lim_{j \to \infty} \frac{a^{\eta(1,j)} \prod_{t = 1}^{j-1} (1-a^{t})^{-1} }  {a^{\eta(2,j)} \prod_{t = 1}^{j-2} (1-a^{t})^{-1} } = \lim_{j \to \infty} (1 - a^{j-1})^{-1} = 1.$$ Similarly $\epsilon_0 = 1$. Finally, $$\epsilon_2 = \lim_{j \to \infty} \frac{a^{\eta(2,j)} \prod_{t = 1}^{j-2} (1-a^{t})^{-1} }  {a^{\eta(3,j)} \prod_{t = 1}^{j-3} (1-a^{t})^{-1} } = \lim_{j \to \infty} a^{j-3} (1 - a^{j-1})^{-1} = 0.$$ Thus the $\epsilon$-sequence of $X$ is $(1, 0, 1)$.
\end{example}

The sequence of $\{0,+\}$'s arising as signs of the
$\epsilon$-sequence is called the {\it $\epsilon$-signature} of $X$.
By \cite[Lemma 7.7]{LP}, for $X \in \Omega$, one cannot have $\epsilon_i(X) > 0$ for
all $i \in \Z/n\Z$, so the $\epsilon$-signature has at least one
$0$.

\begin{example}
For $n = 2$, there are two infinite reduced words $\i =
101010\cdots$ and $\j = 010101 \cdots$.  For $X \in E_\i$, one has
the $\epsilon$-signature $(+,0)$, and for $X \in E_\j$, one has
$(0,+)$.  In this case, the decomposition of \ref{cor:Coxeter} is
disjoint.  Furthermore, Conjecture \ref{conj:injective} holds.  If
$X = e_1(a_1)e_0(a_2)e_1(a_3) \cdots $ then $\epsilon_1(e_1(-a)X) =
0$, and so we must have $a_1 = \epsilon_1(X)$.  Proceeding
inductively, we see that $e_\i$ is injective.
\end{example}

We first establish some basic results about $\epsilon$-signatures.

\begin{lemma} \label{lem:epc}
Assume $n > 2$.  Let $X \in U_{\geq 0}$ and $i \in \Z/n\Z$.  Then
for $a
> 0$,
\begin{enumerate}
   \item if $k \not = i, i-1$ then $\epsilon_k(e_i(a) X) = \epsilon_k(X)$;
   \item $\epsilon_i(e_i(a)X)=\epsilon_i(X) + a > 0$;
   \item $\epsilon_{i-1}(e_i(a)X) > 0$ if and only if
   $\epsilon_{i-1}(X) > 0$ and $\epsilon_{i}(X) > 0$;
\end{enumerate}
\end{lemma}

\begin{proof}
Statements (1) and (2) follow easily from the definition. Statement
(3) follows from the following: $\lim_{q \to \infty} \frac
{x_{p.q}}{x_{p+1,q}} > 0$ and $\lim_{q \to \infty} \frac
{x_{p+1.q}}{x_{p+2,q}} > 0$ if and only if $\lim_{q \to \infty}
\frac {x_{p.q}}{x_{p+1,q} + a x_{p+2,q}} > 0$ (where all $x$'s are
strictly positive and all limits are known to exist).
\end{proof}

\begin{lemma} \label{lem:eps}
Let $\i$ be an infinite reduced word and $i \in \Z/n\Z$ be such that
the first (leftmost) occurrence of $i$ occurs to the left of the
first occurrence of $i+1$ in $\i$.  Then for all $X \in E_{\i}$, we
have $\epsilon_i(X) > 0$.
\end{lemma}

\begin{proof}
This follows immediately from Lemma \ref{lem:epc}(1,2).
\end{proof}

Now suppose $[t^\infty]$ is a braid equivalence class of infinite
reduced words, minimal in its block. Let $\lambda \in \Z^n$ be the
vector used in Subsection \ref{subsec:explicit}.  If $\lambda_i <
\lambda_{i+1}$, then $s_i$ is the first simple generator for some
$\i \in [t^\infty]$, so by Lemma \ref{lem:eps} we have
$\epsilon_i(X) > 0$ for $X \in E_{t^\infty}$.  More generally,

\begin{prop}\label{prop:epsmin}
Let $\Gamma = (\gamma_1,\gamma_2,\ldots,\gamma_k)$ be the set
composition corresponding to $[t^\infty]$, and let $X \in
E_{t^\infty}$.  Then $\epsilon_i(X)
> 0$ for any $(i+1) \notin \gamma_1$.
\end{prop}
\begin{proof}
If $\lambda_{i+1} > 1$, then one can perform the algorithm of
Subsection \ref{subsec:explicit} in such a way that $s_i$ is
performed before $s_{i+1}$.  For example, one can always make
$\lambda_{i} = 1$ (without applying $s_{i+1}$) and then apply $s_i$.
\end{proof}

\subsection{Infinite Coxeter factorizations}\label{sec:infcox}
In fact, for an infinite Coxeter element $c^\infty$, there are
matrices $X \in E_{c^\infty}$ such that Lemma \ref{lem:eps}
completely determines the $\epsilon$-signature of $X$.

\begin{proposition} \label{prop:eps}
There is $\a \in \ell^1_{>0}$ such that $X = e_{c^\infty}(\a) \in
E_{c^{\infty}}$ satisfies $\epsilon_i(X) > 0$ if and only if $s_i$
precedes $s_{i+1}$ in $c$.
\end{proposition}

\begin{proof}
By Lemma \ref{lem:eps} the ``if'' direction holds for all $\a$.  Fix
a reduced decomposition of $c$ and use it to write down a periodic
reduced expression $\i = i_1\,i_2\,i_3\,\cdots$ for $c^{\infty}$.
Let us pick $a_k = \delta^{K+k}$, where $0 < \delta < 1$ and $K>0$
is a fixed constant.  It is clear that $\a \in \ell^1_{>0}$.  We
wish to calculate $\lim_{s \to \infty} x_{i,i+s}/x_{i+1,i+s}$.

Using \cite[Section 7.2]{LP}, $x_{i,s}$ can be expressed as the
total weight of certain tableaux $T$ with shape a column of length
$s$: the entries of the tableau $T$ are strictly increasing, and the
boxes have residues $i$, $i+1$, $\ldots$ as we read from the top to
the bottom. In a box with residue $k$, we must place an integer $b$
such that $i_b = k$ (in the terminology of \cite{LP}, one would
allow any integer, but if $i_b = k$ is not satisfied, then the
weight would be 0).  If the boxes of $T$ are filled with numbers
$b_1,b_2,\ldots,b_s$, then the weight $\wt(T)$ of $T$ is $a_{b_1}
a_{b_2}\cdots a_{b_s}$.  \cite[Lemma 7.3]{LP} then states that
$x_{i,s} = \sum_T \wt(T)$.

Let $S$ be the set of tableaux enumerated by $x_{i,i+s}$ and $S'$
the set of tableaux enumerated by $x_{i+1,i+s}$.  We define a map
$\phi: S \to S'$ by removing the first box from $T$ and then
reducing all entries by $n$ to obtain $T'$.  This map is
well-defined as long as $i+1$ precedes $i$ in $\i$.  By our choice
of $\a$, we have
$$
\wt(T) = a_{b_1}\, \delta^{(s-1)n}\, \wt(T'),
$$
where $b_1$ is the entry of the first box of $T$.  Summing over the
possible choices of $b_1$, we obtain
$$
x_{i,i+s} < \left(\sum_r a_r\right)\, \delta^{(s-1)n} \, x_{i+1,i+s}
= \frac{\delta^{K+(s-1)n}}{1-\delta} \, x_{i+1,i+s}.
$$
It follows that $\epsilon_i(X) = \lim_{s \to \infty}
x_{i,i+s}/x_{i+1,i+s} = 0$, as required.
\end{proof}

%

\begin{example}
The matrix $X$ from Example \ref{ex:X} is clearly an example of such matrix for $c = s_0 s_1 s_2$.
\end{example}

There are however choices of $\a$ such that $X =e_\i(\a)$ does not
satisfy Proposition \ref{prop:eps}.
\begin{proposition}\label{prop:notincreasing}
Let $c$ be a Coxeter element which is not increasing, that is, of
the form $c = s_k s_{k+1} \cdots s_{k-1}$.  Then for each $i\in
\Z/n\Z$, there is some $X \in E_{c^\infty}$ such that $\epsilon_i(X)
> 0$.
\end{proposition}
\begin{proof}
Let $\Gamma = (\gamma_1,\gamma_2)$ be the (two-part) set composition
corresponding to $c$. The non-increasing condition implies that
$|\gamma_1| > 1$.  Thus there is a set composition $\Gamma' =
(\gamma'_1,\gamma'_2,\gamma'_3)$ of $[n]$ refining $\Gamma$,
satisfying $\gamma'_3 = \gamma_2$ and $(i+1) \notin \gamma'_1$.  The
claim then follows from Theorem \ref{thm:TPlimit} and Proposition
\ref{prop:epsmin}.
%
%
\end{proof}

We now use Proposition \ref{prop:eps} to partition $E_{c^{\infty}}$
into two disjoint parts: $E_{c^{\infty}} = A_{c^{\infty}} \sqcup
B_{c^{\infty}}$.  Here $A_{c^{\infty}}$ contains the set of matrices
with $\epsilon$-signature given by by Lemma \ref{lem:eps}, and
$B_{c^{\infty}}$ is the rest of the image.  By Proposition
\ref{prop:eps}, $A_{c^\infty}$ is non-empty.  By Proposition
\ref{prop:notincreasing}, $B_{c^\infty}$ is also non-empty whenever
$c$ is not increasing.

We can now prove part of Conjecture \ref{conj:injective} for
infinite Coxeter elements.

\begin{proposition}\label{prop:ainj}
The map $e_{c^{\infty}}$ is injective when restricted to
$e_{c^\infty}^{-1}(A_{c^{\infty}}) \subset \ell^1_{>0}$.
\end{proposition}

\begin{proof}
Chose a reduced expression for $c$ and let $s_i$ be the first
generator in this expression. Then for any $X \in A_{c^{\infty}}$ we
have $\epsilon_i(X) >0$ and $\epsilon_{i-1}(X) = 0$. We know that $X
= e_i(a) Y$ for $Y \in E_{c'^{\infty}}$, where $c' = s_i c s_i$.  By
Lemma \ref{lem:eps} we have $\epsilon_{i-1}(Y) > 0$.  By Lemma
\ref{lem:epc}(3) we must then have $\epsilon_i(Y) = 0$.  This means
that $a = \epsilon_i(X)$, and thus the factor $e_i(a)$ of $X$ is
unique.  Furthermore, one has $Y \in A_{c'^{\infty}}$, and we may
proceed inductively to obtain all the parameters of $X$.
\end{proof}

We will discuss the topic of injectivity further in Section \ref{sec:problems}.

\begin{remark}\label{rem:Rnotbij}
Propositions \ref{prop:eps} and \ref{prop:notincreasing} show that
$R_\i^\j$ is not in general surjective.  In fact one can find $\i
\to \j$ so that $E_\i \subsetneq E_\j$: as in the proof of
Proposition \ref{prop:notincreasing}, one may find $\i$ so that $\i
\to \j = c^\infty$ and $E_\i \subset B_{c^\infty}$.

Now consider the braid limit $\i = 1(012)^\infty \to (012)^\infty =
\j$ of Example \ref{ex:1012}.  We claim that $R_\i^\j$ is not
injective: we have $e_\i(\a)= e_\i(\a')$ for $\a \neq \a'$.  But
$(012)^\infty$ is injective since $E_{(012)^\infty} =
A_{(012)^\infty}$.  Thus $R_\i^\j(\a) = R_\i^\j(\a')$.
\end{remark}


\subsection{The case $n = 3$}

\begin{prop}\label{prop:cinjective3}
All infinite Coxeter elements are injective for $n = 3$.
\end{prop}
\begin{proof}
By Propositions \ref{prop:eps} and \ref{prop:ainj}, this is the case
for the increasing Coxeter elements $c = 012,120,201$.  Let $X
\mapsto X^{-c}$ denote the ``$c$-inverse'' involution of \cite{LP},
which in $\Omega$ acts by
$$
e_{i_1}(a_1) e_{i_2}(a_2) \cdots \mapsto \cdots e_{i_2}(a_2)
e_{i_1}(a_1).
$$
Now consider the limits $\mu_j(X) = \lim_{i \to -\infty}
x_{i,j+1}/x_{i,j}$, applied to $\cdots e_{i_2}(a_2) e_{i_1}(a_1)$.
One can check that if $X \in E_\i$ then $\mu_j(X^{-c}) >0$ if $j$
precedes $j-1$ in $X$.  The same arguments as for $\epsilon$'s now
shows that decreasing infinite Coxeter elements are injective.
\end{proof}

\section{ASW factorizations} \label{sec:asw}
In this section, we construct for each $X \in \Omega$ a
distinguished factorization $X = e_\i(\a)$, decomposing $\Omega$ as
a disjoint union of subsets which we call ASW-cells.

We will make use of the following well-known fact (see for example
\cite[(15.53)]{FH}).

\begin{lemma} [Three-term Pl\"ucker relations] \label{lem:plucker}
If $\Delta_I$ denotes the minor of a matrix $X$ with row set $I$ and
initial column set, then the following identities are true for any
set $K$ and distinct $i < j < k < l$ not in $K$:
\begin{enumerate}
\item
$\Delta_{K \cup \{i,k\}} \Delta_{K \cup \{j,l\}} = \Delta_{K \cup
\{i,j\}} \Delta_{K \cup \{k,l\}} + \Delta_{K \cup \{i,l\}} \Delta_{K
\cup \{j,k\}}$;
\item
$\Delta_{K \cup \{i,k\}} \Delta_{K \cup \{j\}} = \Delta_{K \cup
\{i,j\}} \Delta_{K \cup \{k\}}+\Delta_{K \cup \{i\}} \Delta_{K \cup
\{j,k\}}$.
\end{enumerate}
\end{lemma}

\subsection{$q$-ASW}
In this section, we assume the reader is familiar with the ASW
(Aissen-Schoenberg-Whitney) factorization from \cite[Section 5]{LP}.

Let $X \in U_{> 0}$ and $q \geq 1$ be an integer. We define the
matrix $M_q(X)$ as follows: $m_{q,i,i} = 1$,
$$m_{q,i,j} = (-1)^{j-i} \lim_{l \to \infty} \frac{\Delta_{\{i,
\ldots, \hat j, \ldots, i+q\},\{l, \ldots,
l+q-1\}}(X)}{\Delta_{\{i+1, \ldots, i+q\},\{l, \ldots,
l+q-1\}}(X)}$$ for $0 < j-i \leq q$ and $m_{q,i,j} = 0$ in all other
cases. Here $\hat j$ denotes omission of the index $j$. By
\cite[Theorem 10.6]{LP} these limits exist and are finite.  Note that $M_1(X) =
M(-\epsilon_1(X),-\epsilon_2(X),\ldots,-\epsilon_n(X))$.

\begin{example} \label{ex:qX}
Take the matrix of Example \ref{ex:X}. We have $$m_{2,1,3} = (-1)^{3-1} \lim_{l \to \infty} \frac{\Delta_{\{1,2\}, \{l.l+1\}} (X)} {\Delta_{\{2,3\}, \{l.l+1\}} (X)} =$$ $$= \lim_{l \to \infty} \frac{a^{\eta(2,l) + \eta(2,l)} \left( \prod_{t=1}^{l-1} (1-a^t)^{-2} -  (1-a^{l-1})^{-1} (1 - a^l)^{-1} \prod_{t=1}^{l-2} (1-a^t)^{-2} \right)} {a^{\eta(2,l) + \eta(2,l)} \left( a^{2-l} \prod_{t=1}^{l-2} (1-a^t)^{-2} -  a^{3-l} (1-a^{l-2})^{-1} (1 - a^{l-1})^{-1} \prod_{t=1}^{l-3} (1-a^t)^{-2} \right)} = $$ $$ = \lim_{l \to \infty} a^{l-2} \frac{(1-a^{l-2})^{-2} (1-a^{l-1})^{-2}  - (1-a^{l-2})^{-2} (1-a^{l-1})^{-1} (1-a^{l})^{-1}} {(1-a^{l-2})^{-2} - a (1-a^{l-2})^{-1} (1-a^{l-1})^{-1}} =$$ $$= \lim_{l \to \infty} \frac{a^{2l-3}}{(1-a^{l-1}) (1-a^{l})}=0.$$ In this manner one computes 
$$ M_2(X) = 
\left(
\begin{array}{c|ccc|ccc|c}
\ddots & \vdots &  \vdots &  \vdots &  \vdots & \vdots&\vdots \\
\hline
\dots& 1 & -1 & 0 & 0 &0 & 0 & \dots \\
\dots&0&1 & -1 &1& 0 & 0 & \dots \\
\dots& 0 &0&1 & -1-a &a & 0 &\dots \\
\hline
\dots& 0 & 0 & 0 & 1 & -1& 0 & \dots \\
\dots& 0 & 0 & 0 &0&1 & -1 & \dots \\
\dots& 0 & 0 & 0 &0&0&1 & \dots \\
\hline & \vdots & \vdots &  \vdots &  \vdots  & \vdots & \vdots &
\ddots\end{array} \right)
.$$
\end{example}

\begin{lemma} \label{lem:qaswtp}
Suppose $X \in U_{> 0}$ is totally positive.  Then $M_q(X) X \in
U_{> 0}$.
\end{lemma}

\begin{proof}
Denote $Y = M_q(X) X$.  It is enough to show $Y \in U_{\geq 0}$
since $M_q(X)$ is finitely supported, and the product of a finitely
supported TNN matrix and a TNN matrix which is not totally positive,
is never totally positive (see \cite[Theorem 5.5 and Theorem
5.7]{LP}).

We claim that for any $j_1< \ldots < j_{k+1}$
$$\Delta_{\{i, \ldots, i+k\}, \{j_1, \ldots, j_{k+1}\}}(Y) = \lim_{l
\to \infty} \frac{\Delta_{\{i, \ldots, i+k+q\}, \{j_1, \ldots,
j_{k+1}, l, \ldots, l+q-1\}}(X)}{\Delta_{\{i+k+1, \ldots, i+k+q\},
\{l, \ldots, l+q-1\}}(X)}.$$ The fact that this limit exists and is
finite is part of the claim to become evident later.  For a fixed
$l$, let
$$m^{(l)}_{q,i,j} = (-1)^{j-i} \frac{\Delta_{\{i, \ldots, \hat j,
\ldots, i+q\},\{l, \ldots, l+q-1\}}(X)}{\Delta_{\{i+1, \ldots,
i+q\},\{l, \ldots, l+q-1\}}(X)}$$ and let $M_q^{(l)}$ be the matrix
filled with entries $m^{(l)}_{q,i,j}$ in rows $i$ through $i+k$, and
coinciding with identity elsewhere. Here we assume $l$ is large
enough so that all needed $m^{(l)}_{q,i,j}$-s are well-defined. Note
that $M_q^{(l)}$ is not infinitely periodic.  We have by definition
$$\lim_{l \to \infty} m^{(l)}_{q,i,j} = m_{q,i,j}.$$

We claim that the entries of $Y^{(l)} = M_q^{(l)} X$ in rows $i$
through $i+k$ and in columns $l$ through $l+q-1$ are zero.  Assume
for now that this is known.  We further observe that multiplication
by $M_q^{(l)}$ does not change the determinant $\Delta_{\{i, \ldots,
i+k+q\}, \{j_1, \ldots, j_{k+1}, l, \ldots, l+q-1\}}(X)$, since
$M_q^{(l)}(X)$ is supported only within the first $q$ diagonals.
Therefore
$$\Delta_{\{i, \ldots, i+k\}, \{j_1, \ldots, j_{k+1}\}}(Y^{(l)}) =
\frac{\Delta_{\{i, \ldots, i+k+q\}, \{j_1, \ldots, j_{k+1}, l,
\ldots, l+q-1\}}(X)}{\Delta_{\{i+k+1, \ldots, i+k+q\}, \{l, \ldots,
l+q-1\}}(X)}.$$ Taking the limit $l \to \infty$ we obtain the needed
statement: the limit on the left exists and equals the limit on the
right, which is thus finite. Now we observe that every row-solid
minor of $Y$ is a limit of a positive value, and thus is
nonnegative. By \cite[Lemma 2.3]{LP} this implies that $Y \in
U_{\geq 0}$.

It remains to argue that the mentioned entries of $Y^{(l)}$ are
zero. We argue that $y^{(l)}_{i,l} = 0$: this follows from the
relation $$x_{i,l} \Delta_{\{i+1, \ldots, i+q\}, \{l, \ldots,
l+q-1\}}(X) - x_{i+1,l} \Delta_{\{i, i+2 \ldots, i+q\}, \{l, \ldots,
l+q-1\}}(X) +$$ $$\ldots \pm x_{i+q,l} \Delta_{\{i, \ldots, i+q-1\},
\{l, \ldots, l+q-1\}}(X) = 0$$ obtained by expanding the determinant
of the submatrix $X_{\{i,i+1,\ldots,i+q\},\{l,l,l+1,\ldots,l+q-1\}}$
(note that column $l$ is repeated) along the first column.  The same
argument works for any choice of a row and a column in the specified
range.
\end{proof}

Note that $M_1(X)$ is the ASW factorization applied to $X$, i.e.
$(M_1(X))^{-1}$ is exactly the curl factored out from $X$ by ASW. It
is then natural to expect $M_q(X)$ to have some maximality property
similar to that of ASW factorization, cf. \cite[Lemma 5.4]{LP}. This
is made precise by the following lemma.

\begin{lemma} \label{lem:qaswmax}
Among all matrices $M$ supported on first $q$ diagonals such that
$MX \in U_{\geq 0}$ the matrix $M_q(X)$ has minimal (most negative)
entries directly above the diagonal.
\end{lemma}

\begin{proof}
Consider the ratio $$\frac{\Delta_{\{i, i+2, \ldots, i+q\},\{l,
\ldots, l+q-1\}}(X)}{\Delta_{\{i+1, \ldots, i+q\},\{l, \ldots,
l+q-1\}}(X)}.$$ When multiplying by $M$ on the left, only the entry
$m_{i,i+1}$ will affect this ratio, since the next $q-1$ entries
$m_{i,j}$ in that row do not influence either determinant, while
beyond that $M$ is zero. By \cite[Lemma 10.5]{LP}, the limits
defining $M_q(X)$ are monotonic.  Thus $m_{q,i,i+1}$ is the minimal
value such that in $MX$ the above ratio remains nonnegative for all
values of $l$.
\end{proof}


\begin{lemma} \label{lem:qaswq}
The matrix $M_q(X)X$ is equal to the matrix obtained by $q$
iterations of ASW factorization on $X$.
\end{lemma}

\begin{proof}
We show that $M_1(M_{q-1}(X) X) M_{q-1}(X)  = M_{q}(X)$, and the
result will follow by induction on $q$.

To simplify the notation we denote $\Delta_I = \Delta_{I,\{l,
\ldots, l+|I|-1\}}(X)$. Let $$a^{(l)}_i = m^{(l)}_{q,i,i+1} -
m^{(l)}_{q-1,i,i+1} = -
\frac{\Delta_{\{i,i+2,\ldots,i+q\}}}{\Delta_{\{i+1,\ldots,i+q\}}} +
\frac{\Delta_{\{i,i+2,\ldots,i+q-1\}}}{\Delta_{\{i+1,\ldots,i+q-1\}}}
= - \frac{\Delta_{\{i,\ldots,i+q-1\}}
\Delta_{\{i+2,\ldots,i+q\}}}{\Delta_{\{i+1,\ldots,i+q-1\}}
\Delta_{\{i+1,\ldots,i+q\}}}$$ which is evidently negative. Let
$a_i = \lim_{l \to \infty} a^{(l)}_i$. We claim that $M(a_1, \ldots,
a_n) M_{q-1}(X) = M_q(X)$, the entries directly above the diagonal
coincide by definition of $a_i$-s. For the rest of the entries, we
perform the following calculation, using Lemma \ref{lem:plucker}:
$$\left(-
\frac{\Delta_{\{i,i+2,\ldots,i+q\}}}{\Delta_{\{i+1,\ldots,i+q\}}} +
\frac{\Delta_{\{i,i+2,\ldots,i+q-1\}}}{\Delta_{\{i+1,\ldots,i+q-1\}}}\right)
\frac{\Delta_{\{i+1, \ldots, \hat j, \ldots, i+q\}}}{\Delta_{\{i+2,
\ldots, i+q\}}} - \frac{\Delta_{\{i, \ldots, \hat j, \ldots,
i+q-1\}}}{\Delta_{\{i+1, \ldots, i+q-1\}}} =$$ $$= -
\frac{\Delta_{\{i,i+2,\ldots,i+q\}}\Delta_{\{i+1, \ldots, \hat j,
\ldots, i+q\}}}{\Delta_{\{i+1,\ldots,i+q\}}\Delta_{\{i+2, \ldots,
i+q\}}} + \frac{\Delta_{\{i, i+2, \ldots, \hat j, \ldots,
i+q\}}}{\Delta_{\{i+2, \ldots, i+q\}}} = - \frac{\Delta_{\{i,\ldots,
\hat j, \ldots, i+q\}}}{\Delta_{\{i+1,\ldots,i+q\}}},$$ which means
$a^{(l)}_i m^{(l)}_{q-1,i+1,j} + m^{(l)}_{q-1,i,j} =
m^{(l)}_{q,i,j}$, and passing to a limit $a_i m_{q-1,i+1,j} +
m_{q-1,i,j} = m_{q,i,j}$ as desired.

Next, we claim that $M(a_1, \ldots, a_n) = M_1(M_{q-1}(X) X)$.
Indeed, by Lemma \ref{lem:qaswtp} and the calculation above we know
that the curl $M(a_1, \ldots, a_n)^{-1}$ can be factored out from
$M_{q-1}(X) X$ so that the result is totally nonnegative. On the
other hand, by Lemma \ref{lem:qaswmax} we see that each parameter
$a_i$ is minimal possible for which such factorization could exist.
This means that $M(a_1, \ldots, a_n)^{-1}$ is exactly the result of
ASW factorization applied to $M_{q-1}(X) X$.
\end{proof}

\begin{example}
The matrix $M_2(X)$ obtained in Example \ref{ex:qX} factors as
$$ M_2(X) = 
\left(
\begin{array}{c|ccc|ccc|c}
\ddots & \vdots &  \vdots &  \vdots &  \vdots & \vdots&\vdots \\
\hline
\dots& 1 & 0 & 0 & 0 &0 & 0 & \dots \\
\dots&0&1 & -1 &0& 0 & 0 & \dots \\
\dots& 0 &0&1 & -a &0 & 0 &\dots \\
\hline
\dots& 0 & 0 & 0 & 1 & 0& 0 & \dots \\
\dots& 0 & 0 & 0 &0&1 & -1 & \dots \\
\dots& 0 & 0 & 0 &0&0&1 & \dots \\
\hline & \vdots & \vdots &  \vdots &  \vdots  & \vdots & \vdots &
\ddots\end{array} \right)
\left(
\begin{array}{c|ccc|ccc|c}
\ddots & \vdots &  \vdots &  \vdots &  \vdots & \vdots&\vdots \\
\hline
\dots& 1 & -1 & 0 & 0 &0 & 0 & \dots \\
\dots&0&1 & 0 &0& 0 & 0 & \dots \\
\dots& 0 &0&1 & -1 &0 & 0 &\dots \\
\hline
\dots& 0 & 0 & 0 & 1 & -1& 0 & \dots \\
\dots& 0 & 0 & 0 &0&1 & 0 & \dots \\
\dots& 0 & 0 & 0 &0&0&1 & \dots \\
\hline & \vdots & \vdots &  \vdots &  \vdots  & \vdots & \vdots &
\ddots\end{array} \right)
$$
The two factors are exactly the results of the usual ASW factorization applied to $X$ twice.
\end{example}

One way to interpret Lemmata \ref{lem:qaswmax} and \ref{lem:qaswq}
is to say that the local maximality of ASW factorization translates
to global maximality: the maximal way to factor out a product of $q$
curls is to greedily factor out a maximal curl at each step. We use
this to derive the following property of ASW factorization on
$\Omega$.

\begin{thm} \label{thm:aswomega}
Let $X \in \Omega$ and let $N_1, N_2, \ldots$ be degenerate curls
obtained by repeated application of ASW factorization to $X$. Then
$X = \prod_{i \geq 1} N_i$.
\end{thm}

\begin{proof}
First, it is clear that $\prod_{i \geq 1} N_i$ exists and is $\leq
X$ entry-wise. It suffices then to show that for any initial part
$X^{(k)} = \prod_{j = 1}^k e_{i_j}(a_j)$ of $X$ we have $\prod_{i
\geq 1} N_i \geq X^{(k)}$ entry-wise.  In fact, it is enough to
check this latter inequality for only the entries directly above the
diagonal, since $(\prod_{i\geq 1}N_i)^{-1} \, X \in U_{\geq 0}$, and
a TNN matrix which has 0's directly above the diagonal is the
identity matrix.

This however follows from Lemma \ref{lem:qaswmax}: since $X^{(k)}=
\prod_{j = 1}^k e_{i_j}(b_j)$ is a product of $k$ curls, the product
$\prod_{i = 1}^k N_i$ has greater entries just above the diagonal.
\end{proof}

\begin{lem}\label{lem:ASWreduced}
Let $X \in \Omega$.  If the ASW factorization is $X = e_{\i}(\a)$,
then $\i$ is necessarily an infinite reduced word.
\end{lem}
\begin{proof}
This is a special case of Lemma \ref{lem:greedyreduced}.
\end{proof}

\subsection{ASW factorization for finitely supported matrices}
Let $X \in U_{\geq 0}$ be finitely supported matrix such that $X^c$ is entire. One can define a
finite version of ASW factorization as follows. For a given $i$, let
$j$ be maximal such that $x_{i+1,j} \neq 0$. Define $\epsilon_i(X) =
\frac{a_{i,j}}{a_{i+1,j}}$.  It is clear that not all the
$\epsilon_i$ can be simultaneously $0$, otherwise by \cite[Theorem 5.5]{LP} $X$ would be a product of non-degenerate whirls, and $X^c$ would not be entire.

Just as for the infinitely supported case, we will call the
factorization in the following Proposition {\it ASW factorization}.

\begin{prop} Let $U_{\geq 0}$ be finitely supported.
\begin{enumerate}
\item
$M(-\epsilon_1, \ldots, -\epsilon_n) X \in U_{\geq 0}$, or in other
words the degenerate curl $N(\epsilon_1, \ldots, \epsilon_n)$ can be
factored out from $X$;
\item
for any other curl $N(a_1, \ldots, a_n)$ that can be factored out
from $X$ we have $a_i \leq \epsilon_i$.
\end{enumerate}
\end{prop}

\begin{proof}
We follow the strategy in the proof of \cite[Theorem 2.6]{LP}.  It
is clear that the set of non-zero entries in $X$ has some NE
corners, which implies that some of $\epsilon_i$-s are zero. We can
group all $i$-s with non-zero $\epsilon_i$-s into sets that share
common $j$ in the definition of $\epsilon_i$ above.  This divides
$\Z/n\Z$ into a number of cyclic intervals.  It was essentially
shown in the proof of \cite[Theorem 2.6]{LP} that the Chevalley
generators corresponding to $i$-s initial (smallest) in those
intervals can be factored from $X$ with parameters equal to the
corresponding $\epsilon_i$-s. One can then iterate this argument to
factor Chevalley generators corresponding to non-initial elements of
the intervals. The resulting collection of Chevalley generators
factored has product equal to $N(\epsilon_1, \ldots, \epsilon_n)$,
as desired. The second statement is clear since if $a_i
> \epsilon_i$ for some $i \in \Z/n\Z$, the product $M(-a_1, \ldots, -a_n) X$ would have a
negative entry in row $i$.
\end{proof}

The following result is the finitely supported analogue of Theorem
\ref{thm:aswomega}.

\begin{proposition} \label{prop:aswomegaf}
Assume $X \in U_{\geq 0}$ is a finite product of Chevalley
generators. Then repeated application of ASW factorization results
in an expression $X = N_1 \dotsc N_l$ of $X$ as a product of
degenerate curls.
\end{proposition}

\begin{proof}
If there was a non-trivial remainder in the above ASW factorization,
by \cite[Theorem 5.5]{LP} this remainder would be a finite
product of non-degenerate whirls. Then $X^{-c}$ would not be entire,
which would be a contradiction.
\end{proof}

\subsection{Uniqueness of $\Omega$ factors}\label{sec:Omegaunique}
From the definition, a matrix $X$ is entire if it is either finitely
supported, or $\lim_{j \to \infty} x_{i,j}/x_{i+n,j} = 0$ for each
$i$.

\begin{lemma} \label{lem:epsentire}
Let $X, Y \in U_{\geq 0}$ be such that $X$ is infinitely supported
and $Y$ is entire.  Then for any $i$ we have $\epsilon_i(XY) =
\epsilon_i(X)$.
\end{lemma}

\begin{proof}
By \cite[Lemmata 5.3 and 5.4]{LP}, one has $\epsilon_i(XY) \geq
\epsilon_i(X)$ since if $N(a_1,a_2,\ldots,a_n)$ can be factored from
$X$ then it can also be factored from $XY$. Thus, it suffices to
show that $\epsilon_i(XY) \leq \epsilon_i(X)$.

For convenience of notation let $X =
(a_{i,j})_{i,j=-\infty}^\infty$, $Y =
(b_{i,j})_{i,j=-\infty}^\infty$ and $XY =
(c_{i,j})_{i,j=\infty}^\infty$.  Let $\epsilon = \epsilon_i(X)$. For
a given $\delta > 0$, let us pick $N$ such that
$\frac{a_{i,k}}{a_{i+1,k}} < \epsilon + \delta$ for $k>N$.  Choose
$C > 0$ such that $C < \frac{a_{i+1,k+n}}{a_{i,k}} \delta$ for $i
\leq k \leq N$.  Such a $C$ exists since $X$ is infinitely
supported.  Now pick $j \gg N$ sufficiently large such that $b_{k,j}
\leq C b_{k+n, j}$ for $i \leq k \leq N$.  This is possible since
$Y$ is entire.  Then for $k \in [i,N]$,
$$a_{i,k} b_{k,j} \leq a_{i,k} b_{k+n,j} C < \delta a_{i+1,k+n}
b_{k+n,j}.$$ We have
$$c_{i,j} = \sum_{k=i}^{j} a_{i,k} b_{k,j} = \sum_{k=i}^{N} a_{i,k}
b_{k,j} + \sum_{k=N+1}^{j} a_{i,k} b_{k,j} $$ $$\leq \delta
(\sum_{k=i}^{N} a_{i+1,k+n} b_{k+n,j}) + (\epsilon + \delta)
(\sum_{k=N+1}^{j} a_{i+1,k} b_{k,j}) < (\epsilon + 2 \delta)
c_{i+1,j}.$$ This holds for sufficiently large $j$, and since we can
choose $\delta$ to be arbitrarily small, we conclude that
$\epsilon_i(XY) \leq \epsilon$ as desired.
\end{proof}

\begin{lemma} \label{lem:epsentire2}
Let $X, Y \in U_{\geq 0}$ be such that $X$ is finitely supported and
$Y$ is infinitely supported, and $\epsilon_i(Y) = 0$ for all $i \in
\Z/n\Z$. Then for any $i$ we have $\epsilon_i(XY) = \epsilon_i(X)$.
\end{lemma}

\begin{proof}
As in the proof of the Lemma \ref{lem:epsentire}, it suffices to
show that $\epsilon_i(XY) \leq \epsilon_i(X)$. As before, let $X =
(a_{i,j})_{i,j=-\infty}^\infty$, $Y =
(b_{i,j})_{i,j=-\infty}^\infty$ and $XY =
(c_{i,j})_{i,j=\infty}^\infty$.  Let $N$ be such that $a_{i+1,N} >
0$ but $a_{i+1,k}=0$ for $k>N$.  Thus by definition $\epsilon_i(X) =
a_{i,N}/a_{i+1,N}$.

Let $\delta > 0$ and choose $C > 0$ such that $a_{i,k} C < a_{i+1,
k+1} \delta$ for $i \leq k \leq N-1$.  Next, choose $j$ large enough
such that $\frac{b_{k,j}}{b_{k+1,j}}<C$ for $i \leq k \leq N-1$.
This can be done since $Y$ is infinitely supported and
$\epsilon_r(Y)=0$ for any $r$.  Then we can write $$c_{i,j} =
\sum_{k=i}^{j} a_{i,k} b_{k,j} = \sum_{k=i}^{N} a_{i,k} b_{k,j} =
\sum_{k=i}^{N-1} a_{i,k} b_{k,j} + a_{i,N} b_{N,j} <
\sum_{k=i}^{N-1} a_{i,k} b_{k+1,j} C + \epsilon a_{i+1,N} b_{N,j}
$$ $$< \sum_{k=i}^{N-1} a_{i+1,k+1} b_{k+1,j} \delta + \epsilon
a_{i+1,N} b_{N,j} < \delta c_{i+1,j} + \epsilon c_{i+1,j} = (\delta
+ \epsilon) c_{i+1,j}.$$ This holds for sufficiently large $j$, and
since we can choose $\delta$ to be arbitrarily small, we conclude
that $\epsilon_i(XY) \leq \epsilon_i(X)$ as desired.
\end{proof}

Define $\L_r$ to be the right limit-semigroup generated by Chevalley
generators (see \cite[Section 8.3]{LP}). In other words, $\L_r
\subset U_{\geq 0}$ is the smallest subset of $U_{\geq 0}$ which
contains Chevalley generators $\{e_i(a) \mid a \geq 0\}$, and is
closed under products, and right-infinite products.  For example,
$\L_r$ contains matrices of the form $X = \prod_{i=1}^\infty
X^{(i)}$ where each $X^{(i)}$ lies in $\Omega$.

The following result proves that the factorization of \cite[Theorem
8.8]{LP} is unique.  Recall the definition of $\mu_i(X)$ from
\cite{LP}, or the proof of Proposition \ref{prop:cinjective3}.

\begin{theorem}\label{thm:uniqueness} \
\begin{enumerate}
\item
Let $X \in U_{\geq 0}$ be entire.  There is a unique factorization
$X = YZ$, where $Y \in \L_r$ and $Z$ satisfies $\epsilon_i(Z) = 0$
for each $i \in \Z/n\Z$.
\item
Let $X \in U_{\geq 0}$ be such that $X^{-c}$ is entire.  There is a
unique factorization $X = Z'Y'$, where $Y' \in \L_l = (\L_r)^{-c}$
and $Z'$ satisfies $\mu_i(Z') = 0$ for each $i \in \Z/n\Z$.
\end{enumerate}
\end{theorem}

\begin{proof}
We prove (1), as (2) is similar.  By Lemmata \ref{lem:epsentire} and
\ref{lem:epsentire2}, applying (possibly finite) ASW factorization
to $Y$ produces the same result as ASW applied to $X$. By Theorem
\ref{thm:aswomega} and Proposition \ref{prop:aswomegaf}, this allows
one to extract the first $\Omega$-factor of $Y$ (or $Y$ itself, if
$Y\in U_{\geq 0}^\pol$).  Since the first factor in $Y$ is
determined by $X$, by transfinite induction we conclude that all
factors are.
\end{proof}


\begin{theorem}\label{thmcor:uniqueness} \
\begin{enumerate}
\item
when a Chevalley generator is factored from an element of $\Omega$
giving a totally nonnegative matrix, the resulting matrix also lies
in $\Omega$;
\item
each element of $\L_r$ has a unique factorization into factors which
lie in $\Omega$, with possibly one factor which is a finite product
of Chevalley generators;
\item
if $X \in \Omega$ and $X = YZ$ where $Y \in \Omega$ and $Z \in
U_{\geq 0}$, then $Z$ is the identity matrix.
\end{enumerate}
\end{theorem}

\begin{proof}
\begin{enumerate}
\item
Let $X \in \Omega$ and assume $X = e_i(t) X'$. Apply Theorem
\ref{thm:uniqueness} to write $X' = Y'Z'$. Then $X = e_i(t) Y' Z'$
and by uniqueness in Theorem \ref{thm:uniqueness} we conclude that
$X = e_i(t) Y'$ and $Z' = I$ is the identity matrix. Thus $X' = Y'
\in \L_r$. If $X'$ is a finite product of Chevalley generators then
$X$ is finitely supported, contradicting $X \in \Omega$. Otherwise
$X' = X'' X'''$ where $X'' \in \Omega$, $X''' \in \L_r$, possibly
$X'''=I$. By Theorem \ref{thm:aswomega} and Lemma
\ref{lem:epsentire}, ASW factorization applied repeatedly to $X$ and
$e_i(t)X''$ produces the same result, and this result is equal to
both $X$ and $e_i(t)X''$. Thus $X''' =I$ and $X' \in \Omega$, as
desired.
\item
If $X$ is a finite product of Chevalley generators, the statement is
clear. Otherwise assume $X = YZ$ with $X, Z \in \L_r$, $Y \in
\Omega$. By Lemma \ref{lem:epsentire} applying ASW repeatedly to $X$
and $Y$ produces the same result, which by Theorem
\ref{thm:aswomega} is equal to $Y$. Thus the factor $Y$ of $X$ can
be recovered uniquely. By transfinite induction we conclude that
every $\Omega$-factor in $X \in \L_r$ is unique.
\item
This follows immediately from (2).
\end{enumerate}
\end{proof}

\subsection{Proof of Proposition \ref{prop:nonreduced}}
\label{sec:proofnonreduced} It is clear that $\cup_\i E_\i$ contains
$\Omega \cup U_{\geq 0}^\pol$.  Let $X = e_\i(\a)$, where $\i$ is
possibly not reduced.  We may assume that $\i$ is infinite.  Let us
apply ASW factorization to $X$, to obtain $X = YZ$, where $Y \in
\Omega \cup U_{\geq 0}^\pol$ and $Z \in U_{\geq 0}$.  If $Z$ is not
the identity matrix, then for some $i \in \Z/n\Z$ we have $s =
z_{i,i+1} > 0$.  Let $s' = x_{i,i+1} = \sum_{i_r = i} a_r$, which we
know is greater than or equal to $s$.  We can find some $k$
sufficiently large that $\sum_{i_r = i \mid r < k} a_r > s' - s$.
The matrix $M = e_{i_k}(-a_k) \cdots e_{i_1}(-a_1)$ is supported on
the first $k$ diagonals and $MX$ is TNN.  Thus by Lemma
\ref{lem:qaswmax} and Lemma \ref{lem:qaswq}, the matrix $M_k(X)$ has
smaller entries on above the diagonal than $M$.  It follows that
$y_{i,i+1} \geq s' - s$.  But this contradicts $s = z_{i,i+1}$.  We
conclude that $Z = I$, and so $X \in \Omega \cup U_{\geq 0}^\pol$.

\subsection{First proof of Theorem \ref{thm:TPlimit}}
\label{sec:proofTPlimit} Let $\a' =  R_\i^\j(\a)$.  It is clear from
the definition of braid limit that $e_{j_1}(a'_1) \cdots
e_{j_k}(a'_k)$ can be factored out of $e_\i(\a)$ on the left.  Since
limits of TNN matrices are TNN (\cite[Lemma 2.4]{LP}), we deduce
that $e_\i(\a) = e_\j(R_\i^\j(\a))\,Z$ where $Z \in U_{\geq 0}$.  By
Theorem \ref{thmcor:uniqueness}(3), $Z$ is the identity matrix.

\subsection{ASW cells}\label{sec:ASWcells}

Assume $X \in U_{\geq 0}$ and let $\sigma(X) = (\sigma_1(X), \ldots,
\sigma_n(X))$ be its $\epsilon$-signature, so that $\sigma_i(X) =
{\rm sign}(\epsilon_i(X)) \in \{0, +\}$.

\begin{lemma} \label{lem:asww}
Let $X \in U_{\geq 0}$.  Let $X = N Y$ be a single application of
ASW factorization to $X$, so that $N =
N(\epsilon_1(X),\ldots,\epsilon_n(X))$ is a (possibly degenerate)
curl. Then
$$
\{i \mid \sigma_i(Y) = +\} \subseteq \{ i \mid \sigma_{i+1}(X) = +
\}.
$$
%
\end{lemma}

\begin{proof}
First suppose $X$ is infinitely supported.  The statement is
trivially true if $\sigma(X)$ consists of all $+$'s.  If
$\sigma_{i+1}(X)=0$ then factorizing $N$ our of $X$ does not change
the $i+1$-st row of $X$. The $i$-th row may or may not change
depending on $\sigma_i(X)$. However, in either case the ratio of new
$i$-th row to the old, and thus also to the new $i+1$-st row becomes
$0$ at the limit, that is $\sigma_i(Y)=0$.

Now suppose $X$ is finitely supported.  We look at the north-east
boundary of the non-zero entries of $X$. If $\epsilon_{i+1}(X)=0$
then the $i+1$-st row is the same in $X$ and $Y$.  Even if in $X$
the last non-zero entries in the $i$-th and $i+1$-st rows are in the
same column, after factoring $N$ out it is not true anymore, and
thus $\epsilon_{i}(Y)=0$.
\end{proof}

A degenerate curl $N$ is a finite product of Chevalley generators.
Define $v(N) \in \aW$ by requiring that $N \in E_{v(N)}$.  We can
describe the possible $v \in \aW$ that result as follows.  A word
$i_1i_2 \cdots i_k$ in the alphabet $\Z/n\Z$ is {\it cyclically
increasing} if no letter is repeated, and whenever $i$ and $i+1$
(taken modulo $n$) are both present, $i$ occurs before $i+1$.  An
element $v \in \aW$ is cyclically increasing if some (equivalently,
every) reduced word for $v$ is cyclically increasing.  Cyclically
increasing elements are exactly the ones occurring as $v(N)$ for
some degenerate curl $N$.
 (The reversed notion of cyclically decreasing elements is studied in
\cite{Lam}.)  Note that a cyclically increasing permutation $v$ is
completely determined by which simple generators $s_i$ occur, and
for $v = v(N(a_1,a_2,\ldots,a_n))$ these are exactly the indices $i$
such that $a_i > 0$.

For $v =s_{i_1}s_{i_2}\cdots s_{i_\ell} \in \aW$ and an integer $k
\geq 0$, let us define $v^{(k)} =s_{i_1-k}s_{i_2-k} \cdots s_{i_\ell
- k}$, obtained by rotating the indices.  Note that $v^{(k)}$ does
not depend on the reduced word of $v$ chosen.  Define the infinite
``falling power'' $v^{[\infty]} = \prod_{k \geq 0} v^{(k)}$,
considered as a possibly non-reduced infinite word, assuming that a
reduced word for $v$ has been fixed.

By Theorem \ref{thm:aswomega}, applying ASW factorization to an
element $X \in \Omega$ repeatedly leads us to a factorization $X =
\prod_{j \geq 1} N_j$ into degenerate curls.  By Lemma
\ref{lem:asww}, we have $\ell(v(N_1)) \geq \ell(v(N_2)) \geq
\cdots$, and at some point the lengths must stabilize: there is some
minimal $l$ such that $\ell(v(N_{l+k})) = \ell(v(N_l))$ for every $k
\geq 0$.  By Lemma \ref{lem:asww} again, we have in fact
$v(N_{l+k})=v(N_{l})^{(k)}$. Thus $X \in E_{wv^{[\infty]}}$, where
$v = v(N_{l+1})$ and $w = \prod_{j=1}^{l} v(N_j)$.

Whenever a pair $(w,v) \in \aW \times \aW$ occurs in the above
manner for some $X \in \Omega$, we say $w$ and $v$ are {\it
{compatible}}, and write $X \in A(w,v)$.  Then ASW factorization
decomposes $\Omega$ into a disjoint union
$$
\Omega = \bigsqcup_{(w,v)} A(w,v)
$$
over the set of compatible pairs.  We call the sets $A(w,v)$
ASW-cells (even though they may have complicated topology).  For the
rest of the section, our aim is to describe the set of compatible
pairs.

We first introduce a version of ASW factorization at the level of
affine permutations.  We shall require (strong) Bruhat order
(\cite{Hum}) on the affine symmetric group in the following, and
shall denote it by $w <_s v$.

\begin{proposition} \label{prop:finasw} \
\begin{enumerate}
\item
Let $w \in \aW$.  Then there is a cyclically increasing $v \in \aW$
such that $v \leq w$ (in weak order), and for any other cyclically
increasing $v' \leq w$ we have $v' <_s v$.  The same result holds
for $\i \in \bW$.  We call the factorization $w = v u$ (resp. $\i =
v \j$) the {\it (combinatorial) ASW factorization} of $w$ (resp.
$\i$).
\item
If $w = v_1 \ldots v_k$ is the result of repeated ASW factorization
of $w \in \aW$, then $v_{i+1} \leq_s v_i^{(1)}$, for $i = 1, \ldots,
k-1$.
\end{enumerate}
\end{proposition}

\begin{proof}
We prove (1).  Suppose first that $w \in \aW$.  Choose any $X \in
E_w$. Let $N$ be the curl factored from $X$ by ASW factorization,
and let $v = v(N)$. We claim that $v$ is the required cyclically
increasing element. First, by Theorem \ref{thm:fin}(2) we know that
$v \leq w$. Suppose $v'$ is another cyclically increasing element
satisfying $v' \leq w$, so that $v'$ is not less than $v$ in Bruhat
order.  Then there must be a simple generator $s_i$ in $v'$ that is
not contained in $v$. Since $v$ contains all $s_i$-s such that
$\epsilon_i(X) > 0$, it has to be the case that $\epsilon_i(X) = 0$.
On the other hand, since $v' \leq w$, one can factor out a curl $N'$
from $X$ satisfying $v(N')=v'$.  This would imply $\epsilon_i(X)
> 0$, a contradiction. In the case of $\i \in \bW$, observe that there
are only finitely many cyclically increasing elements in $\aW$. Thus
one can choose a sufficiently large initial part $w$ of $\i$ such
that for every cyclically increasing $v$ we have $v < \i$ if and
only if $v < w$. This reduces the statement to the established case.

We prove (2).  Pick a representative $X \in E_{w}$. We have just
seen that the reduced word of the ASW factorization of $X$ is the
same as the reduced word of the (combinatorial) ASW factorization of
$w$. The claim now follows from Lemma \ref{lem:asww}.
\end{proof}

\begin{remark}
The factorization of $w \in \aW$ into maximal cyclically increasing
elements in Proposition \ref{prop:finasw} gives the dominant
monomial term of an affine Stanley symmetric function \cite{Lam}.
\end{remark}

\begin{proposition} \label{prop:aswcell}
Suppose $w = v_1 \ldots v_k$ is the result of combinatorial $ASW$
factorization of $w\in \aW$.  A pair $(w,v)$ is compatible if and
only if:
\begin{enumerate}
\item $v = v_k^{(1)}$;
\item $v_k \neq v_{k-1}^{(1)}$.
\end{enumerate}
In particular, $wv^{[\infty]}$ is reduced if (1) and (2) are
satisfied.
\end{proposition}

\begin{proof}
Suppose $(w,v)$ is compatible, arising from $X \in \Omega$ with ASW
factorization $X = \prod_{j \geq 1} N_j$.  As before, let $l$ be
minimal such that $\ell(v(N_l+k)) = \ell(v(N_l))$ for every $k \geq
0$.  We argue that $w = \prod_{j=1}^l v(N_j)$ is the combinatorial
ASW factorization of $w$.  Then both (1) and (2) follow from Lemma
\ref{lem:asww}.  Suppose for some $j$ that $v(N_j)$ is not the
maximal cyclically increasing element that can be factored from
$v(N_j) v(N_{j+1}) \cdots v(N_l)$.  By applying braid and
commutation moves to $N_j N_{j+1} \cdots N_l$ we see that $N_j$ is
not the maximal curl that can be factored out from $N_j N_{j+1}
\ldots N_l$, which contradicts the main property of ASW
factorization.  Thus by definition $w = v(N_1) v(N_{2}) \cdots
v(N_l)$ is the combinatorial ASW factorization of $w$.

Now suppose that $(w,v)$ satisfy the conditions (1) and (2) of the
Proposition.  Let $v' = v_k$, so that $v = (v')^{(1)}$.  Let $c(v')$
be the Coxeter element in which $s_i$ precedes $s_{i+1}$ if and only
if $s_i$ is contained in $v'$. Then there is a length additive
factorization $c(v') = v' u$. Furthermore, it is easy to see that
$c((v')^{(1)}) = uv'$.  Let $Z \in A_{c^\infty}$ (see Section
\ref{sec:infcox} and Proposition \ref{prop:eps}).  Now perform the
ASW factorization of $Z$ to get $Z = \prod_{i=1}^\infty N_i$.  Then
$v(N_1) = v'$, and the argument in the proof of Proposition
\ref{prop:ainj} shows that $\prod_{i=2}^\infty N_i \in
A_{c((v')^{(1)})^\infty}$.  Repeating, we deduce that $v(N_i) =
v^{(i)}$.  Thus $Z$ is in the $(v',v)$ ASW-cell.

Now let $Y = N'_1 N'_2 \cdots N'_{k-1}$, where $N'_i$ is any
degenerate curl satisfying $v(N'_i) = v_i$.  We claim that $X = YZ$
is in the $(w,v)$ ASW-cell.  Let $X_r = N'_r N'_{r+1} \cdots
N'_{k-1} Z$.  We shall show by decreasing induction that $N'_r =
N(\epsilon_1(Y_r), \ldots,\epsilon_n(Y_r))$.  We already know the
base case $Y_k = Z$.  The inductive step follows from Proposition
\ref{prop:finasw}(2) and Lemma \ref{lem:epc}.  It follows that
ASW-factorization applied to $X$ extracts the curls $N'_1, N'_2,
\ldots$, and that $X$ is in the $(w,v)$ ASW-cell.  In particular, we
deduce from Lemma \ref{lem:ASWreduced} that $wv^{[\infty]}$ is
reduced.
\end{proof}

\begin{example}
Let $n=4$, $w = s_1 s_2 s_3 s_0 s_2 s_1 s_3 s_2$, $v = s_1$. Then $v_1 = s_1 s_2 s_3$, $v_2 = s_0 s_2$, $v_3 = s_1 s_3$, $v_4 = s_2$ and the pair $(w,v)$ is compatible.
\end{example}

\begin{remark}\label{rem:whirl}
There is a whirl version of ASW factorization, with whirls replacing
curls, abd $\mu_i$'s replacing $\epsilon_i$'s.  As a result one
obtains a factorization of $X$ into maximal whirl factors.  All the
properties of ASW factorization have an analogous form  that holds
for this whirl-ASW factorization. For example, the analogs of
Theorem \ref{thm:aswomega}, Lemma \ref{lem:epsentire}, Proposition
\ref{prop:finasw} and Proposition \ref{prop:aswcell} hold, where in
the case of the latter two one needs to change the definition of
$v^{(k)}$ to $v^{(k)} = \prod_j s_{i_j+k}$, and cyclically
decreasing permutations occur instead of cyclically increasing ones.
\end{remark}

\section{Totally positive exchange lemma}\label{sec:TPlemma}
\subsection{Statement of Lemma, and proof of Theorem
\ref{thm:TPlimit}}

\begin{thm}[Totally positive exchange lemma]\label{thm:TPex}
Suppose
$$
X = e_r(a) e_{i_1}(a_1) \cdots e_{i_\ell}(a_\ell) = e_{i_1}(a'_1)
\cdots e_{i_\ell}(a'_\ell) e_j(a')
$$
are reduced products of Chevalley generators such that all
parameters are positive.  For each $m \leq \ell$ and each $x \in
\Z/n\Z$ define $S = \{s \leq m \mid i_s = x\}$.    Then
\begin{equation}\label{E:ex}
\sum_{s \in S} a'_{i_s} \leq \begin{cases} \sum_{s \in S} a_{i_s} & \mbox{if $x \neq r$,} \\
a + \sum_{s \in S} a_{i_s} & \mbox{if $x= r$.}
\end{cases}
\end{equation}
\end{thm}

Using the totally positive exchange lemma, we now prove Theorem
\ref{thm:TPlimit}.

\begin{proof}[Proof of Theorem \ref{thm:TPlimit}]
Define $\a' = R_\i^\j(\a)$.  Let $X = e_\i(\a)$, and $X^{(k)} =
e_{i_1}(a_1)\cdots e_{i_k}(a_k)$.  Let $Y = e_\j(\a')$ and define
$Y^{(k)}$ similarly.  By the definition of braid limit, it follows
that for each $k > 0$, there is $k'$ so that $Y^{(k)} < X^{(k')} <
X$ entry-wise.  We need to show that $X^{(k)} < Y$ for each $k$.

Let $Z^{(k)}$ be such that $Y^{(k)}Z^{(k)} = X$.  It is clear from
the the definition of braid limit that $Z^{(k)}$ is TNN (in fact,
$Z^{(k)} \in \Omega$).  Let $Z = Y^{-1}X$.  Then $Z = \lim_{k \to
\infty}Z^{(k)}$ and by \cite[Lemma 2.4]{LP}, $Z$ is TNN.  We shall show
that the entries of $Z$ directly above the diagonal (that is
$z_{i,i+1}$) vanish, which in turn implies that $Z$ is the identity,
or equivalently, $X = Y$.

Fix $i \in \Z/n\Z$, and write $\chi(X) = x_{i,i+1}$ for any $X \in
U$.  We now show that for each $k > 0$, we have $\chi(X^{(k)}) <
\chi(Y)$.  Since $\chi(X) =\chi(Z) + \chi(Y)$, this will prove that
$X = Y$.  Note that $\chi(e_{i_1}(b_1) \cdots e_{i_r}(b_r)) =
\sum_{s: i_s = i} b_s$. If $X$ is a (possibly infinite) product of
Chevalley generators, we let $\chi_r(X)$ be $\chi$ of the product of
the first $r$ generators in $X$.

By Propositions \ref{prop:TPlimit} and \ref{prop:limitexchange}, we
may assume that the braid limit $\i \to \j$ is obtained by infinite
exchange.  Suppose that the generators $i_1,i_2,\ldots,i_k$ are
all``crossed out'' by the $r$-th step in infinite exchange.  We let
$s$ be the rightmost generator of $\i$ to be crossed out in the
first $r$-steps.  Define $k_j$ for $j = 0,1,\ldots, s$ as follows:
set $k_0 = k$ and let $k_j = k_{j-1} + 1$ if the generator crossed
out in the $j$-th step of infinite exchange is to the right of
$i_k$, and $k_j = k_{j-1}$ otherwise. Then $k_r = r$.

Let $A, B, C, \ldots$ be the matrices obtained from $X$ by
performing one, two, three, and so on, iterations of infinite
exchange.  Let $V$ be the matrix obtained after $r$ iterations of
infinite exchange.  Note that the first $r$ factors of $V$ are the
same as the first $r$ factors of $Y$.

Using Theorem \ref{thm:TPex},
\begin{align*}
\chi(X^{(k)}) = \chi_k(X) \leq \chi_{k_1}(A) \leq \chi_{k_2}(B) \leq
\cdots \leq \chi_{k_r}(V) = \chi_r(V) = \chi_{r}(Y) < \chi(Y).
\end{align*}
\end{proof}

We shall give two proofs of Theorem \ref{thm:TPex}.  The first proof
relies on the machinery of the Berenstein-Zelevinsky Chamber Ansatz
\cite{BZ}, and is a direct calculation of the two sides of
\eqref{E:ex}.  The second proof is less direct, but significantly
shorter.

\subsection{First Proof of Theorem \ref{thm:TPex}}
We may assume $n > 2$ for otherwise the statement is vacuous.  Let
us call a subset $S \subset \Z$ {\it $a$-nice} if it occurs as $S =
w(\Z_{\leq a})$ for some $w \in \aW$.  We say that $S \subset \Z$ is
{\it nice} if it is 0-nice.  Clearly a subset can be $a$-nice for at
most one $a$.  If $a = a' + bn$ we will often identify an $a$-nice
subset $I$ with the $a'$-nice subset $I - bn$.

\begin{lem}
Let $I \subset \Z$ and $J = \Z \setminus I$.  Then $I \subset \Z$ is
nice if and only if $I - n \subset I$, $J + n \subset J$ and $|I
\cap \Z_{> 0}| = |J \cap \Z_{\leq 0}|$ is finite.
\end{lem}

A quadruple $D = (I, i, j, J)$ is {\it nice} if it can be obtained
from $w \in \aW$, by setting
$$
D(w) = (I = w(\Z_{\leq -1}),w(0), w(1), J = w(\Z_{> 1})).
$$
In particular, $I$ is $-1$-nice, $I \cup \{i\}$ is $0$-nice, and $I
\cup \{i,j\}$ is $1$-nice.

We will write $Ki$ to denote the set $K \cup \{i\}$, and say that
$Ki$ is nice, if $i \not \in K$, and both $K \cup \{i\}$ and $K$ are
nice. Note that this implies that $i$ is maximal in its residue
class modulo $n$, within $K \cup \{i\}$.  Similarly, we shall use
notation $Kij$, $Kijk$, and so on.  In this latter notation, we will
always assume that $i < j < k$ have distinct residues modulo $n$.

\subsection{Berenstein-Zelevinsky Chamber Ansatz}
We recall some definitions and results from \cite{BZ}.  The results
in \cite{BZ} are stated for finite-dimensional algebraic groups, but
as remarked there, can be extended to the Kac-Moody case.  In
particular, they apply to $U_{\geq 0}^\pol$.  Our notations differ
from theirs by $w \leftrightarrow w^{-1}$.

A {\it chamber weight} is an extremal weight of a fundamental
representation.  Every chamber weight is of the form $w \cdot
\omega_a$ where $w \in \aW$, and $\omega_a$ is a fundamental weight
of $\widehat{sl}(n)$.  Chamber weights in the orbit of $\omega_a$
are in bijection with $a$-nice subsets, via $w \cdot \omega_a
\leftrightarrow w(\Z_{\leq a})$.  Hereon, we identify chamber
weights with nice subsets.

Let $v \in \aW$.  Recall that we denote by $\Inv(v)$ the set of
inversions of $v$.  Let $I$ be a nice-subset. An inversion of $I$ is
a positive root $\alpha_{i}+\alpha_{i+1}+ \cdots + \alpha_{j-1}$
such that $j \in I$ but $i \notin I$.  We denote by $\Inv(I)$ the
set of inversions of $I$.  For convenience, in the following we will
identify positive roots with pairs $i < j$ (with different residues
modulo $n$).

We define, following \cite{BZ}, the set $E^v$ of $v$-chamber weights
by
$$
E^v = \{w(\Z_{\leq a}) \mid w \leq v\}.
$$
If $I \in E^v$, we say $I$ is $v$-nice.

\begin{prop}[{\cite[Proposition 2.8]{BZ}}] \label{P:chamber}
We have $I \in E^v$ if and only if $\Inv(I) \subset \Inv(v)$.
\end{prop}

\begin{example}
Let $n=3$ and let $w = s_1 s_2 s_1 s_0$. The step-by-step computation of $w(\Z_{\leq 0})$ proceeds as follows: $s_0(\Z_{\leq 0}) = \{1\} \cup \Z_{\leq -1}$, $s_1 s_0 (\Z_{\leq 0}) = \{2\} \cup \Z_{\leq -1}$, $s_2 s_1 s_0 (\Z_{\leq 0}) = \{0,3\} \cup \Z_{\leq -2}$, $s_1 s_2 s_1 s_0 (\Z_{\leq 0}) = \{-1, 0,3\} \cup \Z_{\leq -3}$. The inversions of $I = \{-1, 0,3\} \cup \Z_{\leq -3}$ are $\Inv(I) = \{(-2,-1), (-2,0), (-2,3), (1,3), (2,3)\}$. This set is identified with the set of roots $\Inv(I) = \{\alpha_1, \alpha_2, \alpha_1 + \alpha_2, \alpha_1 + \alpha_2 + \delta\}$, where pairs $(-2, 0)$ and $(1,3)$ correspond to the same root $\alpha_1 + \alpha_2$. This set is contained in the set of inversions of $w$ (in fact, equal to it). Therefore for any $w \leq v$ we have $\Inv(I) \subset \Inv(v)$.
\end{example}

We suppose that $v \in \aW$ has been fixed.  Let
$$
M_\bullet = (M_{I})_{I \in E^v}
$$
be a collection of positive real numbers satisfying the relations
\cite[(4.5)]{BZ}
\begin{equation}\label{eq:BZor}
M_{ws_a \Z_{\leq a}} M_{w s_{a+1} \Z_{\leq a+1}} = M_{w \Z_{\leq a}}
M_{w s_a s_{a+1} \Z_{\leq a+1}} + M_{w s_{a+1} s_a \Z_{\leq a}} M_{w
\Z_{\leq a+1}}
\end{equation}
for each $w \in \aW$ such that $w s_a s_{a+1} s_a$ is
length-additive, and such that all mentioned chamber weights lie in
$E^v$.

Suppose $\i = i_1 i_2 \cdots i_\ell$ is a reduced word for $v$, and
$X = e_{i_1}(a_1) \cdots e_{i_\ell}(a_\ell) \in U_{\geq 0}^\pol$,
where the $a_i$ are positive parameters.  For each reduced word $\j
= j_1 j_2 \cdots j_\ell$ of $v$, we define parameters $\a^\j =
(a_1^\j,a_2^\j, \ldots, a_\ell^\j) = R_\i^\j(\a)$ (see Corollary
\ref{lem:Rwelldefined}).

Given a nice quadruple $D = (I,i,j,J)$, we now define the positive
numbers
$$
M(D) = \frac{M_{I \cup \{i,j\}} M_I}{M_{I \cup \{i\}} M_{I \cup
\{j\}}}.
$$
If $D = (I,i,j,J)$, then $J$ is determined by $I,i, j$, so we shall
often write $M(I,i,j)$ instead of $M(I,i,j,J)$ for $M(D)$. We say
that $D$, or $M(D)$, is $v$-nice if all the indexing subsets in this
formula are $v$-nice. By abuse of notation, we write $M(w) :=
M(D(w))$.  Note that if both $w \leq v$ and $ws_0 \leq v$, then
$M(w)$ is $v$-nice.

\begin{thm}[\cite{BZ}]\label{thm:BZ}
There is a bijection between the collections $\{a_k^\j\}$ as $\j$
varies over all reduced words of $v$, and collections of positive
real numbers $M_\bullet = (M_{I})_{I \in E^v}$ satisfying
\eqref{eq:BZor}, given by
$$
a_k^\j = M(D(s_{j_1}s_{j_2} \cdots s_{j_{k-1}})).
$$
\end{thm}

\subsection{Relations for $M_I$ and $M(D)$}

For an arbitrary subset $I \subset \Z$, we say that $M_I$ is
$a$-nice, if $I$ is $a$-nice. We say that $M_I$ is nice if it is
$a$-nice for some $a$.  We first reinterpret Proposition
\ref{P:chamber} in a more explicit manner (see also \cite{BFZ}).

\begin{lem}\label{L:BZ}
Suppose $i < j < k$ have different residues and $K$ is nice.  Then
if one of the three pairs $(Kik,Kj)$, $(Kij,Kk)$, and $(Kjk,Ki)$
consist of nice subsets, then all three do.  If two of the three
pairs consist of $v$-nice subsets, then so is the third one.  We
then have
\begin{equation} \label{E:BZ}
M_{Kik} M_{Kj} = M_{Kij} M_{Kk} + M_{Kjk} M_{Ki}
\end{equation}
\end{lem}
\begin{proof}
The first statement is straightforward, and indeed implies that
$Kijk$ is nice.  To check the second statement, we note that there
are three possible inversions amongst $i < j < k$, and that the
$v$-niceness of each of the three pairs $(Kik,Kj)$, $(Kij,Kk)$, and
$(Kjk,Ki)$ imply that $\Inv(v)$ contains two of these inversions. It
is easy to check that if two of the three pairs are $v$-nice, then
$\Inv(v)$ contains all three inversions.

To obtain \eqref{E:BZ}, apply \eqref{eq:BZor} to some $w \in \aW$
satisfying $w(a) = i$, $w(a+1) = j$, $w(a+2) = k$, and $w(\Z_{<a}) =
K$.
\end{proof}

\begin{lem}\label{L:three}
Suppose $i < j < k$ have different residues, such that $Kijk$ is
nice.  Then assuming $v$-niceness, we have
\begin{align*}
M(Kj,i,k) &= M(Ki,j,k) + M(Kk,i,j) \\
M(Ki,k,j) &= M(Kj,k,i) + M(Kk,i,j).
\end{align*}
Furthermore, in either equation, if two terms are known to be
$v$-nice, then the third term is as well.
\end{lem}
\begin{proof} Using \eqref{E:BZ}, we calculate
\begin{align*}
M(Kj,i,k) - M(Ki,j,k) &= \frac{M_{Kijk}M_{Kj}}{M_{Kij}M_{Kjk}}-
\frac{M_{Kijk}M_{Ki}}{M_{Kij}M_{Kik}}
 \\
& = \frac{M_{Kijk}}{M_{Kij}M_{Kik}M_{Kjk}}\left(M_{Kj}M_{Kik}
- M_{Ki}M_{Kjk}\right)\\
&= M(Kk,i,j).
\end{align*}
The second statement is similar.  The last statement follows from
Lemma \ref{L:BZ}.
\end{proof}

\begin{lem}\label{L:four}
Suppose $i < j < k < l$ have different residues, such that $Kijkl$
is nice.  Then assuming $v$-niceness,
$$
M(Kl,j,k) - M(Ki,j,k) = M(Kk,i,l) - M(Kj, i, l)
$$
and
$$
M(Kk,i,j) - M(Kl,i,j) = M(Ki,k,l) - M(Kj,k,l).
$$
Furthermore, if the terms on the same side of either equation are
$v$-nice, then all six inversions amongst $\{i,j,k,l\}$ are
contained in $\Inv(v)$.
\end{lem}
\begin{proof}
The last statement is checked directly, and implies that all the
subsets in the following calculations are $v$-nice.

We first prove the first equation, omitting $K$ from the notation
for simplicity. In the following we use \eqref{E:BZ} repeatedly.
\begin{align*}
& \frac{M_{ijl} M_j}{M_{ij} M_{jl}} -
\frac{M_{ijk}M_i}{M_{ij}M_{ik}} \\
&= \frac{M_{ijl}(M_{ij}M_k + M_{jk}M_i) - M_{ijk} M_{jl}
M_i}{M_{ij}M_{jl}M_{ik}}
\\
&=\frac{M_{ijl}{M_k}}{M_{jl}M_{ik}} +
\frac{M_i(M_{ijl}M_{kj}-M_{ijk}M_{jl})}{M_{ij}M_{jl}M_{ik}} \\
&= \frac{M_{ijl}{M_k}+ M_{jkl}M_i}{M_{jl}M_{ik}}\\
&= \frac{M_{ijl}M_k}{M_{jl}M_{ik}} + \frac{M_{jkl}(M_{il}M_k - M_l M_{ik})}{M_{ik}M_{kl}M_{jl}}\\
&= \frac{M_k(M_{ijl}M_{kl}+M_{jkl}M_{il}) - M_{jkl}M_l M_{ik}}{M_{ik}M_{kl}M_{jl}}\\
&=\frac{M_{ikl}M_k}{M_{ik}M_{kl}} -
\frac{M_{jkl}M_{l}}{M_{jl}M_{kl}}
\end{align*}
For the second equation, we calculate
\begin{align*}
&\frac{M_{ijk}M_k}{M_{ik}M_{jk}} - \frac{M_{ikl}M_i}{M_{ik}M_{il}} \\
&=\frac{M_{ijk}M_{ik}M_l + M_{ijk}M_{kl}M_i -
M_{ikl}M_{jk}M_i}{M_{ik}M_{jk}M_{il}} \\
&=\frac{M_{ijk}M_l-M_{jkl}M_i}{M_{jk}M_{il}}\\
&= \frac{M_lM_{ijk}M_{jl}+M_lM_{jkl}M_{ij} - M_{jkl}M_jM_{il}}{M_{il}M_{jk}M_{jl}}\\
&= \frac{M_{ijl}M_l}{M_{il}M_{jl}} - \frac{M_{jkl}M_j}{M_{jk}M_{jl}}
\end{align*}
\end{proof}

\subsection{Explicit formula for difference of sum of parameters}
\label{sec:explicit}

Let $v = s_{i_1}s_{i_2}\cdots s_{i_{\ell(v)}}$, and $w =
s_{i_1}s_{i_2}\cdots s_{i_\ell}$.  Let $$X = e_{\bar r}(a)
e_{i_1}(a_1) \cdots e_{i_{\ell(v)}}(a_{\ell(v)}) = e_{i_1}(a'_1)
\cdots e_{i_{\ell(v)}}(a'_{\ell(v)}) e_j(a')$$ as in Theorem
\ref{thm:TPex}, and without loss of generality we assume $j = 0$. We
write $\bar r \in \Z/n\Z$ instead of $r$ as we shall use the latter
for a specific representative of $\bar r$.  Let $M_I$ for $I \in
E^v$ be defined via Theorem \ref{thm:BZ}, using the parameters
$a_k^\j$ for $X$.

For any $u \leq v$, define $N(u)$ as follows. Pick a reduced
factorization $u = s_{j_1} \cdots s_{j_k}$ and write $X =
e_{j_1}(b_1)\cdots e_{j_{k}}(b_k) Y$ where $Y$ is in $E^{u^{-1}v}$.
Then $N(u) = \sum_{s \mid j_s = 0} b_{i_s}$.  Thus to prove the
theorem we must show that for every $w$ such that $\ell(s_rw) >
\ell(w)$ and $w, s_rw \leq v$, we have $N(s_rw) \geq N(w)$.

To prove the theorem, we may further assume that $i_\ell = 0$, and
we let $w' = w s_0 < w$.  Unless otherwise specified, $r \in \Z$ is
the maximal representative of $\bar r$ such that $r \in w(\Z_{\leq
1})$.

\begin{lem}\label{L:vw}
Suppose $v, w \in \aW$ and $r \in \Z$ is such that $\ell(s_rw) >
\ell(w)$. If $w \leq v$ and $s_rw \leq v$ then $n > w^{-1}(r+1) -
w^{-1}(r) > 0$.
\end{lem}
\begin{proof}
The inequality $w^{-1}(r+1) - w^{-1}(r) > 0$ follows from
$\ell(s_rw)
> \ell(w)$.  Suppose $w^{-1}(r+1) - w^{-1}(r) > n$.  Then $(r+1 < r+n)$ is
an inversion in $w$. But then $(r+1 < r+n)$ will also be an
inversion in $v$.  This is impossible as $(r < r+1)$ is an inversion
in $s_rw$, which means it is also an inversion in $v$.
\end{proof}

Note that $w \leq v$ and $s_r w \leq v$ implies that, $w$-nice and
$s_r w$-nice subsets are also $v$-nice.

\begin{lem}\label{L:ind}
Suppose that we are in the situation of Lemma \ref{L:vw}.  Let
$D(w') = (I,i,j,J)$.
\begin{enumerate}
\item If $r, r+1$ both lie in $I$, then $N(s_r w) - N(w)
= 0$.
\item If ($i = r$ and $r+1 \in J$) or ($j = r+1$ and $r \in I$), then $N(s_rw) - N(w) = 0$.
\item Otherwise $N(s_r w) - N(w) = M(D'(w'))$ where $D'(w') = (I',i',j',J')$ is
obtained from $D(w') = (I,i,j,J)$ by setting $i'= r$, $j' = r+1$ and
\begin{align*}
I' &= \begin{cases} I \setminus \{r\} \cup \{j\} & \mbox{if $r \in
I$ but $r+1 \notin
I$} \\
I &\mbox{if $\{r,r+1\} \cap I = \emptyset$.}
\end{cases}
\end{align*}
\end{enumerate}
\end{lem}

\begin{example} \label{ex:tpe}
Let $n=4$, $v = s_0 s_1 s_2 s_1 s_3 s_0 s_1 s_3$, $j=0$. Then $v s_0 = s_1 v$ so that $\bar r = \bar 1$. Let $w = s_0 s_1 s_2 s_1 s_3 s_0$, and thus $w' = s_0 s_1 s_2 s_1 s_3$. We have an equality $$e_1(a) e_0(a_1) e_1(a_2) e_2(a_3) a_1(a_4) e_3(a_5) e_0(a_6) e_1(a_7) e_3(a_8) =$$ $$= e_0(a_1') e_1(a_2') e_2(a_3') a_1(a_4') e_3(a_5') e_0(a_6') e_1(a_7') e_3(a_8') e_0(a_9')$$ for some positive parameters, and we are interested in the value of $N(s_1 w) - N(w) = a_1 + a_6 - a_1' - a_6'$.  We compute that $w(\Z_{\leq 1})= \{1,3\} \cup \Z_{\leq -1}$, from which we find $r = 1$. We further compute $D(w')= (I, -4, 3, J)$, where $I = \{-3,-2,-1,1\} \cup \Z_{\leq -5}$, $J = \{0,2\} \cup \Z_{\geq 4}$. Then we are in case (3), and furthermore in situation $r \in I$ but $r+1 \not \in I$. This allows us to find $D'(w') = (I', 1, 2, J')$ where $I' = \{-3,-2,-1, 3\} \cup \Z_{\leq -5}$, $J' = \{-4, 0\} \cup \Z_{\geq 4}$. Then $M(D'(w'))$ is the needed manifestly positive value of $N(s_1 w) - N(w)$.
\end{example}

It may not be clear that $D'(w')$ is a $v$-nice set, but this will
follow from our calculations.

\begin{proof}[Proof of Theorem \ref{thm:TPex}]
According to Lemma \ref{L:ind}, the difference $N(s_r w) - N(w)$
that we are interested in is manifestly nonnegative.
\end{proof}

To prove the lemma we proceed by induction on the number of times
$0$ occurs in $\i$.  By assumption $0$ occurs at least once.  In the
following calculations, all the nice subsets that occur will in fact
be $v$-nice, and this will follow from the last statements of
Lemmata \ref{L:three} and \ref{L:four}; we will not mention this
explicitly.

\subsection{Base Case}
Suppose $0$ occurs once in $\i$.  Then by Theorem \ref{thm:BZ},
$N(w) = M(w')$.  Let $(I,i,j,J) = D(w')$.  We note that our
assumption implies that $I \cup \{i\} = \Z_{\leq 0}$.

{\bf Case $r < 0$:} we have $N(s_rw) = M(s_rw')$.  Suppose first
that $r, r + 1 \in I$.  Then $M(w') = M(s_r w')$, so that $N(s_rw) -
N(w) = 0$, agreeing with Lemma \ref{L:ind}(1).  Otherwise, we must
have $i = r+1$. Let $K = I - \{r\}$.
 Then using Lemma \ref{L:three},
\begin{align*}
N(s_rw) - N(w) = M(K(r+1),r,j) - M(Kr,r+1,j) = M(Kj, r, r+1)
\end{align*}
which is $M(D'(w'))$, as required.

{\bf Case $r = 0$:} first note that we cannot have $i = 0$ and $j =
1$, for this would mean that $s_0 w' s_0$ is not length-additive.
Suppose first that $0 \in I$ and $1 \in J$.  Then by Lemma
\ref{L:vw}, $0,1,i,j$ all have distinct residues modulo $n$.  Let $K
= I \setminus \{0\}$ and $L = J \setminus \{1\}$. Then using Lemma
\ref{L:four}, we have $N(s_rw) - N(w) = M(Ki,0,1) + M(K1,i,j) -
M(K0,i,j) = M(Kj,0,1)$ as required. Suppose that $i = 0$ and $1 \in
J$.  Then using \eqref{E:BZ}, we have $N(s_rw) - N(w) = M(I,0,1) +
M(I,1,j) - M(I,0,j) = 0$, agreeing with Lemma \ref{L:ind}(2).  The
last case $0 \in I$ and $j = 1$ is similar.

 {\bf Case $r > 0$:} this is similar to $r < 0$.

\subsection{Inductive Step}
Now suppose that the letter 0 occurs more than once in $\i$.  Let $u = s_{i_1}
s_{i_2} \cdots s_{i_{\ell'}}$, where $i_{\ell'} =0$ and $i_{\ell' +
1}, i_{\ell'+2}, \ldots, i_{\ell-1}$ are all distinct from $0$.  We
shall assume that Lemma \ref{L:ind} is known to hold for $u$.  Let
us first compare $D(u') = (A,a,b,B)$ with $D(w') = (I,i,j,J)$. Since
$s_{i_{\ell'}} = s_0$, we have $I \cup \{i\} = A \cup \{b\}$ and $J
\cup \{j\} = B \cup \{a\}$.  Furthermore, one notes that we cannot
have both $i = b$ and $j = a$.  Also one cannot have both $i = r$
and $j = r+1$.

We make two preparatory remarks:
\begin{enumerate}
\item
We shall use Lemma \ref{L:four} repeatedly in the following, where
$\{i,j,k,l\}$ of the Lemma will usually be $\{i,j,r,r+1\}$.  The
assumption that $Kijkl$ is nice will follow from the fact that the
positions of $i, j , r, r+1$ in $w$ are within a ``window'' of size
$n$.
\item
In the beginning we chose $r$ to be the maximal representative of
$\bar r$ in $w(\Z_{\leq 1})$.  This choice of $r$ is also the
maximal representative of $\bar r$ for $u$, except in one case: when
$a = r+n$, $b \neq r+n+1$, $r+n \in J$, and $\{r,r+1\} \subset I$.
\end{enumerate}

By Theorem \ref{thm:BZ}, we have
\begin{equation}\label{E:Nbz}
N(s_rw) - N(w) = N(s_ru) - N(u) + M(s_rw') - M(w').
\end{equation}

{\bf Case 1:} $\{r, r+1\} \subset I$.  We have $M(w') = M(s_rw')$.
If $\{r, r+1 \} \subset A$ as well, then by induction and
\eqref{E:Nbz} we have $N(s_rw) - N(w) = N(s_ru) - N(u) + M(s_rw') -
M(w') = 0$ by Lemma \ref{L:ind}(1).  Otherwise $r \in A$ and $r+1 =
b$ (this includes the case $a = r+n$). Then $N(s_rw) - N(w) = 0$ as
well by Lemma \ref{L:ind}(2). In either case, we have verified that
$N(s_rw)- N(w)$ agrees with Lemma \ref{L:ind}(1).

\smallskip

{\bf Case 2:} $r \in I$ and $i = r+1$.  We have two possibilities
for $(A,a,b,B)$: (a) $\{r,r+1\} \subset A$, (b) $r \in A$ and $b =
r+1$.  In either case, the inductive hypothesis says that $N(s_ru')
- N(u') = 0$. We have $M(w') = M(Kr,r+1,j)$ and $M(s_r w') =
M(K(r+1),r,j)$ where $K = I \setminus \{r\}$.  By Lemma
\ref{L:three} and \eqref{E:Nbz}, $N(s_rw) - N(w) = M(K(r+1),r,j)-
M(Kr,r+1,j) = M(Kj, r, r+1) = M(D'(w'))$, as required.

\smallskip

{\bf Case 3:} $r+1 \in J$ and $j = r$.  Same as Case 2.

\smallskip

{\bf Case 4:} $r \in I$ and $j = r+1$.  By length-additivity of $w =
w' s_0$, we have $i < r+1$.  Let $K = I \setminus\{r\}$.
 We have three possibilities for $(A,a,b,B)$: (a) $r \in A$ and $r+1 \in
B$, (b) $r \in A$ and $a = r+1$, (c) $r + 1 \in B$ and $b = r$. In
all three cases, one has $N(s_ru)-N(u) = M(Ki,r,r+1)$.  One
calculates using Lemma \ref{L:three} and \eqref{E:Nbz} that
$$
N(s_rw) - N(w) = M(Ki,r,r+1) + M(K(r+1),i,r)- M(Kr,i,r+1) = 0
$$
agreeing with Lemma \ref{L:four}(2).

{\bf Case 5:} $r+1 \in J$ and $i = r$.  Same as Case 4.

{\bf Case 6:} $r \in I$ and $r+1 \in J$.  Let $K = I
\setminus\{r\}$.

We have three possibilities for $(A,a,b,B)$: (a) $r \in A$ and $r+1
\in B$, (b) $r+1 \in B$ and $b = r $, and (c) $r \in A$ and $a =
r+1$. In all three cases, we have $N(s_ru)-N(u) = M(Ki,r,r+1)$ and
calculate using \eqref{E:Nbz}
\begin{align*}
N(s_rw) - N(w) &= M(Ki,r,r+1) + M(K(r+1),i,j) - M(Kr,i,j) \\ &=
M(Kj,r,r+1)
\end{align*}
using the two forms of Lemma \ref{L:four}, depending on whether $i <
j < r < r+1$, $i < r < r+1 < j$, or $r < r +1 < i < j$.  This agrees
with Lemma \ref{L:ind}(3).

This completes the proof of Lemma \ref{L:ind}.

\subsection{Second Proof of Theorem \ref{thm:TPex}}
We use the notation for $w$ and $v$, and $N(u)$ of Section
\ref{sec:explicit}. Without loss of generality we can assume $s_x =
s_0$, as before.

\begin{lemma} \label{lem:tpdesc}
It suffices to prove Theorem \ref{thm:TPex} in the case $w = s_{i_1}
\dotsc s_{i_\ell}$ has a single right descent $s_0$.
\end{lemma}

\begin{proof}
Write $w = uy$, where $\ell(u) + \ell(y) = \ell(w)$ and $y \in W$.
Then $N(w) = N(u)$ and $N(s_rw) = N(s_ru)$, so we may replace $w$ by
$u$.
%
\end{proof}

\begin{lemma} \label{lem:tpjoin}
Suppose $w$ has a unique right descent $s_0$, and that $s_r w > w$
and the join $v' = w \vee s_r w$ exists in weak order.  Then
\begin{enumerate}
\item
if $w^{-1}(r+1)=w^{-1}(r)+1 = l+1$, then $l \neq 0$ modulo $n$ and
$v' = w s_{l}$;
\item
if $w^{-1}(r+1) =k$, $w^{-1}(r) = l$ and $k > l+1$, then $[l,k]$
contains a unique number $m$ of residue $0$ modulo $n$ and  $v' = w
s_{l} s_{l+1} \dotsc s_{m-1} s_{k-1} s_{k-2} \dotsc s_{m+1} s_m$.
\end{enumerate}
\end{lemma}

\begin{proof}
By Lemma \ref{L:vw}, we have $w^{-1}(r) < w^{-1}(r+1) < w^{-1}(r) +
n$.

If $w(r+1)=w(r)+1 = l+1$ then $w s_{l} = s_r w$ and thus $s_r w
> w$. This implies that $s_r w = w \vee s_r w$.

Assume now $w^{-1}(r+1) =k$, $w^{-1}(r) = l$, and $k > l+1$. Then
the sequence $r = w(l), w(l+1), \ldots, w(k) = r+1$ cannot be
increasing, and thus has at least one descent. Since the only right
descent of $w$ is $s_0$, the first claim follows. Furthermore, it
has to be the case that $$r = w(l) < w(l+1) < \cdots < w(m)
> w(m+1) < w(m+2) < \cdots < w(k) = r+1.$$ Then we
see that $v' = w s_{ l} s_{{l+1}} \dotsc s_{m-1} s_{{k-1}} s_{{k-2}}
\dotsc s_{m+1} s_m$ is reduced since at each step an inversion is
created, resulting in
$$v'([l,k]) = w(l+1), \cdots, w(m), r+1, r,
w(m+1), w(m+2),  \cdots, w(k-1).$$ One also has $$w s_{l} s_{{l+1}}
\dotsc s_{m-1} s_{{k-1}} s_{{k-2}} \dotsc s_{m+1} s_m = s_r w s_{l}
s_{{l+1}} \dotsc s_{m-1} s_{{k-1}} s_{{k-2}} \dotsc s_1$$ and so $v'
> w, s_r w$ in weak order. It remains to argue that $v'$ is the
minimal upper bound.  The inversion set $\Inv(w \vee s_rw)$ contains
the inversion of $(r< r+1)$ in $s_r w$, and the inversions $\{(r+1 <
w(l+1)), \ldots, (r+1 < w(m))\}$ together with $\{(w(m+1)<r),
\ldots, (w(k-1)<r)\}$ in $w$.  By biconvexity (see Section
\ref{sec:biconvex}), $\Inv(w \vee s_rw)$ must also contain $\{(r
<w(l+1), \ldots, (r < w(m))\}$ and $\{(w(m+1)<r+1), \ldots,
(w(k-1)<r+1)\}$.  These extra inversions are present in $\Inv(v)$
and the number of extra inversions is exactly $\ell(v) -\ell(s_rw)$.
Thus $v' = w \vee s_r w$.
\end{proof}

\begin{proof}[Proof of Theorem \ref{thm:TPex}]
By Lemma \ref{lem:tpdesc} we can assume $w$ has a single right
descent $s_0$. Since $v > w, s_r w$ we know that $w$ and $s_r w$
have a join $v$ in weak order.  Furthermore, the join $v'$ is given
by Lemma \ref{lem:tpjoin}.  It remains to note that $N(s_r w) =
N(v')$ since $v' = s_r w y$ with $\ell(s_rw) + \ell(y) = \ell(v')$
and $y \in W$.  On the other hand, $N(w) \leq N(v')$ as well. Thus
$N(s_r w) \geq N(w)$, as desired.
\end{proof}

\begin{example}
In the situation of Example \ref{ex:tpe} the join of $w = s_0 s_1 s_2 s_1 s_3 s_0$ and $s_1 w$ is exactly $s_0 s_1 s_2 s_1 s_3 s_0 s_1 s_3 s_0 = s_1 s_0 s_1 s_2 s_1 s_3 s_0 s_1 s_3$, which shows that $N(s_1 w) - N(w) = a_1 + a_6 - a_1' - a_6' = a_9'$ is manifestly positive.
\end{example}

\section{Greedy factorizations} \label{sec:greed}
Suppose $X \in U_{\geq 0}$.  A factorization $X = e_{i}(a) X'$ with
$a \geq 0$ and $X' \in U_{\geq 0}$ is called {\it greedy} if
$e_{i}(-a') X$ is not TNN for $a'
> a$. Since limits of TNN matrices are TNN \cite{LP}, we can
equivalently say that $X = e_i(a) X'$ is greedy if
$$
a = \sup( a' \geq 0 \mid e_i(-a') X \in U_{\geq 0})
$$
where the right hand side is always equal to $\max( a' \geq 0 \mid
e_i(-a') X \in U_{\geq 0})$.

More generally, a factorization $X = e_{i_1}(a_1) e_{i_2}(a_2)
\cdots e_{i_r}(a_r)X'$ is called greedy if the factorization
$e_{i_k}(a_k) \left(e_{i_{k+1}}(a_{k+1}) \cdots
e_{i_r}(a_r)X'\right)$ is greedy for every $k \in [1,r]$. A
factorization $X = e_{i_1}(a_1) e_{i_2}(a_2) \cdots $ is called
greedy if the factorization $e_{i_k}(a_k) \left(e_{i_{k+1}}(a_{k+1})
\cdots\right)$ is greedy for every $k \geq 1$.

As was shown in Proposition \ref{prop:notinjective}, the maps $e_\i$
are not injective in general.  Restricting to greedy factorizations
fixes this problem to some extent: for each $X$ and infinite reduced
word, there is at most one greedy factorization $X =e_\i(\a)$.

\begin{prop}\label{prop:greedyexist}
Let $X \in \Omega$.  Then $X$ has a complete greedy factorization.
\end{prop}
\begin{proof}
By Theorem \ref{thmcor:uniqueness}, we may factor (infinitely many)
Chevalley generators from $X$ greedily in any manner, and the
resulting product will be equal to $X$.
\end{proof}

Thus greedy factorizations do ``cover'' $\Omega$.

\subsection{Minor ratios for greedy parameters}
If $I = \{i_1 < i_2 < \cdots <i_l\}$ and $J = \{j_1 < j_2< \cdots
<j_k\}$ are two sets of integers of the finite cardinality, we say
that $I$ is less than or equal to $J$, written $I \leq J$, if $i_r
\leq j_r$ for each $r \in [1, \min(k,l)]$.  We say that $I$ is much
smaller than $J$ and write $I \ll J$ if $i_r < j_r$ for each $r \in
[1, \min(k,l)]$

One can use limits of minor ratios to factor an element of $\Omega$
greedily.  Let $I = i_1 < i_2 < \ldots < i_l$ and $I' = i'_1 < i'_2
< \ldots < i'_l$ be two sets of row indices such that $I \leq I'$.
Let $h = \min(i_1, i'_1)$ and let $I_k = I \cup \{h-k,\ldots,h-1\}$
and $I'_k = I' \cup \{h-k,\ldots,h-1\}$. In particular, one has $I_0
= I$ and $I'_0=I'$.


The following Lemma will be proved in Section \ref{sec:imm}.

\begin{lemma} \label{lem:mlr}
Let $X \in U_{> 0}$ be totally positive, and $I \leq I'$ be fixed.
Let $J_k$ be a sequence of column sets such that $|J_k| = k+l$ and
$J_{k-1} \ll J_{k}$. The limit
$$\ell = \lim_{k \to \infty}
\frac{\Delta_{I_k,J_k}(X)}{\Delta_{I'_k,J_k}(X)}$$ exists and does
not depend on the choice of sequence $J_k$.  Furthermore, $\ell \leq
\frac{\Delta_{I_k,J_k}(X)}{\Delta_{I'_k,J_k}(X)}$ for any $J_k$.
\end{lemma}

Lemma \ref{lem:mlr} allows us to introduce the notation
$$\frac{X_{[\ldots I]}}{X_{[\ldots I']}} = \lim_{k \to \infty}
\frac{\Delta_{I_k,J_k}(X)}{\Delta_{I'_k,J_k}(X)}$$ that does not
include $\{J_k\}$ in it. The following theorem is the key motivation
for looking at this kind of minor ratio limits.

\begin{proposition} \label{prop:lmr}
Suppose $X \in U_{>0}$ is totally positive.  Let $a =
\frac{X_{[\ldots i]}}{X_{[\ldots i+1]}}$ and set $X' = e_i(-a)X$.
Then $X = e_i(a) X'$ is a greedy factorization.
\end{proposition}

\begin{proof}
Assume $a' > a$. Then there exists $k$ such that
$\frac{\Delta_{I_k,J_k}(X)}{\Delta_{I'_k,J_k}(X)} < a'$.  If we
denote $Y = e_i(-a') X$ then $\Delta_{I_k,J_k}(Y) =
\Delta_{I_k,J_k}(X) - a' \Delta_{I'_k,J_k}(X) < 0$, and thus $Y$
cannot be totally nonnegative.

On the other hand, we argue that $X' = e_i(-a) X \in U_{\geq 0}$. By
\cite[Lemma 2.3]{LP}, it suffices to check nonnegativity of only the
row-solid minors of $X'$.  Furthermore it suffices to look at minors
with bottom row $i$, since other row-solid minors do not change when
$X$ is multiplied by $e_i(-a)$.  But we have $\Delta_{I,J}(X') =
\Delta_{I,J}(X) - a\Delta_{I',J}(X)$, where $I$ is a solid minor
ending in row $i$ and $I' = (I \setminus \{i\}) \cup \{i+1\}$.  By
the definition of $a$ and the last statement of Lemma \ref{lem:mlr},
we conclude that any such minor in $X'$ is nonnegative.
\end{proof}

One can use Proposition \ref{prop:lmr} to compute the coefficients
in the greedy factorization for any finite sequence of Chevalley
generators.  We illustrate it by the following lemma.

\begin{lemma}\label{lem:triple}
Let $X \in U_{> 0}$ be totally positive.
\begin{enumerate}
\item
If $X = e_i(a_1)e_{i+1}(a_2)e_i(a_3)X'''$ is a greedy factorization
then $$a_1 = \frac{X_{[\ldots i-1,i]}}{X_{[\ldots i-1,i+1]}} \ \ \ \
\ a_2 = \frac{X_{[\ldots i-1, i+1]}}{X_{[\ldots i-1, i+2]}} \ \ \ \
\ a_3 = \frac{X_{[\ldots i, i+1]}}{X_{[\ldots i+1,i+2]}} /
\frac{X_{[\ldots i-1,i+1]}}{X_{[\ldots i-1, i+2]}};$$

\item
if $X = e_{i+1}(a_1)e_{i}(a_2)e_{i+1}(a_3)X'''$ is a greedy
factorization then $$a_1 = \frac{X_{[\ldots i,i+1]}}{X_{[\ldots
i,i+2]}}\ \ \ \ \ a_2 = \frac{X_{[\ldots i-1,i, i+2]}}{X_{[\ldots
i-1,i+1, i+2]}} \ \ \ \ \ a_3 = \frac{X_{[\ldots i-1,i]}}{X_{[\ldots
i-1,i+2]}} / \frac{X_{[\ldots i-1,i,i+2]}}{X_{[\ldots i-1,i+1,
i+2]}}.$$
\end{enumerate}
\end{lemma}

\begin{proof}
In the following, we shall write $X^k_{[\ldots I]}$ to mean
$\Delta_{I_k,J_k}(X)$, where we assume that some sequence $J_k$ has
been fixed, satisfying $J_{k-1} < J_k$  and $|J_k| = k+3$ for each
$k$. When we write $\frac{X^k_{[I]}}{X^k_{[I']}}$ with $|I| = |I'|<
3$ we assume that the initial part of $J_k$ of size $k+|I|$ is used
as the column sequence.

We prove the formulae for $a_1$ and $a_2$ in the first case first.
We already know that if $X = e_i(a_1)X'$ is a greedy factorization
then $a_1 = \frac{X_{[\ldots i]}}{X_{[\ldots i+1]}}$.  Assume $X' =
e_{i+1}(a_2)X''$ is greedy.  Then we have \begin{align*}a_2 &=
\frac{X'_{[\ldots i+1]}}{X'_{[\ldots i+2]}} = \lim_{k \to \infty}
\frac{X^k_{[\ldots i-1, i, i+1]}}{X^k_{[\ldots i-1, i, i+2]} - a_1
X^k_{[\ldots
i-1, i+1, i+2]}} \\
& = \lim_{k \to \infty} \frac{X^k_{[\ldots i-1, i,
i+1]}}{X^k_{[\ldots i-1, i, i+2]} -  \frac{X^k_{[\ldots
i-1,i]}}{X^k_{[\ldots i-1,i+1]}} X^k_{[\ldots i-1, i+1, i+2]}}
\\
 &=\lim_{k \to \infty} \frac{X^k_{[\ldots i-1, i, i+1]} X^k_{[\ldots
i-1, i+1]}}{X^k_{[\ldots i-1, i, i+2]}X^k_{[\ldots i-1, i+1]} -
{X^k_{[\ldots
i-1, i]}} X^k_{[\ldots i-1, i+1, i+2]}} \\
&= \lim_{k\to \infty} \frac{X^k_{[\ldots i-1, i, i+1]} X^k_{[\ldots
i-1, i+1]}}{{X^k_{[\ldots i-1, i+2]}} X^k_{[\ldots i-1, i, i+1]}} =
\frac{X_{[\ldots i-1, i+1]}}{X_{[\ldots i-1, i+2]}}.
\end{align*}
The three-term Pl\"ucker relation (Lemma \ref{lem:plucker}) is used
here.

The proof of the formulae for $a_1$ and $a_2$ in the second case is
similar. Assume now again that $X = e_i(a_1)X'$ is greedy. When we
factor $e_{i+1}(a_2)e_{i}(a_3)$ from $X'$ greedily we get
\begin{align*}a_3 &= \frac{X'_{[\ldots i-1,i, i+2]}}{X'_{[\ldots i-1,i+1, i+2]}} =
\lim_{k \to \infty}\frac{X^k_{[\ldots i-1,i, i+2]} -
\frac{X^k_{[\ldots i-1,i]}}{X^k_{[\ldots i-1,i+1]}} X^k_{[\ldots
i-1,i+1, i+2]}}{X^k_{[\ldots i-1,i+1,
i+2]}} \\
&= \lim_{k \to \infty}\frac{X^k_{[\ldots i-1,i, i+2]} X^k_{[\ldots
i-1,i+1]} - {X^k_{[\ldots i-1,i]}} X^k_{[\ldots i-1,i+1,
i+2]}}{X^k_{[\ldots i-1,i+1, i+2]} X^k_{[\ldots i-1,i+1]}} =
\frac{X_{[\ldots i-1,i, i+1]} X_{[\ldots i-1,i+2]} }{X_{[\ldots
i-1,i+1, i+2]} X_{[\ldots i-1,i+1]}}.
\end{align*} The proof of the formula for $a_3$ in the second case
is similar.
\end{proof}



\subsection{Complete greedy factorizations}
\begin{lem}\label{lem:greedyreduced}
Let $X \in \Omega$.  If $X = e_{\i}(\a)$ is greedy, then $\i$ is
necessarily an infinite reduced word.
\end{lem}
\begin{proof}
Assume $\i$ is not reduced. Take the first initial part $ws_i$ of
$\i$ which is not reduced. By the strong exchange condition
\cite[Theorem 5.8]{Hum} one can find $s_j$ inside $w$ so that $w s_i
= u s_j v s_i = uv$. Then $s_j v = v s_i$ and using the
corresponding braid moves in the factorization of $X$ one can
rewrite $X = \ldots e_j(a) e_j(a') \ldots$ where $a'>0$. This means
that the original factor $e_j(a)$ was not greedy -- a contradiction
implying the lemma.
\end{proof}

\begin{thm} \label{thm:grp}
Let $\i \to \j$ be a braid limit of infinite reduced words, and $\a
 \in \ell^1_{>0}$. Then $X = e_\i(\a)$ is greedy if and only if $e_\j(R_\i^\j(\a))$
is greedy.
\end{thm}
\begin{proof}
By Lemma \ref{lem:entire}, $X$ is totally positive.  Greediness is a
local property, and thus it suffices to check that it is preserved
under braid and commutation relations. In case of commuting $s_i$
and $s_j$ it is clear from Proposition \ref{prop:lmr} that factoring
out $e_i$ does not effect the parameter of greedy factorization of
$e_j$ and vice versa.  For braid moves, by Lemma \ref{lem:triple},
it suffices to check that
\begin{align*} &e_i\left(\frac{X_{[\ldots i]}}{X_{[\ldots
i+1]}}\right)e_{i+1}\left(\frac{X_{[\ldots i-1, i+1]}}{X_{[\ldots
i-1, i+2]}}\right)e_i\left(\frac{X_{[\ldots i, i+1]}}{X_{[\ldots
i+1,i+2]}} / \frac{X_{[\ldots i-1,i+1]}}{X_{[\ldots i-1,
i+2]}}\right)\\
&= e_{i+1}\left(\frac{X_{[\ldots i+1]}}{X_{[\ldots i+2]}}\right)
e_i\left(\frac{X_{[\ldots i, i+2]}}{X_{[\ldots i+1, i+2]}}\right)
e_{i+1}\left(\frac{X_{[\ldots i]}}{X_{[\ldots i+2]}} /
\frac{X_{[\ldots i,i+2]}}{X_{[\ldots i+1, i+2]}}\right).
\end{align*} This is straightforward, using the three-term Pl\"ucker
relations (Lemma \ref{lem:plucker}) and \eqref{E:chevrel2}.
\end{proof}

\subsection{Proof of Lemma \ref{lem:mlr}}\label{sec:imm}
Since $X \in U_{>0}$, all the minor ratios in the limit are
well-defined.

Roughly speaking, as we let $k \to \infty$ the set $J_k$ grows and
moves to the right. We argue that each of the two processes -
increasing in size without moving and moving to the right without
change in size - does not increase the ratio
$\frac{\Delta_{I,J}(X)}{\Delta_{I',J}(X)}$. In fact, it was already
shown in \cite[Lemma 10.5]{LP} that if $J \leq J'$ have the same
cardinality, then
$$\frac{\Delta_{I,J}(X)}{\Delta_{I',J}(X)} \geq
\frac{\Delta_{I,J'}(X)}{\Delta_{I',J'}(X)}.$$

To establish Lemma \ref{lem:mlr}, it thus remains to consider the
case of $J$ increasing in size without moving.

\begin{lemma} \label{lem:tlgr}
Suppose that $J' = J \cup J''$ for some set of columns $J''$ each
element of which is bigger than the elements of $J$.  Then
$$\frac{\Delta_{I,J}(X)}{\Delta_{I',J}(X)} \geq
\frac{\Delta_{I_k,J'}(X)}{\Delta_{I'_k,J'}(X)}$$ where $k = |J''|$.
\end{lemma}

The proof is similar to the one of \cite[Lemma 10.5]{LP} and uses
Rhoades and Skandera's Temperley-Lieb immanants (or TL-immanants).
These are functions $\Imm_{\tau}^{\mathrm{TL}}(Y)$ of a $n \times n$
matrix $Y$, where $\tau$ is a Temperley-Lieb diagram. We use the
same notation as in \cite[Section 10.2]{LP}, referring the reader
there, or to \cite{RS} for the definitions.

\begin{theorem}
\label{th:immdecomp} \cite[Proposition~2.3, Proposition~4.4]{RS}\ If
$Y$ is a totally nonnegative matrix, then
$\Imm_{\tau}^{\mathrm{TL}}(Y) \geq 0$.  For two subsets $I,J\subset
[n]$ of the same cardinality and $S=J\cup (\bar I)^\wedge$, we have
$$
\Delta_{I,J}(Y) \cdot \Delta_{\bar I, \bar J}(Y) = \sum_{\tau \in
\Theta(S)} \Imm_{\tau}^{\mathrm{TL}}(Y).
$$
\end{theorem}

We also need the following property of Temperley-Lieb immanants.

\begin{lemma} \label{lem:tlz}
If columns $i$ and $i+1$ of a matrix $X$ are equal and the vertices
$i$ and $i+1$ in the TL-diagram $\tau$ are not matched with each
other, then $\Imm^{\mathrm{TL}}_{\tau}(X) = 0$.
\end{lemma}

\begin{proof}
Follows from \cite[Corollary 15]{RS2} since vertices $i$ and $i+1$
on the column side of a TL-diagram are matched if and only if $s_i$
is a right descent of the corresponding $321$-avoiding permutation
$w$. Alternatively, the claim follows immediately from the network
interpretation of Temperley-Lieb immanants given in \cite{RS}.
\end{proof}

\begin{proof}[Proof of Lemma \ref{lem:tlgr}]
Clearly it is enough to prove the lemma for $|J''|=1$, since the
argument can be iterated. Let $Y$ be the submatrix of $X$ induced by
the rows in $I \cup I'_1 = I_1 \cup I'$ and columns $J \cup J'$,
where we repeat a row or a column if it belongs to both of the sets
(that is, $I \cup I'_1$ and $J \cup J'$ are considered as
multisets).  We index rows and columns of $Y$ again by $I \cup I'_1$
and $J \cup J'$. Whenever there is a repeated column we consider the
right one of the two to be in $J'$.
 Similarly whenever there is a repeated row we consider the bottom
one of the two to be in $I'_1$. Then $I'_1 = \bar I$, $J' = \bar J$
and we can apply Theorem \ref{th:immdecomp} to the products
$\Delta_{I,J}(X)\Delta_{I'_1,J'}(X)$ and $\Delta_{I',J}(X)
\Delta_{I_1,J'}(X)$.

By Lemma \ref{lem:tlz}, all Temperley-Lieb immanants of $Y$ in which
vertices of $J$ are not matched with their counterparts in $J'$ are
zero. Thus we can restrict our attention to the immanants whose
diagrams have this property. The rest of the points consist of $I
\cup I'_1$ and the single point from $J''$, which we think of as
lying on a line arranged from left to right.  The points in $I \cup
I'$ are colored so that in any initial subsequence (reading from the
left) the number of white points is at least as large as the number
of black points. The colorings corresponding to
$\Delta_{I,J}(X)\Delta_{I'_1,J'}(X)$ and $\Delta_{I',J}(X)
\Delta_{I_1,J'}(X)$ agree on all vertices but two: the rightmost
one, coming from $J''$, and the leftmost one, added when passing
from $I'$ to $I'_1$ (or from $I$ to $I_1$). The first product
corresponds to coloring leftmost vertex black and rightmost vertex
white, while the second product has colors the other way around. Now
it is easy to see that any non-crossing matching compatible with the
second product is also compatible with the first one, since in the
former the leftmost vertex has to be matched with the rightmost
vertex. The claim of the lemma now follows from Theorem
\ref{th:immdecomp}.
\end{proof}

\begin{proof}[Proof of Lemma \ref{lem:mlr}]
It follows from \cite[Lemma 10.5]{LP} and Lemma \ref{lem:tlgr} that
$$\frac{\Delta_{I_k,J_k}(X)}{\Delta_{I'_k,J_k}(X)} \geq
\frac{\Delta_{I_{k+1},J_{k+1}}(X)}{\Delta_{I'_{k+1},J_{k+1}}(X)}$$
and thus the limit exists and is not bigger than each individual
fraction $\frac{\Delta_{I_k,J_k}(X)}{\Delta_{I'_k,J_k}(X)}$. The
argument for the independence on choice of $\{J_k\}$ is similar to
that in \cite[Theorem 10.6]{LP}. Namely, for two different choices
of column sequence and for a particular term in one of them, one can
always find a smaller term in the other one. This implies the limits
cannot be different.
\end{proof}

\section{Open problems and conjectures}
\label{sec:problems}
\def\C{\mathbb C}
\noindent {\sc From Section \ref{sec:polyloop}.}

In \cite{GLS}, Geiss, Leclerc, and Schr\"{o}er studied (in the
Kac-Moody setting) the cluster algebra structure of the coordinate
ring $\C[U^w]$ of the unipotent cells $U^w$ obtained by intersecting
$U^\pol$ with the Bruhat cells. Their work should be compared with
our Proposition \ref{prop:mv}, which gives a complete list of
positive minors for totally nonnegative elements in each unipotent
cell. Presumably, the cluster variables in each cluster for
$\C[U^w]$ give rise to a minimal set of totally positive criteria
(cf. \cite{FZ}).

\medskip \noindent {\sc From Section \ref{sec:braidlim}.}

\begin{problem}
Characterize, in terms of infinite reduced words, the minimal
elements of the limit weak order for other affine types.
\end{problem}

\begin{question}
Are minimal elements of limit weak order necessarily fully
commutative?
\end{question}

In other affine types infinite Coxeter elements are still reduced
\cite{KP}, but may not necessarily be fully commutative.  Note also
that in affine type $B$ there are {\it {more}} minimal blocks than
Coxeter elements.

\medskip \noindent {\sc From Section \ref{sec:omega}.}

We conjecture that the inclusion relations of the $E_\i$ is exactly
the limit weak order.

\begin{conjecture}
The converse of the Corollary \ref{cor:Eprec} holds: if $E_{[\j]}
\subset E_{[\i]}$ then $[\i] \leq [\j]$. Furthermore, if $[\i] <
[\j]$ then the containment $E_{[\j]} \subset E_{[\i]}$ is strict.
\end{conjecture}

\begin{problem}\label{prob:intersect}
Describe completely the pairs $([\i],[\j])$ such that $E_{[\i]}$ and $E_{[\j]}$ have non-trivial intersection.  Describe these intersections completely.
\end{problem}

Problem \ref{prob:intersect} might be solved by an affirmative answer to the following question.
\begin{question}\label{q:braid}
Assume $e_{\i}(\a) = X = e_{\j}(\b)$.  Can the equality $e_\i(\a) = e_\j(\b)$ always be proved using braid
limits? What about braid limits where one is allowed to go through intermediate non-reduced products, that is one is allowed to ``split" the Chevalley generators as in Example \ref{ex:1012}?
\end{question}

An affirmative solution to Question \ref{q:braid} may involve a long sequence of braid limits: if $e_{\i}(\a) = X = e_{\j}(\b)$ there does not always exist a factorization $X = e_{\k}(\c)$ and braid limits $\k \to \i$ and $\k \to \j$, such that $\a = R_\k^\i(\c)$ and $\b =R_\k^\j(\c)$.  For example, by Proposition \ref{prop:TPlimit} this is the case if $\i = \j$ but $\a \neq \b$.  Also there does not always exist a factorization $X = e_{\k}(\c)$ and braid limits $\i \to \k$ and $\j \to \k$, such that $\c = R_\i^\k(\a)$ and $\b =R_\j^\k(\b)$.  For example consider the situation in Example \ref{ex:coxeternotdisjoint}.

\medskip

It would also be interesting to describe the topology of each $E_\i$, and of their intersections (cf. Theorem \ref{thm:Lus}).

\medskip \noindent {\sc From Section \ref{sec:inj}.}
Let $X \in \Omega$ and let $I(X)$ be the set of equivalence classes
$[\i]$ of infinite reduced words $\i$ such that $X \in E_{\i}$. It
is clear from Corollary \ref{cor:Eprec} and Theorem
\ref{thm:TPlimit} that $I(X)$ is a lower order ideal in limit weak
order.  Question \ref{q:braid} partly motivates (but is not implied by)
the following conjecture.

\begin{conjecture}[Principal ideal conjecture] \label{con:pic}
For any $X \in \Omega$, the ideal $I(X)$ is a principal ideal.
\end{conjecture}

One special case of Conjecture \ref{con:pic} is $X \in A_{c^\infty}$.  In this case, it may be reasonable to conjecture that $I(X) = \{[c^\infty]\}$; that is, these matrices only have a single factorization, which is an infinite Coxeter factorization (cf. Proposition \ref{prop:ainj}).  Similarly, it may be reasonable to conjecture that if $X \in B_{c^\infty}$ then $X$ has a factorization other than the $c^\infty$ factorization.  This is consistent with Example \ref{ex:1012}.

We briefly explain some consequences and variations of the Principal ideal conjecture.  The following conjecture is a combinatorial consequence of Conjecture \ref{con:pic}.

\begin{conjecture} \label{con:piweak}
Suppose $X \in E_{\i}$ and $X \in E_{\j}$.  Then the join $[\i] \vee
[\j]$ exists in the limit weak order.
\end{conjecture}

The condition that the join $[\i] \vee [\j]$ exists can be made more
precise.
\begin{proposition} \label{pr:cimp}
Let $\i$ and $\j$ be two infinite reduced words.  Then the join
$[\i] \vee [\j]$ exists in the limit weak order if and only if there
does not exist a (finite) root $\alpha \in \Delta_0^+$ such that
both $\alpha$ and $\delta - \alpha$ lie in $\Inv(\i) \cup \Inv(\j)$.
\end{proposition}
\begin{proof}
Given $I = \Inv(\i) \cup \Inv(\j)$, we define a partial order
$\preceq$ on $[n]$ as follows.  For each $1 \leq a < b \leq n$, set
$a \prec b$ whenever $\alpha_{ab} \in I$, and $ b \prec a$ whenever
$\delta -\alpha_{ab} \in I$.  Transitivity and the fact that this is
a partial order (rather than a preorder) follows from Proposition
\ref{prop:table}, and the assumption that $\alpha_{ab}$ and $\delta
- \alpha_{ab}$ are not simultaneously in $I$.  Now pick any total
order $\preceq'$ which extends the partial order $\preceq$, consider
$\preceq'$ as defining a maximal face of the braid arrangement, and
let $[\k]\in \bW$ be the corresponding element under the bijection
of Theorem \ref{thm:blocks}.  Note that $[\k]$ is a maximal element
of $\bW$, and that it is an upper bound for $[\i]$ and $[\j]$.
\end{proof}

Recall that in Conjecture \ref{conj:injective} we conjectured that if $\i$ is an infinite reduced word which is minimal in its block then $e_\i$ is injective.

\begin{proposition}
Conjecture \ref{con:piweak} implies Conjecture \ref{conj:injective}.
\end{proposition}

\begin{proof}
We know that an element $[\i]$ minimal in its block has a representative $\i = t_{\lambda}^{\infty}$ for some translation element $t_{\lambda}$. Choose a reduced expression $t_{\lambda} = s_{j_1} \dotsc s_{j_l}$. Assume $e_{\i}$ is not injective, that is, we have $X = e_{\i}(\a) = e_{\i}(\a')$ for $\a \neq \a'$. We may assume that $a_1 \neq a'_1$.   For otherwise, we may write $X = e_w(a_1,a_2,\ldots,a_r)e_{\i'}(a_{r+1},\ldots) = e_w(a_1,a_2,\ldots,a_r) e_{\i'}(a'_{r+1},\ldots)$.  The infinite reduced word $\i'$ is equal to $(t_{v^{-1} \cdot \lambda})^\infty$, where $w = vt_\mu$ for some $v \in W$ and $\mu \in Q^\vee$.  In particular, $[\i']$ is minimal in its block. 

Now, without loss of generality assume $a_1 < a'_1$. Then $e_{j_1}(-a_1) X$ belongs to both $E_{t_{\lambda}^{\infty}}$ and $E_{t_{s_{j_1} \cdot \lambda}^{\infty}}$.  Since $\alpha_{j_1}$ is an inversion of $t_{\lambda}$, we have $\ip{\alpha_{j_1},\lambda} < 0$ (see the proof of Proposition \ref{prop:bijbraid}).  But then $\ip{\alpha_{j_1},s_{j_1}\lambda} > 0$, so $\delta - \alpha_{j_1}$ is an inversion of $t_{s_{j_1} \cdot \lambda}$.  Assuming Conjecture \ref{con:piweak}, this contradicts Proposition \ref{pr:cimp}.
\end{proof}

A problem significantly harder than Conjecture \ref{conj:injective} is

\begin{problem}
For each $X \in \Omega$ and infinite reduced word $\i$, completely describe $e_\i^{-1}(X)$.
\end{problem}

\medskip \noindent {\sc From Section \ref{sec:asw}.}

\begin{question}
Assume $X \in \Omega$ lies in the ASW cell $A(w,v)$. For a fixed
choice of $i$, in which ASW cells may the matrix $e_i(a) X$ lie as
$a$ assumes all positive values? More generally, where may $Y X$ lie
if $Y \in E_u$ for a fixed $u \in \aW$?
\end{question}

Let $X \in \Omega$.  ASW factorization gives rise to a distinguished
factorization of $X$.  But whirl ASW factorization (Remark
\ref{rem:whirl}) gives rise to another distinguished factorization.
Can we get every factorization of $X$ using a mixture of these
operations?

\begin{question}
Let $X \in \Omega$. Apply to $X$ ASW or whirl ASW factorization
repeatedly, choosing freely which of the two to apply at each step.
Is it true that for every $[\i] \in I(X)$ one can find a sequence of
ASW or whirl ASW choices that shows $X \in E_{[\i]}$?
\end{question}

In Proposition \ref{prop:nonreduced}, we showed that every $X =
e_\i(\a)$ for $\i$ not necessarily reduced lies in $\Omega \cup
U_{\geq 0}^\pol$.  One can obtain a (possibly finite) reduced word
$\j$ from $\i$, as follows.  Recall that the Demazure product is
defined by
$$
w \circ s_i = \begin{cases} ws_i & \mbox{$ws_i > w$,} \\ w &
\mbox{otherwise.} \end{cases}
$$
The Demazure product is associative.  We define the reduction $\j =
j_1 j_2 \cdots$ of $\i = i_1 i_2 \cdots$ by requiring that $\j$ is
reduced and that the list $s_{j_1}, s_{j_1}s_{j_2},
s_{j_1}s_{j_2}s_{j_3}, \ldots$, coincides with $s_{i_1},
s_{i_1}\circ s_{i_2}, s_{i_1} \circ s_{i_2} \circ s_{i_3}, \ldots$,
after repetitions are removed.

\begin{question}
Assume that $\i$ is an infinite non-reduced word and $X \in E_{\i}$.
Let $\j$ be the reduction of $\i$.  Is it true that $X \in E_\j$?
\end{question}

Given a factorization $X = e_\i(\a)$, one can attempt to produce a
factorization $X = e_\j(\a')$ by ``adding'' each of the generators
$e_{i_r}(a_r)$ one at a time.  However, when the product is not
reduced, many previously calculated parameters may change when the
additional factor $e_{i_r}(a_r)$ is introduced.  A priori, we have
no guarantee that in the limit some of the parameters do not go to
$0$.

\begin{example}
The simplest example is a non-reduced product that starts $$X = e_1(a) e_2(b) e_1(c) e_2(d) \dotsc.$$ Then when we multiply by the fourth factor, it gets absorbed into the previous three factors as follows: $e_1(a) e_2(b) e_1(c) e_2(d) = e_1(a + \frac{cd}{b+d}) e_2(b+d) e_1(\frac{bc}{b+d})$. The third parameter has decreased from $c$ to $\frac{bc}{b+d}$.
\end{example}

\medskip \noindent {\sc From Section \ref{sec:TPlemma}.}

Infinite products of Chevalley generators also make sense for
general Kac-Moody groups.  We intend to study them in the future
\cite{LPKM}.

\begin{conjecture}
The TP Exchange Lemma (Theorem \ref{thm:TPex}) holds in Kac-Moody
generality.
\end{conjecture}

\medskip \noindent {\sc From Section \ref{sec:greed}.}

For a finite reduced word $\i$, Berenstein and Zelevinsky \cite{BZ} gave an expression for the parameters $\a$ in the matrix $X = e_\i(\a)$, in terms of the minors of (the twist matrix of) $X$.  Because of the lack of injectivity of $e_\i$ in the case that $\i$ is infinite, this problem cannot be easily posed in our setting.  However, it does make sense if we restrict to greedy factorizations.

\begin{problem}
Let $\i$ be an infinite reduced word.  Assume that $X \in \Omega$ has a greedy factorization $e_\i(\a)$. Find an explicit formula
for the parameters $a_j$ in the spirit of Lemma \ref{lem:triple}.
Find an explicit formula that is manifestly positive.
\end{problem}

\begin{example}
Let $n = 3$ and suppose that a greedy factorization of $X$ starts with $e_1(a) e_2(b) e_0(c) \dotsc$. Then it can be computed that $$c = {\frac{X_{[\ldots -1, 0, 3]}}{X_{[\ldots -1, 0, 5]}}} / \left({\frac{X_{[\ldots -1, 0, 4]}}{X_{[\ldots, -1, 0, 5]}} - \frac{X_{[\ldots -1,0,1]}}{X_{[\ldots -1,0,2]}}} \right).$$ While explicit, this expression is not manifestly positive. One can use the Temperley-Lieb immanants of Section \ref{sec:imm} to prove the positivity of the denominator.  However, it seems desirable to have an expression which is manifestly positive in terms of the minors of $X$.
\end{example}

Let $J(X)$ be the set of equivalence classes $[\i]$ of infinite
reduced words $\i$ such that $X$ has a greedy factorization of the
form $e_{\i}(\a)$.

\begin{question}
Is it true that $I(X) = J(X)$ for any $X \in \Omega$? Equivalently,
if $X$ has a factorization of the form $e_{\i}(\a)$, does it
necessarily have a greedy factorization of the same form?
\end{question}

By Theorem \ref{thm:grp} and Theorem \ref{thm:TPlimit}, $J(X)$ is an
ideal in limit weak order.

\begin{conjecture}
The ideal $J(X)$ is principal.
\end{conjecture}

One can also state an analog of the weaker Conjecture
\ref{con:piweak} for greedy factorizations.

\end{document}